\documentclass[reqno,12pt]{amsart}%
\usepackage[body={17cm,21cm}]{geometry}

\usepackage[T1]{fontenc}
\usepackage{graphicx}
\usepackage{mathrsfs}
\usepackage{amscd,latexsym,amsthm,amsfonts,amssymb,amsmath,amsxtra}
\usepackage[colorlinks, urlcolor=blue,  citecolor=blue]{hyperref}
\usepackage{enumerate}
\usepackage[all]{xy}%
\providecommand{\U}[1]{\protect\rule{.1in}{.1in}}
\RequirePackage{amsmath}
\RequirePackage{amssymb}

\newcommand{\BG}{{\mathbb {G}}}

\newcommand{\BP}{{\mathbb {P}}}

\newcommand{\Br}{{\mathrm{Br}}}

\newcommand{\Gal}{{\mathrm{Gal}}}

\newcommand{\Ker}{{\mathrm{Ker}}}

\newcommand{\ord}{{\mathrm{ord}}}

\newcommand{\Pic}{\mathrm{Pic}}

\renewcommand{\mod}{\ \mathrm{mod}\ }

\newcommand{\et}{{\operatorname{\acute{e}t}}}

\numberwithin{equation}{section}

\theoremstyle{remark}

\newtheorem{defi}{\rm{\textbf{Definition}}}[section]

\newtheorem{rem}[defi]{\rm{\textbf{Remark}}}

\theoremstyle{plain}

\newtheorem{thm}[defi]{\rm{\textbf{Theorem}}}

\newtheorem{cor}[defi]{\rm{\textbf{\textbf{Corollary}}}}

\newtheorem{lem}[defi]{\rm{\textbf{Lemma}}}

\newtheorem{prop}[defi]{\rm{\textbf{\textbf{Proposition}}}}

\begin{document}

  \title[Brauer-Manin obstruction for Markoff surfaces] {Brauer-Manin obstruction for Markoff surfaces}

\author{J.-L. Colliot-Th\'el\`ene}
\address{Universit\'e Paris Sud\\Math\'ematiques, B\^atiment 307\\91405 Orsay Cedex\\France}
\email{jlct@math.u-psud.fr}

\author{Dasheng Wei}
\address{Academy of Mathematics and System Science \\ Chinese Academy of Sciences \\ Beijing 100190,  China;
  School of Mathematical Sciences, University of Chinese Academy of Sciences \\ Beijing 100049, China
}
\email{dshwei@amss.ac.cn}

\author{Fei Xu}
\address{School of Mathematical Sciences \\ Capital Normal University \\ Beijing 100048, China}
\email{xufei@math.ac.cn}

\date{submitted 30 September 2018; revised, June 11th, 2019}

\begin{abstract}
Ghosh and Sarnak have studied integral points on  surfaces defined by an equation
$x^2+y^2+z^2-xyz= m$ over the integers.
For these affine surfaces, we systematically study the Brauer group
 and  the Brauer-Manin obstruction to    the
 integral Hasse principle.
We prove  that strong approximation for integral points on any such surface,
away from any finite set of places, fails, and that, for $m\neq 0, 4$,  the Brauer group does
not control  strong approximation.
 \end{abstract}


\maketitle

\section{Introduction}

Fix $m\in \Bbb Z$. Let $d:=m-4$.
 Let $ \mathcal{U}_{m} \subset {\Bbb A}^3_{\Bbb Z}$
be the affine scheme over $\Bbb Z$ defined by the equation
\begin{equation} \label{markoff1} x^2+y^2+z^2-xyz= m .\end{equation}
It is equivalently defined  by the equation
\begin{equation} \label{markoff2} (2z-xy)^2- 4d = (x^2-4)(y^2-4),  \end{equation}
  by the equation
\begin{equation} \label{markoff3}  (x-y-z+2)^2-d=(x+2)(y-2)(z-2),  \end{equation}
as well as similar ones obtained by permutation of coordinates.

The surface  $U_{m}=\mathcal{U}_{m}\times_{\Bbb Z}\Bbb Q$
over $\Bbb Q$ is called a Markoff surface. 
Unless otherwise mentioned, we assume $m\neq 0$ and $d\neq 0$. These are the conditions for
  $U_{m}$ to be smooth.

In \cite{GS}, A. Ghosh and P. Sarnak have studied the set $\mathcal{U}_{m}(\Bbb Z)$ of integral solutions
 of
such equations.
A key tool is the action of the automorphism group $\Gamma$
generated by the following three types of elements

(a) the Vieta involution: $(x, y, z)\mapsto (yz-x, y, z) $.

(b) the sign change: $(x, y, z) \mapsto (-x,-y, z)$.

(c) the permutations of $x, y, z$.

We denote
$\mathcal{U}_{m}(A_{\Bbb Z})
= \prod_{p}
\mathcal{U}_{m}(\Bbb Z_{p}),$ where $p$ runs through all primes
and $\infty$, and $\Bbb Z_{\infty}={\Bbb R}$.
Let
$${\mathcal{U}_{m}(A_{\Bbb Z}  )}_{\bullet} =
 \prod_{p < \infty}
  \mathcal{U}_{m}(\Bbb Z_{p})
 \times \pi_{0}(U_{m}(\Bbb R))$$
where $ \pi_{0}(U_{m}(\Bbb R))$ is the set of connected components of $U_{m}(\Bbb R)$.
Let  $$\mathcal{U}_{m}(A_{\Bbb Z})_{\bullet}^{\Br} \subset {\mathcal{U}_{m}(A_{\Bbb Z}  )}_{\bullet} $$ be the subset consisting of elements
which are orthogonal to $\Br(U_{m})$ for the Brauer-Manin pairing 
$$\mathcal{U}_{m}(A_{\Bbb Z})_{\bullet} \times \Br(U_{m}) \to \Bbb Q/\Bbb Z$$
(see \cite[\S 1]{CTXu09}).
This is called the (reduced) Brauer-Manin set of $\mathcal{U}_{m}$.
\medskip

Here are some of the main results from \cite{GS}.

(0) $\mathcal{U}_{m}(A_{\Bbb Z}) =\emptyset$ if and only if
$m \equiv 3$  mod $ 4$ or $m \equiv \pm 3$ mod $9$.
Other values of $m$ are called ``admissible''.

(1) For $m$ admissible and ``generic''  (\cite[p. 3]{GS}, see Proposition \ref{induction} below),
 following Markoff, Hurwitz and Mordell,  
Ghosh and Sarnak develop a reduction theory : there
exists a bounded fundamental domain in $\Bbb R^3$ for integral solutions.
In particular  the set $\mathcal{U}_{m}(\Bbb Z)/\Gamma$ is finite.

(2) Suppose   
that $m$
 is not a square.
 Then $\mathcal{U}_{m}(\Bbb Z)$ is Zariski dense in $\mathcal{U}_{m}$ if and only if $\mathcal{U}_{m}(\Bbb Z)$ is not empty
 \cite[(1.5)]{GS}. Zariski density still holds if  $m$ is a square and contains an odd prime factor congruent to 1
 modulo 4 \cite[final comment in \S 5.2.1]{GS}.

(3)  Strong approximation need not hold, i.e.
$\mathcal{U}_{m}(\Bbb Z)$ need not be dense in
 $\mathcal{U}_{m}(A_{\Bbb Z})_{\bullet}$ (see \cite[p. 21]{GS}).
  This uses the quadratic reciprocity law.

(4) There are infinitely many $m$'s such that $\mathcal{U}_{m}$ does not satisfy the  integral Hasse principle.
 The examples in \cite{GS} are all of the shape $d=r.v^2$, with $r = \pm 2$, $r=12$, $r=20$,
and  specific properties for the primes dividing $v$. The arguments use quadratic reciprocity.
They are in the same spirit as earlier examples  \cite{CTXu09,CTW} accounted for by the integral Brauer-Manin obstruction.
From a historical point of view, it is interesting to note that
 examples very close to those of \cite{GS} are already given in  Mordell's 1953 paper \cite[\S 3]{Mo}.

 (5) For ``generic'' values of $m$,
 reduction theory  leads to examples where  $\mathcal{U}_{m}(A_{\Bbb Z}) \neq \emptyset$  but
   $\mathcal{U}_{m}(\Bbb Z) = \emptyset$. On the basis of intensive numerical experiments, Ghosh and Sarnak
  suggest that there are many such examples that cannot be explained by
   a reciprocity argument, i.e. for which, in our language,    $\mathcal{U}_{m}(A_{\Bbb Z})_{\bullet}^{\Br} \neq \emptyset$. More precisely they predict
   a count for the set of $m$'s with local solutions and no global solution which is much higher than what their families of
   counterexamples produce.
     
 \bigskip

 The cubic surface $X_{m} \subset {\Bbb P}^3_{\Bbb Q}$
given by the homogeneous equation $t(x^2+y^2+z^2)-xyz= mt^3$ is smooth as soon as $m \neq 0, 4$.
 The surface $U_{m}={\mathcal U}_{m}\otimes_{\Bbb Z}\Bbb Q$ is the complement in $X_{m}$ of the hyperplane section $H$ defined by
 plane section $t=0$. 
Its geometric fundamental group is trivial (Prop. \ref{pi1}).
Thus $U_{m}$, or rather the pair $(X_{m},H)$,  is in a strong sense a  log K3 surface \cite[Definition 2.4]{H}.

 The search for integral points on $\mathcal{U}_{m}$ bears some analogy with the search for
 rational points on smooth, projective $K3$-surfaces $W$. For this latter situation, Skorobogatov has put forward
 the conjecture : The closure of the set $W(\Bbb Q)$ in the adelic set $W(A_{\Bbb Q})_{\bullet}$ is just the Brauer-Manin
 set $W(A_{\Bbb Q})_{\bullet}^{\Br}$.
One may wonder whether there is a similar result for integral points on log K3 surfaces $U$.
 Here some restriction must be made. It may indeed happen that the set $\mathcal{U}(\Bbb Z)$ is not empty but
not Zariski dense in $U$ (Harpaz  \cite[Theorem 1.4]{H}; Jahnel and Schindler \cite[Theorem 2.6]{JS}).

Here are some questions raised by the paper of Ghosh and Sarnak.

(A) A first problem is to check that all counterexamples in \cite{GS}
 are of Brauer-Manin type, and to search for as many families of counterexamples as possible.

This problem is best handled by solving problems (B) and (C) :

(B) For arbitrary $m$, can one determine $\Br(U_{m})/\Br(\Bbb Q)$ ? Is this quotient finite ?

(C)  For arbitrary $m$, can one determine $\mathcal{U}_{m}(A_{\Bbb Z})_{\bullet}^{\Br}$ ?

(D) When (how often) is the closure of $\mathcal{U}_{m}(\Bbb Z) $ equal to the Brauer-Manin set
  $\mathcal{U}_{m}(A_{\Bbb Z})_{\bullet}^{\Br}$ ?

 \medskip

Here are the main results of our paper.

(a)   We solve Problem (A), i.e. we check that the  counterexamples to the integral Hasse principle
based on the quadratic reciprocity law
 in \cite{GS} are of Brauer-Manin type,
 and  we  produce more families of counterexamples of the same kind.

 (b) We solve Problem (B) for all values of $m$.  This in principle solves Problem (C).

(c) Over an arbitrary ground field, we give generators for the algebraic part of the Brauer group of $U$,
and  we systematically study the ``transcendental part'' of the Brauer group of $U$.

 (d) We get a satisfactory answer to Problem (D).
More precisely, we prove (see Theorem \ref{wsa}):
\begin{thm} Let $m\in \mathbb Z$ be any integer. Suppose $\mathcal U_m (A_\mathbb Z)\neq \emptyset$. For any finite set $S$
of primes the image of the natural map $\mathcal U_m(\mathbb Z) \to \prod_{p \notin S} \mathcal U_m (\mathbb Z_p)$  is not dense.
\end{thm}

The proof of this theorem does not involve the Brauer group, it only uses reduction theory. It should be compared  with the statement
at the bottom of page 2 of \cite{GS}, with reference to \cite{BGS},  that if   $d=m-4 >0$ is a square, then $\mathcal U_m$ ``satisfies a form of strong approximation''. See Remark \ref{compareBGS} below.

 As a corollary, one gets (see Corollary \ref{notsa})
\begin{cor}
 Suppose $m\neq 0, 4$ and $\mathcal{U}_{m}(A_{\Bbb Z})_{\bullet}^{\Br}\neq \emptyset$. Then  $\mathcal U_m(\mathbb Z)$ is not dense in $\mathcal{U}_{m}(A_{\Bbb Z})_{\bullet}^{\Br}$.
 \end{cor}

Since there are infinitely many $m\neq 0,4$ such that $\mathcal U_m(\mathbb Z)$ is Zariski dense in $\mathcal U_m$ by \cite[\S 5.2]{GS}, we obtain infinitely many log K3 surfaces where integral points are Zariski dense but are not dense in the integral Brauer-Manin sets (see Corollary \ref{notzariski}).

Such a  behaviour had not been yet observed, even in the context of rational points.
  If one allows discussion of density in the real locus, one may  only compare this with the examples of smooth projective surfaces $X/\Bbb Q$ with the property that
the closure of $X(\Bbb Q)$ in $X(\Bbb R)$ does not coincide with a union of connected components of
the real locus $X(\Bbb R)$ \cite[\S 5]{CTSkSD}.

  \bigskip

This  work  was  started in Beijing in November 2017 and  posted on arXiv in August 2018. 
In a preprint posted on arXiv in July 2018,
 D. Loughran and V. Mitankin  \cite{LM} have made an independent study.
 With the restrictions $m, d, md$ not squares,  they   independently solve problem (B).
Their  paper also solves Problem (A),  produces some more types of counterexamples,
and  gives an asymptotic lower bound  for the number of integers $m$   
 giving rise to such counterexamples.
   Our stock of counterexamples enables us
  to produce a slightly better asymptotic lower bound than \cite[Theorem 1.5]{LM}.

 With the same restriction that $m, d, md$ are not squares, towards Problem (C), Loughran and Mitankin
 establish the beautiful result that the only possible examples with
 $\mathcal{U}_{m}(A_{\Bbb Z}) \neq \emptyset$ and
 $\mathcal{U}_{m}(A_{\Bbb Z})^{\Br}  = \emptyset$ satisfy that the class of $d=m-4$
 in $\Bbb Q^{*}/\Bbb Q^{*2}$ lies in the subgroup spanned by $\pm1, 2, 3, 5$.
 This finiteness result,
 which is in the spirit of the finiteness of exceptional spinor classes
 in the  study of the representation of an integer by a ternary quadratic form (see \cite[Remark 7.11]{CTXu09}),
 explains why the examples in \cite{GS} based on the quadratic reciprocity law were of a rather special type.
It is used in \cite{LM}  to show that there are indeed far less values of $m$
 with Brauer-Manin counterexamples than the  number of values of $m$ predicted by \cite{GS}
 for counterexamples to the integral Hasse principle.

\bigskip

{\bf Notation} Let $k$ be a field and $\overline k$ a separable closure of $k$.
We let $g=g_{k}=\Gal({\overline k}/k)$ be the absolute Galois group.
A $k$-variety is a separated $k$-scheme of finite type.
If $X$ is a $k$-variety, we write $\overline{X}=X \times_{k} \overline{k}$.
We let $k[X]=H^0(X,O_{X})$ and $\overline{k}[X]= H^0({\overline X},O_{\overline X})$.
If $X$ is an integral $k$-variety, we let $k(X)$ denote the function field of $X$.
If $X$ is a geometrically integral $k$-variety, we let $\overline{k}(X)$ denote
the function field of $\overline{X}$.
We  let $\Pic(W)=H^1_{Zar}(W,\BG_{m})=
H^1_{\et}(W,\BG_{m})$ denote the Picard group of a scheme $W$. We let
$\Br(W)=H^2_{\et}(W,\BG_{m})$ denote the Brauer group of  a scheme $W$.
Suppose $W$ is a smooth integral $k$-variety.  The natural map $\Br(W) \to \Br(k(W))$ is injective,
hence $\Br(W) $ is a torsion group. An element of $\Br(k(W))$ whose order is prime to
the characteristic of $k$ belongs to $\Br(W)$ if and only its
 residues at all codimension 1 points of $W$ vanish.
We let $$\Br_{1}(X) = \Ker [\Br(X) \to \Br({\overline X})]$$ denote the algebraic Brauer group
of a $k$-variety $X$ and we let $\Br_{0}(X) \subset \Br_{1}(X)$ denote the image of
$\Br(k) \to \Br(X)$. The image of $\Br(X) \to \Br({\overline X})$ is sometimes referred to
as the ``transcendental Brauer group'' of $X$.

Given a field $F$ of characteristic zero containing a primitive $n$-th root of unity $\zeta=\zeta_{n}$,
  we have $H^2(F,\mu_{n}^{\otimes 2}) = H^2(F, \mu_{n}) \otimes \mu_{n}.$
  The choice of $\zeta_{n}$ then defines an isomorphism
  $\Br(F)[n] = H^2(F,\mu_{n}) \cong H^2(F,\mu_{n}^{\otimes 2})$.
  Given two elements
 $f,g \in F^{\times}$, they have classes  $(f)$ and $(g)$ in $F^{\times}/F^{\times n}= H^1(F,\mu_{n})$.
 One denotes  $(f,g)_{\zeta} \in \Br(F)[n]=H^2(F,\mu_{n})$
 the class corresponding to the cup-product $(f) \cup (g) \in H^2(F,\mu_{n}^{\otimes 2})$.
 Suppose $F/E$ is a finite Galois extension with Galois group $G$.
 Given $\sigma \in G$ and $f,g \in F^{\times}$, we have
 $\sigma((f,g)_{\zeta_{n}})= (\sigma(f),\sigma(g))_{\sigma(\zeta_{n})} \in \Br(F)$.
 In particular, if $\zeta_{n} \in E$, then  $\sigma((f,g)_{\zeta_{n}})= (\sigma(f),\sigma(g))_{\zeta_{n}}$.
 For all this, see \cite[\S 4.6, \S 4.7]{GS06} and in particular \cite[Prop. 4.7.1]{GS06}.

  Let $R$ be a discrete valuation ring with field of fractions $F$ and residue field $\kappa$.
   Let $v$ denote the valuation $F^{\times} \to \Bbb Z$.
   Let $n>1$ be an integer invertible in $R$. Assume $F$ contains a primitive $n$-th
   root of unity $\zeta$.   For $f,g, \in F^{\times}$, we have  the residue map
   $$\partial_{R} : H^2(F, \mu_{n}) \to H^1(\kappa,\Bbb Z/n) \cong H^1(\kappa,\mu_{n}) = \kappa^{\times}/\kappa^{\times n},$$
   where $H^1(\kappa,\Bbb Z/n) \cong H^1(\kappa,\mu_{n}) $ is induced by the isomorphism
   $\Bbb Z /n \simeq \mu_{n}$ sending $1$ to 
   $\zeta$.
   This map sends the class of $(f,g)_{\zeta} \in \Br(F)[n]=H^2(F, \mu_{n})$ to
   \begin{equation}{\label{explicitresidue}}
(-1)^{v(f)v(g)}  \ {\rm class}(g^{v(f)}/f^{v(g)}) \in \kappa^{\times}/\kappa^{\times n}.
   \end{equation}
For a proof of these well known facts, see \cite{GS06}. Here are precise references.
Residues in Galois cohomology with finite coefficients are defined in   
 \cite[Construction 6.8.5]{GS06}. Comparison of residues in 
 Milnor $K$-Theory and Galois cohomology is given
 in \cite[Prop. 7.5.1]{GS06}. The explicit formula for the residue
 in Milnor's group $K_{2}$ of a discretely valued field
  is given in \cite[Example 7.1.5]{GS06}.

\bigskip

{\bf Structure of the paper}

Let $k$ be a field of characteristic zero.
 Let $m \in k $. Assume $m(m-4) \neq 0$.
Let $X_{m} \subset {\Bbb P}^3_{k}$ be the smooth cubic surface defined by the projective equation
$$t(x^2+y^2+z^2)-xyz= mt^3.$$
Let $U=U_{m} \subset X_{m}$ be the smooth affine cubic surface defined by the affine equation
$$x^2+y^2+z^2-xyz= m.$$

 In \S 2, we study the Galois modules $\Pic(\overline X_{m}), \Pic(\overline U_{m}), \Br(\overline U_{m})$.
 We show $\Br(\overline U_{m}) \simeq \Bbb Q/\Bbb Z (-1)$.
 In \S 3, we compute $\Br(X_{m})=\Br_{1}(X_{m})$ and the algebraic part $\Br_{1}(U_{m})$ of $\Br(U_{m})$.
 In \S 4, we compute  the transcendental part of $\Br(U_{m})$, namely the quotient $\Br(U_{m})/\Br_{1}(U_{m})$.
We then turn to the case $k=\Bbb Q$ and $m$ is an integer.
In \S 5, we show how to compute the integral Brauer-Manin obstruction for the
affine scheme $\mathcal{U}_{m}$ over $\Bbb Z$ defined by $x^2+y^2+z^2-xyz= m.$
 We then show that the counterexamples to the integral Hasse principle for $\mathcal{U}_{m}$  in \cite{GS}
may all be explained by a combination of integral Brauer-Manin obstruction and
reduction theory. We increase the stock of such counterexamples, thus leading to an improvement
on a counting result in \cite{LM}. In \S 6, we prove that strong approximation never holds
for Markoff type surfaces. 
Section \S 7 is an appendix giving the structure of
the real locus $U_{m}(\Bbb R)$ depending on the value of $m\in \Bbb R$.

\section{Computation of Brauer groups I, general setting}

\begin{prop}\label{br-tran}
Let $X$ be a smooth, projective, geometrically rational surface over a field $k$ of characteristic zero.
Suppose that $U$ is an open subset of $X$ such that $X\setminus U$ is
  the union of three distinct
$k$-lines, by which we mean a smooth projective curve  isomorphic to ${\bf P}_k^1$.
Suppose any two lines intersect each another transversely in one point, and
that the  
  three  intersection points are distinct.
Let  $L$ be  one of the three lines and $V \subset L$ be the complement of
the 2  intersection points of $L$ with the other two lines. Then
the residue map $$\partial_{L}: \ \Br(\bar k(X)) \to H^1(\bar k(L),\Bbb Q/\Bbb Z)$$
induces a $g$-isomorphism
 $$\Br(\overline U) \xrightarrow{\cong}  H^1(\overline V,\Bbb Q/\Bbb Z) \simeq  H^1(\overline {\Bbb G}_{m}, \Bbb Q/\Bbb Z) \simeq \Bbb Q/\Bbb Z(-1).$$
  \end{prop}

\begin{proof}  Since $X$ is smooth, the homology of the  Bloch-Ogus complex
$$ H^2(\bar k(X), \Bbb Q/\Bbb Z(1)) \to \oplus_{x \in \overline X^{(1)} } H^1(\bar k(x), \Bbb Q/\Bbb Z)  \to    \oplus_{x \in \overline X^{(2)} } H^0(\bar k(x), \Bbb Q/\Bbb Z(-1))$$
at the second term is $H^1_{Zar}(\overline X,{\mathcal H}^2_{\overline X}(\Bbb Q/\Bbb Z(1)) )$ by \cite[(6.1) Theorem]{BO}.  The spectral sequence
$$ E_2^{p,q}=H^p_{Zar}(\overline X, {\mathcal H}^q_{\overline X} (\Bbb Q/\Bbb Z(1))) \Rightarrow H^{p+q}_{\et}(\overline X, \Bbb Q/\Bbb Z(1))  $$
in \cite[(6.3) Corollary]{BO} implies that $H^1_{Zar}(\overline X,{\mathcal H}^2_{X}(\Bbb Q/\Bbb Z(1)) )$ is a subgroup of $H^3_{\et}(\overline X, \Bbb Q/\Bbb Z(1))$.
Since $$H^1_{\et}(\overline X,\mu_{n})= \Pic(\overline X)[n]=0$$ for all $n>0$ by the Kummer sequence, one has $$H^3_{\et}(\overline X, \Bbb Q/\Bbb Z(1))= \varinjlim_{n} H^3_{\et}(\overline X,\mu_{n})=0$$ by Poincar\'e duality.
Therefore the above Bloch-Ogus complex is exact.

 Since $X$
is a smooth,
projective,
 geometrically rational surface, $\Br(\overline X)=0$ and the following diagram of exact sequences
$$ \xymatrix{
& \Br(\overline X)=0  \ar[r] & H^2(\bar k(X), \Bbb Q/\Bbb Z(1)) \ar[r] \ar[d]_{\simeq}   & \oplus_{x \in \overline X^{(1)} } H^1(\bar k(x), \Bbb Q/\Bbb Z)  \ar[d] \\
0 \ar[r] & \Br(\overline U)  \ar[r] & H^2(\bar k(U), \Bbb Q/\Bbb Z(1)) \ar[r] & \oplus_{x\in \overline U^{(1)}} H^1(\bar k(x), \Bbb Q/\Bbb Z)  }$$
commutes by \cite[(3.9)]{CT}.
Let $\{L_{1}, L_{2}, L_{3}\}$ be the set of three lines in $X\setminus U$ and let $\{P_1, P_2, P_3\}$ be the set of three intersection points of $L_1, \ L_2$ and $L_3$ such that $P_i \not \in L_i$ for $1\leq i\leq 3$. Set $$V_i=L_i \setminus \{ P_j\}_{j\neq i}\simeq_k \Bbb G_m$$ for $1\leq i\leq 3$. Combining the above diagram  with the above Bloch-Ogus exact sequence yields the following exact sequence, where the maps are given by the residues
$$ 0\rightarrow \Br(\overline U) \rightarrow \oplus_{i=1}^3 H_{\et}^1(\overline V_{i}, \Bbb Q/\Bbb Z) \rightarrow \oplus_{i=1}^3 H^0(\bar k(P_i), \Bbb Q/\Bbb Z(-1)) . $$
For each $i$,  we have $V_{i} \simeq \mathbb G_{m}$.
The residue map induces the following short exact sequence
$$ 0\rightarrow H^1_{\et}(\overline V_i, \Bbb Q/\Bbb Z) \rightarrow \oplus_{j\neq i} H^0_{\et}(\bar k(P_j), \Bbb Q/\Bbb Z(-1))\xrightarrow{\sum_{j\neq i}} \Bbb Q/\Bbb Z \rightarrow 0.$$
After twisting by roots of unity, this simply follows from the exact sequence
$$ 1 \to \overline{k}^\times \to \overline{k}[ \mathbb G_{m}]^\times  \to \mathbb{Z} \oplus    \mathbb{Z}  \to  \mathbb{Z} \to 0$$
induced by the map sending a rational function on $\mathbb G_{m}$ to its
  divisor  at $0$ and at $\infty$.
One thus has
$g$-isomorphisms
$$ \Br(\overline U)  \simeq H_{\et}^1(\overline V_i, \Bbb Q/\Bbb Z) \simeq H^1(\overline {\Bbb G}_{m}, \Bbb Q/\Bbb Z) \simeq \Bbb Q/\Bbb Z(-1) $$ for $1\leq i\leq 3$.
\end{proof}

For cubic surfaces over an   algebraically closed field $k$, one has the following result.

\begin{prop} \label{unit}
Let $X   \subset {\bf P}^3_{k}$ be a smooth, projective, cubic surface over a  field $k$ of characteristic zero.
Suppose a plane ${\bf P}^2_{k}  \subset {\bf P}^3_{k}$ cuts out on $\bar X$ three  lines $L_{1}, L_{2}, L_{3}$ over $\bar k$.
Let $U \subset X$ be the complement of this plane.
 Then the natural map $\bar k^{\times} \to \bar k[U]^{\times}$ is an isomorphism of Galois modules and the natural map
 $$ 0  \to \oplus_{i=1}^3 \Bbb Z L_{i} \to \Pic(\overline X ) \to \Pic(\overline U) \to 0  $$  is an exact sequence of Galois lattices.
 \end{prop}

\begin{proof}  We may assume $k=\bar k$.
Let
 $$aL_{1}+bL_{2}+cL_{3}=0 \in \Pic(X)$$ with $a,b, c \in \Bbb Z$. By the assumption that $(L_{i}.L_{i})=-1$ and $(L_{i}.L_{j})=1$ for $i\neq j$, one has
$$-a+b+c=0, \ \  a-b+c=0, \ \  a+b-c=0 . $$  This implies that $a=b=c=0$.

To complete the proof, one only needs to show that $\Pic(U)$ is
torsion free.

 Let $e_1, e_2, \cdots, e_6$ and $l$ be  given by \cite[Chapter V, Proposition 4.8]{Hartshorne}.

Suppose that one of $L_1, L_2$ and $L_3$ is in $\{e_1, \cdots, e_6\}$. Say that $L_1=e_1$. 
Consider the two disjoint sets of classes of lines on $X$ :
$$\{l-e_1-e_i:  \  2\leq i \leq 6 \} \ \ \ \text{and} \ \ \  \{ 2l-\sum_{k\neq i} e_k: \ 2\leq i\leq 6 \}.$$ By inspecting the intersection property of $L_1, L_2, L_3$,
 one sees that
  $L_2$ is in one of these sets, and  $L_3$    is in the other one. 
 Without loss of generality, one can assume that $L_2=l-e_1-e_2$. Then $$L_3= 2l-\sum_{k\neq 2} e_k . $$ By \cite[Chapter V, Proposition 4.8]{Hartshorne}, one concludes that
$\Pic(X)/(\oplus_{i=1}^3 \Bbb Z L_{i}) $ is free.

Otherwise, all $L_1$, $L_2$ and $L_3$ are in $\{l-e_i-e_j:  \  1\leq i <j \leq 6 \}$. Say $$L_1=l-e_1-e_2,  \ \ L_2=l-e_3-e_4 \ \ \text{and} \ \  L=l-e_5-e_6 . $$  Then
$ \Pic(X)/(\oplus_{i=1}^3 \Bbb Z L_{i}) $ is free by \cite[Chapter V, Proposition 4.8]{Hartshorne}.

\medskip

{\it Alternative completion of the proof}
The first argument shows that $L_{1}, L_{2}, L_{3}$ are linearly independent.
It also shows that $k^{\times} = k[U]^{\times}$.
Since the determinant of  the system of equations is $\pm 4$, and $\Pic(X)$ is torsion free,
the only torsion that could exist  in $\Pic(U) $ is 2-primary.
Let us show there is no 2-torsion in $\Pic(U)$. If there was, there would exist a principal
divisor  on $X$ of the shape $2D + L_{1}$,
or $2D+L_{1}+L_{2}$, or $2D+L_{1}+L_{2}+L_{3}$.
By the well known configuration of the
27 lines on a cubic surface, there exists a line $L$ on $X$ which meets $L_{1}$ in one point
and does not meet $L_{2}$ or $L_{3}$. Intersection with $L$ rules out the three possibilities.
\end{proof}

The following corollary applies to number fields and more generally
to function fields of varieties over a number field.

\begin{cor}
Let $k$ be a field of characteristic zero such that in any finite field extension
there are only finitely many roots of unity.
Let $X   \subset {\bf P}^3_{k}$ be a smooth, projective, cubic surface over $k$.
Suppose a plane cuts out on $X$ three nonconcurrent  lines.
Let $U \subset X$ be the complement of the plane section.
Then the quotient $\Br(U)/\Br_0(U)$ is finite.
\end{cor}

\begin{proof}
Let  $g={\rm Gal}(\overline{k}/k)$ where $\overline{k}$ is an algebraic closure of $k$.
Since $\overline{k}^{\times} = \overline{k}[U]^{\times}$, we have an exact sequence
$$ \Br(k) \to \Ker[\Br(U) \to \Br({\overline U})^g ] \to H^1(g,\Pic({\overline U})) $$
by \cite[Lemma 2.1]{CTXu09}.
Since $\Pic({\overline U})$ is free of finite rank by Proposition \ref{unit},  $H^1(g,\Pic({\overline U}))$ is finite.

  Let $K \subset {\overline k}$ be a field over which one of the three lines, call it $L$,  is defined.
Let $g_{K}={\rm Gal}({\overline k}/K)$.
The isomorphism
$$ \Br({\overline U}) \xrightarrow{\cong} \Bbb Q/\Bbb Z(-1)$$
attached to the line $L$
is $g_{K}$-equivariant.  We thus have
$$ \Br({\overline U})^g \subset \Br({\overline U})^{g_{K}} \simeq \Bbb Q/\Bbb Z(-1)^{g_{K}}$$
Since there are finitely many roots of unity in $K$, the group $\Bbb Q/\Bbb Z(-1)^{g_{K}}$
is finite (use Lemma \ref{twist}). Thus $ \Br({\overline U})^g$ is finite.
 The result  now follows from the above exact sequence.
 \end{proof}

\begin{lem}\label{twist} Let $k$ be a field of characteristic 0.
Let  $g={\rm Gal}(\overline{k}/k)$.
 Let $\mu_{\infty}({\overline k})=\Bbb Q/\Bbb Z(1)$ be the subgroup of roots of unity in ${\overline k}^{\times}$.
  Then
${\mathbb Q}/{\mathbb Z}(-1)^{g} $ is (noncanonically) isomorphic
to   $\mu_{\infty}(k)$, the group of roots of unity in $k$.
\end{lem}

\begin{proof}  We only need to show:   $\mathbb Z/n \subset {\mathbb Q}/{\mathbb Z}(-1)^{g} $  holds if and only if $\mu_n\subset k$.

If $\mu_n\subset k$, obviously $\mathbb Z/n \subset {\mathbb Q}/{\mathbb Z}(-1)^g $. On the other hand,
let $a \in {\mathbb Q}/{\mathbb Z}(-1)$ be of order $n$. For any $\sigma\in g$, then $\sigma(a) =\chi(\sigma)^{-1}a$, here $\chi$ is the cyclotomic character. Therefore, if $a$ is a fixed point, then $(\chi(\sigma)-1)a =0$ for any $\sigma \in g$, i.e., $\chi(\sigma)-1\equiv 0 \mod n$. This implies $\mu_n\subset k$.
\end{proof}

\section{Computation of Brauer groups II, algebraic parts}

For Markoff surfaces,
one can further compute the algebraic part of Brauer groups explicitly by using the equations.

\begin{lem} \label{lines}
 Let $k$ be a field of characteristic zero and $\overline k$   an algebraic closure of $k$.
Let $m\in k$ and $d=m-4$.
Let $X_{m} \subset \Bbb P^3_{k}$ be  defined by the equation
$$ t(x^2+y^2+z^2)-xyz=mt^3.$$
  Then $X_{m}$ is smooth over $k$ if and only if
 $md \neq 0$.  If $md\neq 0$, fix a square root $ \sqrt{m} \in \overline{k}$ and a square root  $\sqrt{d} \in \overline{k}$.
 Then the 27 lines on  $\overline{X}_{m}$ are defined over $k(\sqrt{m}, \sqrt{d})$   by the following equations
$$ L_1: \ x=t=0; \ \ \ \ \ \ L_2: \ y=t=0; \ \ \ \ \ \ L_3: \ z=t=0$$
and
$$\begin{cases}  & l_1(\epsilon, \delta) :  \ x= 2\epsilon t, \ y-\epsilon z=\delta \sqrt{d} t \\
& l_2(\epsilon, \delta) : \ y= 2\epsilon t, \ z-\epsilon x= \delta \sqrt{d} t \\
& l_3(\epsilon, \delta): \ z= 2\epsilon t, \ x-\epsilon y =\delta \sqrt{d} t \\
& l_4(\epsilon, \delta): \ x=\epsilon \sqrt{m} t, \ y=\frac{1}{2}(\epsilon \sqrt{m} + \delta \sqrt{d} ) z  \\
& l_5(\epsilon,\delta) : \  y= \epsilon \sqrt{m} t, \ z=\frac{1}{2}(\epsilon \sqrt{m} + \delta \sqrt{d} ) x \\
& l_6(\epsilon, \delta): \ z= \epsilon \sqrt{m} t, \ x=\frac{1}{2} (\epsilon \sqrt{m} + \delta \sqrt{d}) y
\end{cases} $$
 with $\epsilon=\pm 1$ and $\delta = \pm 1$. Moreover, the intersection numbers
 satisfy $$(l_i(\epsilon,\delta). l_j(\epsilon,\delta))=0 $$ for any fixed pair $(\epsilon, \delta)$, whenever $1\leq i\neq j\leq 6$.
\end{lem}

\begin{proof}
For $m=4$, the singular points  are
$$(x:y:z:t)=(2\varepsilon : 2\eta : 2 \varepsilon \eta : 1)$$
with $\varepsilon=\pm 1, \eta =\pm 1$.
 For $m=0$, there is only one singular point, namely $(0:0:0:1)$.
Assume $m \neq 0,4$. 
Any line $L$  on $X_{m}$ which is not in the plane $t=0$ meets this plane in one point,
and that point must be on one of the lines $L_{1}, L_{2}, L_{3}$.
Say it is $L_{1}$. The plane containing $L$ and $L_{1}$ is one of the planes
through $L_{1}$ which intersects $X_{m}$ in three lines. Writing down the   planes
through each $L_{i}$ with this property (there are 5 such planes for each $L_{i}$) produces all    lines on $X_{m}$,
which are indeed 27 in number.
 \end{proof}
\medskip

For the  sake of simplicity,   wherever there is no ambiguity, for  each $i=1,\dots,6$ we shall write
$l_{i} = l_{i}(1,1)$ .

\begin{prop}\label{x}
Let $k$ be a field of characteristic zero and
 $m\in k\setminus \{0,4\}$. Set $d=m-4$.
Let $ X_{m} \subset \Bbb P^3_{k}$ be defined by the equation
\begin{equation} \label{equ-hom} t(x^2+y^2+z^2)-xyz=mt^3. \end{equation}
If $[k(\sqrt{m}, \sqrt{d}) : k]=4$, then
$$ \Br(X_{m})/\Br_0(X_{m}) =\Br_1(X_{m})/\Br_0(X_{m}) \cong \Bbb Z/2 $$ with a generator
$$ \{ ((\frac{x}{t})^2-4,d) =((\frac{y}{t})^2-4, d)= ((\frac{z}{t})^2-4, d) \} $$ over $t\neq 0$.

If $d\not\in k^{\times 2}$ and $m\in k^{\times 2}$, then
$$ \Br(X_{m})/\Br_0(X_{m})=\Br_1(X_{m})/\Br_0(X_{m}) \cong (\Bbb Z/2)^2 $$ with two generators $$ \{ ((\frac{x}{t})^2-4,d), \ ((\sqrt{m}-\frac{x}{t})(\frac{x}{t}+2), d) \} $$ over $t\neq 0$.

If $d\in k^{\times 2}$ or $d \cdot m\in k^{\times 2}$, then $\Br(k)= \Br_1(X_{m})=\Br(X_{m})$
\end{prop}

\begin{proof} For ease of notation, we set $X=X_{m}$.
Since $X$ is geometrically rational, one has $\Br(X)=\Br_1(X)$. One clearly has $X(k)\neq \emptyset$.
By the Hochschild-Serre spectral sequence (see \cite[Lemma 2.1]{CTXu09}), one has an isomorphism
 \begin{equation} \label{b-x} \Br_1(X)/\Br_0(X) \simeq H^1(k, \Pic({\overline X})) . \end{equation}
 
  By  Lemma \ref{lines}, the six lines $l_i, i=1,\dots,6$ on the cubic surface ${\overline X}$  are skew to one another, hence may be simultaneously blown down
 to $\BP^2$ (see \cite[Chapter V, Proposition 4.10]{Hartshorne}).
The class $\omega$ of the canonical bundle on ${\overline X}$
is equal to $-3l + \sum_{i=1}^6 l_{i}$, where $l$ is the inverse image of the class of lines in $\BP^2$.
 We have the following intersection properties:
  $(l.l)=1$ and  $(l.l_i)=0$ for $1\leq i\leq 6$. The classes $l$ and $l_{i}, i=1,\dots,6$ form a basis of $\Pic({\overline X})$.

  Since
 $$ (L_j. l_{i}) =\begin{cases} 1 \ \ \ & \text{$i-j \equiv 0$ or $3 \mod 6$} \\
 0 \ \ \ & \text{otherwise} \end{cases} $$
where $L_j$ are the lines in Lemma \ref{lines} with $1\leq j\leq 3$ and $1\leq i\leq 6$, one concludes that
\begin{equation}\label{l} L_j= l-l_j -l_{j+3} \end{equation} in $\Pic({\overline X})$ for $1\leq j\leq 3$ by \cite[Chapter V, Proposition 4.8 (e)]{Hartshorne}.  

\bigskip

 (1) Suppose $d\not\in k^{\times 2}$ and $md \not\in k^{\times 2}$.

There is $\sigma \in \Gal(k(\sqrt{d}, \sqrt{m})/k)$ such that $$\sigma(\sqrt{d})=-\sqrt{d} \ \ \ \text{and} \ \ \  \sigma(\sqrt{m})=\sqrt{m} . $$
Since the intersection numbers
\begin{equation}\label{int} (\sigma l_j(1,1). l_{i}(1,1))=(l_j(1,-1). l_i(1,1))=\begin{cases}  0 \ \ \ & \text{$i=j+3$} \\
1 \ \ \ & \text{$ i\neq j+3 $} \end{cases} \end{equation} and
\begin{equation}\label{int1} (\sigma l_{3+j}(1,1). l_i(1,1))=(l_{3+j}(1,-1). l_i(1,1)) =\begin{cases} 0 \ \ \ & \text{$i=j$} \\
1 \ \ \ & \text{$ i\neq j$} \end{cases} \end{equation} for $1\leq j\leq 3$,  one obtains
 \begin{equation} \label{rep}\sigma l_j= 2l-\sum_{i\neq j+3} l_i  \ \ \ \text{and} \ \ \ \sigma l_{3+j}= 2l -\sum_{i\neq j} l_i  \end{equation} in $\Pic({\overline X})$ by \cite[Chapter V, Theorem 4.9]{Hartshorne} for $1\leq j\leq 3$. This implies that
 \begin{equation} \label{rep1} \sigma l= 5l-2 \sum_{i=1}^6 l_i \end{equation} by (\ref{l}). Then
 \begin{equation} \label{sigma} \ker(1+\sigma) = \langle (l-l_1-l_2-l_3),  (l_1-l_4),  (l_2-l_5),  (l_3-l_6)  \rangle \end{equation}
and \begin{equation}\label{sigma-im} (1-\sigma) \Pic({\overline X}) = \langle 2(l-l_1-l_2-l_3),  (l_1-l_4+l_3-l_6),  (l_2-l_5-l_3+l_6),  (l_2-l_5+l_3-l_6)  \rangle \end{equation}  by (\ref{rep}), (\ref{rep1}).

\bigskip

Given a finite cyclic group $G=\langle\sigma\rangle$ and a $G$-module $M$, recall that we have
isomorphisms $H^1(G,M) \cong \hat{H}^{-1}(G,M)$, where the latter group is the quotient of
${}^{N_{\sigma}}(M)$, the set of elements of $M$ of norm $0$, by its subgroup $(1-\sigma)M$.

\medskip
(1a) Suppose
$d \notin k^{\times 2}$
and $m \in k^{\times 2}$.  Then
$$ H^1(k, \Pic({\overline X}))= H^1(\langle\sigma\rangle, \Pic({\overline X})) \simeq \hat{H}^{-1}(\langle\sigma\rangle, \Pic({\overline X}))  \cong (\Bbb Z/2)^2$$
by \cite[(1.6.6) and (1.6.12) Proposition]{NSW} and (\ref{sigma}) and (\ref{sigma-im}).

\bigskip

(2) Suppose  $m\not\in k^{\times 2}$ and $md \not\in k^{\times 2}$.

There is $\tau\in \Gal(k(\sqrt{d}, \sqrt{m})/k)$ such that  $$\tau(\sqrt{m})=-\sqrt{m} \ \ \ \text{and} \ \ \ \tau(\sqrt{d})=\sqrt{d} . $$ Since the intersection numbers
 \begin{equation} \label{int2} (\tau l_{j+3}(1,1). l_i(1,1)) = (l_{j+3}(-1, 1). l_i(1,1)) = \begin{cases}  0 \ \ \ & \text{$ 1\leq i\leq 3$ and $i=j+3$} \\
1 \ \ \ & \text{$4\leq i\leq 6$ and $i\neq j+3$}  \end{cases} \end{equation}
for $1\leq j\leq 3$, one obtains
\begin{equation} \label{rep-t} \tau l_{j+3} = l-\sum_{4\leq i\neq j+3\leq 6} l_i  \end{equation}  in $\Pic({\overline X})$ by \cite[Chapter V, Theorem 4.9]{Hartshorne} for $1\leq j\leq 3$. This implies that
\begin{equation} \label{rep-t1} \tau l=2l-\sum_{i=4}^6 l_i \end{equation} by (\ref{l}).   Then
\begin{equation} \label{tau} \ker(1+\tau)= \langle l-l_4-l_5-l_6 \rangle  \ \ \ \text{and} \ \ \  \ker(1-\tau) = \langle l_1, l_2, l_3, (l-l_4), (l-l_5), (l-l_6) \rangle \end{equation}
and
\begin{equation}\label{im} (1-\tau) \Pic({\overline X}) = \langle l-l_4-l_5-l_6 \rangle \end{equation} by (\ref{rep-t}), (\ref{rep-t1}).

\medskip

(2a)
If $m \notin  k^{\times 2}$ and $d\in k^{\times 2}$,
then
$$ H^1(k, \Pic({\overline X}))=H^1(\langle\tau\rangle, \Pic({\overline X})) \simeq \hat{H}^{-1}(\langle\tau\rangle, \Pic({\overline X})) =0$$
by \cite[(1.6.6) and (1.6.12) Proposition]{NSW} and (\ref{tau}) and (\ref{im}).

If  $d\in k^{\times 2}$ and
 $m\in   k^{\times 2}$, then we also have
$ H^1(k, \Pic({\overline X}))=0$. Indeed, in that case all 27 lines are defined over $k$
and the action of the Galois group on $\Pic({\overline X})$ is the trivial action.

\bigskip

(3) Suppose that none of $d$, $m$, $dm$ is a square, that is
$[k(\sqrt{m}, \sqrt{d}) : k]=4$.

 Then
$$ H^1(k, \Pic({\overline X}))= H^1(G, \Pic({\overline X})) $$ by \cite[(1.6.6) Proposition]{NSW}, where $G=\Gal(k(\sqrt{m}, \sqrt{d})/k)$. Let $\sigma, \tau\in G$ be as above. Then one has the following exact sequence
$$ 0\rightarrow H^1(\langle \sigma\rangle, \Pic ({\overline X})^{\langle\tau\rangle}) \rightarrow H^1(G, \Pic({\overline X})) \rightarrow H^1(\langle\tau\rangle, \Pic({\overline X})) =0$$
by  \cite[(1.6.6) and (1.6.12) Proposition]{NSW} and (\ref{tau}) and (\ref{im}). Since
$$ \ker(1+\sigma) \cap \Pic({\overline X})^{\langle \tau\rangle} = \langle  (l-l_4-l_2-l_3),  (l-l_5-l_1-l_3),  (l-l_6-l_1-l_2)  \rangle$$ by (\ref{sigma}), (\ref{tau}) and
$$ (1-\sigma) \Pic({\overline X})^{\langle \tau\rangle} = [(1-\sigma )\Pic({\overline X})]\cap \Pic({\overline X})^{\langle \tau\rangle} $$
$$= \langle (2l-l_1-2l_2-l_3-l_4-l_6), (l_2-l_3-l_5+l_6), (2l-2l_1-l_2-l_3-l_5-l_6)\rangle $$
by (\ref{rep}), (\ref{rep1}), (\ref{sigma-im}), (\ref{tau}) and (\ref{im}), one concludes that
$$H^1(k, \Pic({\overline X}))= [\ker(1+\sigma) \cap \Pic({\overline X})^{\langle \tau\rangle}] / [(1-\sigma) \Pic({\overline X})^{\langle \tau\rangle}] \cong \Bbb Z/2 . $$

\bigskip

(4) Suppose $m, d \notin k^{\times 2}$ and $md \in  k^{\times 2}$,
i.e. $k(\sqrt{m})=k(\sqrt{d})\neq k$.

 Let $\rho$ be the generator of $\Gal(k(\sqrt{m})/k)$. Computing the intersection numbers $$ (\rho l_{j+3}(1,1). l_i(1,1))=(l_{j+3}(-1,-1). l_i(1,1)) =\begin{cases} 1 \ \ \ & \text{$1\leq i\neq j \leq 3$} \\
0 \ \ \ & \text{otherwise} \end{cases} $$ for $1\leq j\leq 3$,
one obtains  
\begin{equation} \label{rho} \rho l_{j+3} = l-\sum_{1\leq i\neq j\leq 3} l_i  \end{equation}  for $1\leq j\leq 3$. Then
\begin{equation} \label{rho1} \rho l= 4l -\sum_{i=1}^3 l_i - \sum_{i=1}^6 l_i \end{equation} by (\ref{rep}) and (\ref{rho}).  Since
$$ \ker(1+\rho) =(1-\rho) \Pic({\overline X})=\langle (l-l_2-l_3-l_4), (l-l_1-l_3-l_5), (l-l_1-l_2-l_6) \rangle $$ by (\ref{rep}), (\ref{rho}) and (\ref{rho1}), one concludes that
$$H^1(k, \Pic({\overline X}))=H^1(\langle\rho\rangle, \Pic({\overline X})) \cong \hat{H}^{-1}(\langle\rho\rangle, \Pic({\overline X})) =0.
 $$

\bigskip

Now we produce  concrete generators in $\Br_1(X)$ for  $\Br_1(X)/\Br(k) \cong H^1(k, \Pic({\overline X}))$. If $d\in  k^{\times 2}$ or $md \in  k^{\times 2}$, we have just seen that $\Br_1(X)/\Br(k)=0$.
Let us consider the other cases.

\medskip

 Let $U$ be the open subset of $X$ defined by $t\neq 0$. Then   equation (\ref{equ-hom}) is equivalent to
\begin{equation}
 (2z-xy)^2- 4d= (x^2-4)(y^2-4)
  \end{equation}
for $U$.  Since
$$ \{x\pm 2=0\} \cap \{((x\mp2)(y^2-4)=0 \} $$ is a closed subset of codimension $\geq 2$ on $U$, one obtains that $(x\pm 2, d)\in \Br_1(U)$. This implies that
$$ B= (x^2-4,d) =(y^2-4, d)= (z^2-4, d)\in \Br_1(U).  $$
The residues of $B$ at the lines $L_1$, $L_2$ and $L_3$ which form the complement of $U$ in $X$
 (cf.  Lemma \ref{lines}) are easily seen to be trivial.  One thus has $B\in \Br_1(X)$.

\bigskip

If  $m\in k^{\times 2}$,
 equation (\ref{equ-hom}) is equivalent to
 $$  (2y-\sqrt{m}z)^2- dz^2= 4(x-\sqrt{m})(yz-x-\sqrt{m}) $$
for $U$.  Then $(\sqrt{m}-x, d) \in \Br_1(U)$ by the same argument as above. This implies that $$M=((x+2)(\sqrt{m}-x),d)\in \Br_1(U). $$
Then $M\in \Br_1(X)$ by computing the residues of $M$ at $L_1$, $L_2$ and $L_3$ as above.

\bigskip 

To show that these elements
 $B$ and $M$
are not constant, one uses the conic fibration
$$ \pi:  \ U\rightarrow \Bbb A^1; \ (x, y, z) \mapsto x . $$ The generic fibre $U_\eta \xrightarrow{\pi_\eta} \eta$ induces $$\pi_{\eta}^*: \Br(\eta) \rightarrow \Br(U_\eta) \ \ \ \text{with} \ \ \  \ker (\pi_\eta^*) =(x^2-4, m-x^2)$$ by \cite[Theorem 5.4.1]{GS06}.

 \medskip

If  $[k(\sqrt{m}, \sqrt{d}) : k]=4$, then the residue of $(x^2-4,d)$ at $(x^2-m)$ is different from that of $(x^2-4, m-x^2)$. This implies that $\pi_\eta^*(x^2-4,d)$ is not constant by the Faddeev exact sequence (see \cite[Corollary 6.4.6]{GS06}). Since $\pi_\eta^*(x^2-4,d)$ is the pull-back of $B$ 
by the natural map
$U_\eta\rightarrow U$, 
one concludes that $B$ is not constant, hence $B$ generates
$\Br_1(X)/\Br(k)=\mathbb Z/2$.

\medskip

If  $d\not\in k^{\times 2}$ and $m\in k^{\times 2}$, then we have  the residues
$$  \partial_{P} (x^2-4, d) = \begin{cases} d\in k^\times/k^{\times 2}  \ \ \ & \text{if $P\in \{(x\pm 2)\}$} \\
1\in k^\times/k^{\times 2} \ \ \ & \text{otherwise} \end{cases} $$
and
$$\partial_P ((\sqrt{m}-x)(x+2), d)= \begin{cases} d\in k^\times/k^{\times 2} \ \ \ & \text{if $P\in \{(x+2), (x- \sqrt{m}) \} $} \\
1\in k^\times/k^{\times 2} \ \ \ & \text{otherwise} \end{cases} $$
and
$$ \partial_P (x^2-4, m-x^2) = \begin{cases} d\in k^\times/k^{\times 2} \ \ \ & \text{if $P\in \{(x\pm 2), (x\pm \sqrt{m})\} $} \\
1\in k^\times/k^{\times 2} \ \ \ & \text{otherwise}  \end{cases} $$ for all closed points $P$ of $\Bbb P^1$.
Then $$\pi_\eta^*(x^2-4,d), \ \ \ \pi_\eta^*((\sqrt{m}-x)(x+2), d) \ \ \ \text{ and } \ \ \ \pi_\eta^*((x^2-4,d)\cdot ((\sqrt{m}-x)(x+2), d))$$
are not constant by the Faddeev exact sequence. Therefore $B$ and $M$ have  independent classes in $\Br_1(X)/\Br(k) \cong (\mathbb Z/2)^2$, hence generate that group.
\end{proof}

\medskip

\begin{rem}
If $d\in k^{\times 2}$, then $X_{m}$ contains two skew $k$-rational lines, e.g. $l_{1}$ and $l_{2}$.
If $d \cdot m\in k^{\times 2}$, then  $\overline{X}_{m}$ contains the two  lines
$l_{4}(1,1)$ and $l_{4}(-1,-1)$ defined over the quadratic field extension $k(\sqrt{m})$,
which are  conjugate to each other and do not meet.
As  for any smooth projective cubic surface with this property,  this implies that $X_{m}$ is
$k$-birational to projective space $\Bbb P^2_{k}$.
 This general fact goes back to L. Euler in  the case of the diagonal cubic surface $x^3+y^3+z^3+t^3=0$  and a generalisation is due  to B. Segre. Segre's result was completed by Swinnerton-Dyer's paper \cite{swd}.
 Therefore $\Br(X)=\Br(k)$. We keep this part of the computation in Proposition \ref{x}
 because some intermediate results will later be used.
\end{rem}

\medskip

\begin{thm} \label{br} Let  $k$  be a field of characteristic zero and let  $m\in k \setminus \{0,4\}$ and $d=m-4$.
Let $ U_{m}$ be the affine $k$-variety defined by (\ref{markoff1}).

If $[k(\sqrt{m}, \sqrt{d}) : k]=4$
  then
$$\Br_1 (U_{m})/\Br_0(U_{m}) \cong (\Bbb Z/2)^3 $$ with the generators $\{(x-2,d), \ (y-2, d), \ (z-2, d)\}$.

If   $d \notin k^{\times 2}$ and $dm \in k^{\times 2}$
then
$$\Br_1 (U_{m})/\Br_0(U_{m}) \cong (\Bbb Z/2)^2 $$ with the generators $\{(x-2,d), \ (y-2, d)\}$.

If $d \notin k^{\times 2}$ and $m \in k^{\times 2}$,
then
$$ \Br_1(U_{m})/\Br_0(U_{m}) \cong (\Bbb Z/2)^4 $$ with the generators $\{(x-2,d), \ (y-2, d), \ (z-2, d), \ (x-\sqrt{m}, d) \}$.

Otherwise, i.e. if $d \in k^{\times 2}$,  then $\Br_1(U_{m})=\Br_0(U_{m})$.
\end{thm}
 
 \begin{proof}
 We keep notation as in Lemma \ref{lines}. For ease of notation, we set $U=U_{m}$.
 Let $l\in \Pic({\overline X})$ as in the proof  of Proposition \ref{x}.
 Then $\Pic({\overline U})$ is given by the following quotient group
  \begin{equation}\label{Z4}
  ((\oplus_{i=1}^6 \Bbb Z l_i )  \oplus \Bbb Z l ) /( l-l_j-l_{j+3}:  1\leq j\leq 3) \cong  \oplus_{i=1}^4 \Bbb Z [l_i]
  \end{equation}
 by Proposition \ref{unit} and formula (\ref{l}). Here given a divisor $D$ on $\overline{X}$ we denote by    $[D]$  the image   in $\Pic({\overline U})$
 of its class in $\Pic({\overline X})$.
By Proposition \ref{unit} we have ${\overline k}^{\times}=\overline{k}[U]^{\times}$.
The Hochschild-Serre spectral sequence (see \cite[Lemma 2.1]{CTXu09}) then gives
 an injective homomorphism
 \begin{equation} \label{b-p} \Br_1(U)/\Br_0(U) \hookrightarrow H^1(k, \Pic({\overline U})). \end{equation}
 In fact, it is an isomorphism since the smooth compactification $X$ of $U$ has rational points,
 hence also $U$ (any smooth cubic surface over an infinite field $k$ is $k$-unirational as soon as it has a $k$-rational point).

$\bullet$
Case $[k(\sqrt{m}, \sqrt{d}) : k]=4$. Let $G=\Gal(k(\sqrt{m}, \sqrt{d}) /k)$.
Let $\sigma$ and $\tau$ be the generators of $\Gal(k(\sqrt{m}, \sqrt{d})/k)$ satisfying
$$\sigma (\sqrt{d})=-\sqrt{d}, \ \  \sigma(\sqrt{m}) =\sqrt{m}; \ \ \  \ \tau(\sqrt{d}) =\sqrt{d} , \ \ \tau(\sqrt{m})=-\sqrt{m} . $$

Then in $\Pic({\overline U})$ we have the following equalities
\begin{equation}\label{act}  \sigma ([l_i]) = - [l_i]  \end{equation} for $1\leq i\leq 4$ by (\ref{rep}),
 $\tau ([l_i]) = [l_i] $ for $1\leq i\leq 3$ and
  \begin{equation} \label{act2}  \tau([l_4])=-[l_1] + [l_2]+[l_3] -[l_4]    \end{equation} by (\ref{rep-t}).
Since $\Pic({\overline U})$ is free and $\Gal(\bar k/k(\sqrt{m}, \sqrt{d}))$ acts on $\Pic({\overline U})$ trivially, one obtains that
$$ H^1(G, \Pic({\overline U})) \cong H^1(k, \Pic({\overline U}))$$ by \cite[(1.6.6) Proposition]{NSW}.
 Let $H$ be the subgroup of $G$ generated by $\sigma$.
 Then
 $$\Pic({\overline U})^H=0$$ by the equation (\ref{act}). Therefore
$$ H^1(G,\Pic({\overline U})) \cong H^1(H, \Pic({\overline U}))^{G/H} $$ by \cite[(1.6.6) Proposition]{NSW}.
Since
$$H^1(H, \Pic({\overline U}))\cong
 \hat{H}^{-1}(\langle\sigma\rangle, \Pic({\overline U}))
 \cong \oplus_{i=1}^4 (\Bbb Z/2) [l_i] $$ by \cite[(1.6.12) Proposition]{NSW} and the equation (\ref{act}),
one concludes
$$ H^1(k, \Pic({\overline U})) \cong H^1(H, \Pic({\overline U}))^{G/H} \cong \oplus_{i=1}^3 (\Bbb Z/2) [l_i] $$ by (\ref{act2}).

\medskip

$\bullet$
Case $k(\sqrt{m})=k(\sqrt{d})\neq k$. Let $\rho$ be the generator of $\Gal(k(\sqrt{m})/k)$. Since (\ref{rep}) is still available, 
one has $\rho ([l_i]) = - [l_i] $ for $1\leq i\leq 3$. By (\ref{rho}),  one obtains
$$ \rho ([l_4])=[l_1] - [l_2]-[l_3] + [l_4] . $$
Therefore
$$ H^1(k, \Pic({\overline U})) = H^1(\langle\rho\rangle, \Pic({\overline U})) \cong \hat{H}^{-1}(\langle\rho\rangle, \Pic({\overline U}))
\cong \oplus_{i=1}^2 (\Bbb Z/2) [l_i] . $$

\medskip

$\bullet$
Case $k(\sqrt{d})\neq k(\sqrt{m})= k$. Let $\sigma$ be the generator of $\Gal(k(\sqrt{d})/k)$. Since the intersection formulae (\ref{int}) and (\ref{int1}) are still available, one has $\sigma ([l_i]) = - [l_i] $ for $1\leq i\leq 4$. Then
$$ H^1(k, \Pic({\overline U})) =
H^1(\langle\sigma\rangle, \Pic({\overline U})) \cong \hat{H}^{-1}(\langle\sigma\rangle, \Pic({\overline U}))
\cong \oplus_{i=1}^4 (\Bbb Z/2) [l_i] . $$

 \medskip

$\bullet$
The   remaining case is $d\in k^{\times 2}$. If also $m \in k^{\times 2}$, then the Galois action on
the lattice $\Pic({\overline U})$ is trivial, hence $H^1(k,\Pic({\overline U}))=0$. Suppose
$m \notin k^{\times 2}$.
 Let $\tau$ be the generator of $\Gal(k(\sqrt{m})/k)$. Since
$$ \ker(1+\tau)=\langle [l_1]- [l_2] -[l_3]+ 2  [l_4] \rangle $$ and $$(1-\tau)([l_4])= [l_1]- [l_2] -[l_3]+ 2[l_4]  $$ by (\ref{act2}), one concludes that $H^1(k, \Pic({\overline U}))=0$.

\bigskip

 Let us now produce concrete elements in $\Br_1(U)$. Using equation (\ref{markoff2})
one sees that  the quaternion class $(x\pm 2,d)$ is in $\Br_1(U)$ by the same argument as that in Proposition \ref{x}. 
Similar equations give the same result for $(y \pm 2,d)$ and $(z \pm 2,d)$.

The plane $t=0$ cuts out the three lines $(L_{1},L_{2}, L_{3})$, each with multiplicity 1. The plane $x\pm 2t=0$ cuts out  $L_{1}$ and two  lines
each defined over $k(\sqrt{d})$. From this we compute the residues:
 $$ \partial_{L_i}((x \pm 2  t)/t,  d)=\begin{cases} 1 \in k^\times/(k^{\times})^2  \ \ \  & \text{$i=1$} \\
  d   \in k^\times/(k^{\times})^2 \ \ \ & \text{$i=2$ and $3$.} \end{cases} $$  Similarly, one has
$$ \partial_{L_i}((y \pm 2  t)/t,  d)=\begin{cases} 1 \in k^\times/(k^{\times})^2  \ \ \  & \text{$i=2$} \\
  d   \in k^\times/(k^{\times})^2 \ \ \ & \text{$i=1$ and $3$} \end{cases} $$
  and
  $$ \partial_{L_i}((z \pm 2  t)/t,  d)=\begin{cases} 1 \in k^\times/(k^{\times})^2  \ \ \  & \text{$i=3$} \\
  d   \in k^\times/(k^{\times})^2 \ \ \ & \text{$i=1$ and $2$.} \end{cases} $$
This computation of residues will enable us to establish independence modulo 2 of various
classes in $\Br_1 (U)/\Br_0(U)$.

Using equation (\ref{markoff3}) one gets
 \begin{equation}\label{relation-br}
 ((x-2)(y-2)(z-2),d)=(x^2-4,d).
 \end{equation}

\medskip

 When $[K:k]=4$, the quaternion $(x^2-4,d)$ is not constant by Proposition \ref{x}. Therefore $\{(x-2,d), \ (y-2, d), \ (z-2, d)\}$ is a set of generators of $\Br_1 (U)/\Br_0(U) \cong (\mathbb Z/2)^3$.

 \medskip

 When $k(\sqrt{d})=k(\sqrt{m})\neq k$, then $\{(x-2,d), \ (y-2, d) \}$ is a set of generators of $\Br_1 (U)/\Br_0(U) \cong (\mathbb Z/2)^2$.
\medskip

When $m\in k^{\times 2}$ and $d\not\in k^{\times 2}$,   equation (\ref{markoff1}) can be written as
$$ (2y-\sqrt{m} z)^2-dz^2=4(x-\sqrt{m})(yz-x-\sqrt{m}) . $$
Then $(x-\sqrt{m},d)\in \Br_1(U)$ by the same argument as that in Proposition \ref{x}.  Since $(x-\sqrt{m},d)$ has the same residues as $(x-2,d)$ at $L_i$ for $1\leq i\leq 3$, the class $(x-\sqrt{m},d)$ in $\Br_1(U)/\Br_0(U)$ is different from $(x-2, d)$, $(y-2,d)$ and $(z-2,d)$ by Proposition \ref{x}. Since
$$ ((x-\sqrt{m})(y-2)(z-2), d)= ((x-\sqrt{m})(x+2), d) $$ is not a constant element by (\ref{markoff3}) and Proposition \ref{x},  one concludes that
 $$\{(x-2,d), \ (y-2, d), \ (z-2, d), \ (x-\sqrt{m}, d) \}$$ is a set of generators of  $\Br_1(U)/\Br_0(U) \cong (\mathbb Z/2)^4$.
\end{proof}

\begin{rem} Note  that the classes $\{ (x+2,d), \ (y+2, d), \ (z+2, d)\} $ in  $\Br_1(U_{m})/\Br_0(U_{m})$ in 
Theorem \ref{br} are not independent because (\ref{markoff1}) can also be written as
\begin{equation}
(x+y+z+2)^2-d=(x+2)(y+2)(z+2).
\end{equation}
\end{rem}

\section{Computation of Brauer groups III, transcendental parts}

Let $k$ be a field of characteristic zero, and  $m\in k\setminus \{0,4\}$. Let $d=m-4 \neq 0$.
Let $X \subset \Bbb P^3_{k}$ be the smooth cubic surface defined by the equation
$$ t(x^2+y^2+z^2)-xyz=mt^3.$$
 Let $U$ be the affine open sub-variety of $X$ given by $t\neq 0$, i.e. by the affine equation
 $$ x^2+y^2+z^2-xyz=m.$$
 By Proposition \ref{br-tran}, we have  $\Br({\overline U}) \simeq \Bbb Q/\Bbb Z$.
In this section, we determine the transcendental Brauer group $\Br(U)/\Br_{1}(U) \subset \Br({\overline U})$ of $U$.

We here set
$$l_{i} = l_{i}(1,1)  \  \  \text{and} \ \ l^{-}_{i} = l_{i}(1,-1).$$

For computational reasons, in this section  we  contract ${\overline X}$ to $\Bbb P^2_{\bar k}$ over $\bar k$ by sending the 6 lines  $l^{-}_{i} $ to  6 points. The 3 lines $\{L_{i}\}_{i=1}^3$ correspond to three lines in $\Bbb P^2_{\bar k}$ by this contraction and each of these three corresponding lines passes through one pair among the 6 points by \cite[Chapter V, Theorem 4.9]{Hartshorne}.
We let $l^{-} \in \Pic({\overline X})$ be the inverse of the class of a line in $\Bbb P^2_{\bar k}$.
The contraction induces an isomorphism $$V:= {\overline U}\setminus  \{\bigcup_{i=1}^6  l^{-}_{i}  \}   \simeq \Bbb G_m \times_{\bar k} \Bbb G_m$$ over $\bar k$.

\medskip
Though this will not be used in the paper, it is worth noticing the following consequence.
\begin{prop} \label{pi1}
The (Grothendieck) geometric fundamental group $\pi_{1}(\overline U)$ is trivial.
\end{prop}
\begin{proof} Recall ${\rm char}(k)=0$.
Since $V$ is open in ${\overline U}$,
the group  $\pi_{1}(\overline U)$ is a quotient of $\pi_{1}(V)$. The group $ \pi_{1}(\Bbb G_m \times_{\bar k} \Bbb G_m) = \hat{\Bbb Z}^2$ is abelian.
From the above isomorphism we conclude that $\pi_{1}({\overline U})$ is abelian. It is thus isomorphic to  the profinite completion of the system
of groups $H^1({\overline U}, \Bbb Z/n)$.
By Proposition \ref{unit},  
  ${\overline k}^{\times} \simeq {\overline k}[U]^{\times}$ and $\Pic({\overline U})$ is torsion free. 
  The Kummer sequence then gives $H^1({\overline U}, \Bbb Z/n) \simeq \Pic({\overline U})[n]=0.$
\end{proof}

Using    Proposition \ref{unit} and Lemma \ref{lines}, we get:
  \begin{equation}\label{Z4bis}
 \Pic({\overline U}) = ((\oplus_{i=1}^6 \Bbb Z l_i^{-})  \oplus \Bbb Z l^{-} ) /( l^{-}-l_j^{-}-l_{j+3}^{-}:  1\leq j\leq 3) \cong  \oplus_{i=1}^4 \Bbb Z [l_i^{-}].
  \end{equation}

More precisely,   the composite $\theta$ of the natural maps
$$ \oplus_{i=1}^4 \Bbb Z [l_i^{-}] \to \Pic({\overline X}) \to \Pic({\overline U})$$
is an isomorphism. Under the inverse isomorphism 
$\theta^{-1}$,  the classes  of $l_{i}^{-}$ in $\Pic({\overline U})$ 
for $i=1,2,3,4$ are sent to $[l_{i}^{-}]$, the class of $l_{5}^{-}$
is sent to $[l_{1}^{-}]-[l_{2}^{-}]+[l_{4}^{-}]$, and the class of $l_{6}^{-}$
is sent to $[l_{1}^{-}]-[l_{3}^{-}]+[l_{4}^{-}]$.
The composite map  $$\Bbb Z [l^{-}]  \oplus \oplus_{i=1}^6 \Bbb Z [l_i^{-}]
= \Pic({\overline X}) \to \Pic({\overline U}) \to \oplus_{i=1}^4 \Bbb Z [l_i^{-}]=\Bbb Z^4$$
is given by
\begin{equation}\label{Z4ter}(\chi_{0},  \chi_1,\cdots, \chi_6) \mapsto (\chi_0+ \chi_1+\chi_5+\chi_6,\chi_2-\chi_5,\chi_3- \chi_6,   \chi_{0}+ \chi_4+\chi_5+\chi_6).
\end{equation}

 As we shall see below, the restriction map $\Br({\overline U}) \to \Br(V)$
is an isomorphism.
At least over some field extension of $k$ one may thus
compute the transcendental elements in $\Br({\overline U})$
by pull-back of  $\Br(\Bbb G_m \times_{\bar k} \Bbb G_m) \simeq \Bbb Q/\Bbb Z$.

\begin{thm} \label{tr-br} Let $n$ be a positive integer and $\zeta \in\bar k$ be a primitive $n$-th root of unity.
Keep notation as in Lemma \ref{lines} and Theorem \ref{br}.
Then the unique cyclic group of order $n$ in $\Br({\overline U})$ is generated by the cyclic algebra $R_n=(\frac{f}{g}, \frac{u}{v})_{\zeta}$
of dimension $n^2$, where
$$ \begin{cases} f=\frac{1}{2}(\sqrt{m}-\sqrt{d}-2)xz +\sqrt{d}xt+(2-\sqrt{m})yt +\sqrt{d}zt -\sqrt{m}\cdot \sqrt{d} t^2 \\
g=\frac{1}{2}(\sqrt{m}+\sqrt{d}-2)yz -\sqrt{d}yt+(2-\sqrt{m})xt -\sqrt{d}zt +\sqrt{m}\cdot \sqrt{d} t^2 \\
u= \frac{1}{2}(\sqrt{m}-\sqrt{d} -2) xy +\sqrt{d}yt + (2-\sqrt{m}) zt + \sqrt{d} xt -\sqrt{m}\cdot \sqrt{d} t^2 \\
v= \frac{1}{2}(\sqrt{m}+ \sqrt{d} -2) xz -\sqrt{d} zt +(2-\sqrt{m}) yt -\sqrt{d} xt + \sqrt{m}\cdot \sqrt{d} t^2 \end{cases}.$$
\end{thm}
\begin{proof}

By Bezout's theorem (see \cite[Chapter I, Theorem 7.7]{Hartshorne}),
one has
$$ \begin{cases} \{f=0\} \cap X = L_1+L_3+l_1(1,-1) + l_3(1,1)+ l_4(1,-1) + l_6(1,1) \\
\{g=0\} \cap X = L_2 + L_3 + l_2(1,-1) + l_3(1,1) +l_5(1,-1) + l_6(1,1) \\
\{u=0\} \cap X = L_1 + L_2 +  l_1(1,1) + l_2(1,-1) + l_4 (1,1) +l_5 (1,-1) \\
\{v=0\} \cap X =  L_1+ L_3 + l_1(1,1) + l_3(1,-1) + l_4(1,1) + l_6(1,-1)   \end{cases} $$
where $L_i$ with $1\leq i\leq 3$ and $l_j (\epsilon, \delta)$ with $1\leq j \leq 6$, $\epsilon=\pm 1$ and $\delta=\pm 1$ are given by Lemma \ref{lines}.
 For instance, one checks that each of the lines appearing on the right hand side of the first formula is contained in  the projective quadric defined by $f=0$.
  Since the degree of $f$ is 2 and that of the cubic surface is $3$,
Bezout's theorem implies that the multiplicity of each line in $ \{f=0\} \cap X$ is 1.

This implies:
\begin{equation} \label{div}  \begin{cases} div(\frac{f}{g}) =L_1-L_2 + l_1(1,-1) - l_2(1,-1) + l_4(1,-1) -l_5(1,-1) \\
div(\frac{u}{v}) = L_2-L_3 +l_2(1,-1) - l_3(1,-1) + l_5(1,-1) - l_6(1,-1) . \end{cases}  \end{equation}

Let us first prove that
  the restriction map
$\Br({\overline U}) \to \Br(V)$
is an isomorphism. Indeed, the lines $ l_{i}^{-}= l_{i}(1,-1)$ are skew to one another,
and each of them intersects the plane $t=0$ in just one point, call it $P_{i}$.
Let $m_{i}: = l_{i}^{-} \setminus \{P_{i}\} \cong \mathbb A^1_{\bar k}$.
We thus have an exact sequence
$$0 \to \Br({\overline U}) \to \Br(V) \to \oplus_{i=1}^6 H^1_{\et}(m_{i}, \Bbb Q/\Bbb Z).$$
But $H^1_{\et}(m_{i}, \Bbb Q/\Bbb Z)=  H^1_{\et}(\mathbb A^1_{\bar k}, \Bbb Q/\Bbb Z)=0.$
 We thus have $R_n\in \Br({\overline U})$.

The line $L_{1}$ does not appear in the divisor of $u/v$.
In the divisor of $f/g$ it appears with valuation 1.
The residue of $R_{n}$ at the generic point of $L_{1}$
is thus given by the class in $k(L_{1})^{\times}/k(L_{1})^{\times n}$
 of the rational function induced by
$u/v$ on $L_{1}$.The divisor of that function is
a linear combination of points which in particular contains
$L_{3} \cap L_{1}$ with multiplicity $-1$. Thus the
order of the residue is $n$, and $R_{n}$ itself is of order $n$,
hence generates $\Br({\overline U})[n]$.
\end{proof}

  The 27  lines are defined over any field $E$ containing $k(\sqrt{d}, \sqrt{m})$. Over such a field $E$, we may consider
  the complement  $V/E$ of  the 6 lines $l_{i}^{-}$.
The same localisation argument  together  with 
the property $H^1_{\et}(E,\Bbb Q/\Bbb Z) \simeq H^1_{\et}({\mathbb A}^1_{E}  ,\Bbb Q/\Bbb Z)$
 yields an exact sequence
 $$0  \to \Br(U_{E}) \to \Br(V) \to \oplus_{i=1}^6 H^1(E, \Bbb Q/\Bbb Z).$$

  We are interested in the computation of the transcendental Brauer group over
  the ground field. For this, an explicit computation of residues   
  at the generic points
  of the lines $l_{i}^{-}$
  seems necessary. 

  Since $f,g, u,v$ and  each of the curves $D=l_{i}^{-}$ are defined  over $K=k(\sqrt{d}, \sqrt{m})$,
  using formula (\ref{explicitresidue})
   we can compute the residues $\partial_{D}(R_n)$ over any field $E$ containing $K$ and $\mu_n$ in
$$ H^1(E(D),\Bbb Z/n) \simeq E(D)^\times/E(D)^{\times n}. $$

These residues, as explained above, actually take their values in $E^{\times}/ E^{\times n}$.

\begin{prop}\label{nonclosedresidues}
With notation as above :

For $D=l_2^{-}$,
$ \partial_{D} (R_n)= \frac{\sqrt{m}+\sqrt{d}-2}{\sqrt{m}-\sqrt{d}-2}=-\frac{1}{2}(\sqrt{d}+\sqrt{m}) \in E^{\times}/ E^{\times n} $

For $D= l_5^{-}$, $ \partial_{D}(R_n) = \frac{\sqrt{m}-\sqrt{d}}{2}\cdot \frac{\sqrt{m}+\sqrt{d}-2}{\sqrt{m}-\sqrt{d}-2} =-1 \in E^{\times}/ E^{\times n}. $

$$  \partial_{D} (R_n)= \begin{cases}  -1 \in E^{\times}/ E^{\times n}\ \ \ & \text{$D\in \{l_1^{-}, \ l_3^{-})\}$} \\
\frac{\sqrt{d}-\sqrt{m}}{2} \in E^{\times}/ E^{\times n} \ \ \ & \text{$D\in \{l_4^{-}, \ l_6^{-}\} $}  \end{cases} $$
\end{prop}
\begin{proof}
In the course of our computations, we shall make tacit use of the equality
\begin{equation} \label{obvious}
(\frac{\sqrt{d} - \sqrt{m}}   {2}).
(\frac{\sqrt{d} + \sqrt{m}}{2})= -1
\end{equation}.

Let us compute $\partial_{D}(R_n)$
for $D=l_2^{-}$.
Since
$$ g= [\frac{1}{2}(\sqrt{m}+\sqrt{d}-2)y-\sqrt{d}](z-x+\sqrt{d}) + (y-2) [ \frac{1}{2}(\sqrt{m}+\sqrt{d}-2) x - \frac{1}{2} \sqrt{d} (\sqrt{m} + \sqrt{d}) ]$$
and
$$ u=(2-\sqrt{m})(z-x+\sqrt{d}) + (y-2) [\frac{1}{2}(\sqrt{m}-\sqrt{d} -2) x +\sqrt{d} ], $$ one has
$$\frac{g}{u}= \frac{[\frac{1}{2}(\sqrt{m}+\sqrt{d}-2)y-\sqrt{d}](\frac{z-x+\sqrt{d}}{y-2} )+ [ \frac{1}{2}(\sqrt{m}+\sqrt{d}-2) x - \frac{1}{2} \sqrt{d} (\sqrt{m} + \sqrt{d}) ]}{(2-\sqrt{m}) (\frac{z-x+\sqrt{d}}{y-2}) + [\frac{1}{2}(\sqrt{m}-\sqrt{d} -2) x +\sqrt{d} ]} .$$
Since
$$ \frac{z-x+\sqrt{d}}{y-2}=\frac{xz-y-2}{z-x-\sqrt{d}} $$ by (\ref{markoff1}), one obtains that
$$ \partial_{D}(R_n)=
 - \frac{v}{u} \cdot \frac{g}{f}  = -\frac{v}{f} \cdot \frac{(\sqrt{m}-2) \cdot \frac{x(x-\sqrt{d})-4}{-2\sqrt{d}} +\frac{1}{2}(\sqrt{m}+\sqrt{d}-2) x - \frac{1}{2} \sqrt{d} (\sqrt{m} + \sqrt{d})} {(2-\sqrt{m}) \cdot \frac{x(x-\sqrt{d})-4}{-2\sqrt{d}} +\frac{1}{2} (\sqrt{m}-\sqrt{d} -2) x +\sqrt{d} }   $$
$$ = \frac{v}{f} \cdot \frac{(\sqrt{m}-2)[x(x-\sqrt{d})-4]-(\sqrt{m}+\sqrt{d} -2)\sqrt{d} x+d(\sqrt{m}+\sqrt{d})}{(\sqrt{m}-2)[x(x-\sqrt{d})-4]+(\sqrt{m}-\sqrt{d}-2)\sqrt{d} x +2d}  .$$
Since
$$ f|_{D}= \frac{1}{2}(\sqrt{m}-\sqrt{d}-2) x^2 +\sqrt{d}[3-\frac{1}{2}(\sqrt{m}-\sqrt{d})] x + 2(2-\sqrt{m})-d -\sqrt{m}\cdot \sqrt{d}  $$
and
$$ v|_{D}= \frac{1}{2}(\sqrt{m}+\sqrt{d}-2) x^2 - \sqrt{d}[1 + \frac{1}{2}(\sqrt{m}+\sqrt{d})] x + d+ 2(2-\sqrt{m}) + \sqrt{m} \cdot \sqrt{d} , $$
one concludes that
$$ \partial_{D} (R_n)= \frac{\sqrt{m}+\sqrt{d}-2}{\sqrt{m}-\sqrt{d}-2}=-\frac{1}{2}(\sqrt{d}+\sqrt{m}) \in E(D)^\times/E(D)^{\times n}. $$

For $D= l_5^{-}$, one has
$$ g= [\frac{1}{2} (\sqrt{m}+\sqrt{d}-2) y -\sqrt{d}]\cdot [z-\frac{1}{2}(\sqrt{m}-\sqrt{d}) x] + (y-\sqrt{m})[\frac{1}{2}(2+\sqrt{d}-\sqrt{m})x-\sqrt{d}] $$
and
$$ u= (2-\sqrt{m})[z-\frac{1}{2}(\sqrt{m}-\sqrt{d})x] + (y-\sqrt{m}) [\frac{1}{2}(\sqrt{m}-\sqrt{d}-2)x+\sqrt{d}] . $$ Since
$$\frac{z-\frac{1}{2}(\sqrt{m}-\sqrt{d}) x}{y-\sqrt{m}}=\frac{xz-y-\sqrt{m}}{z-\frac{1}{2}(\sqrt{m}+\sqrt{d})x}  $$ by (\ref{markoff1}),
one obtains that
$$ \partial_{D}(R_n) = -\frac{v}{f} \cdot \frac{\frac{1}{2} (\sqrt{m} + \sqrt{d})(\sqrt{m}-2)  \cdot \frac{(\sqrt{m}-\sqrt{d})x^2-4\sqrt{m}}{-2\sqrt{d}x} +\frac{1}{2}(2+\sqrt{d}-\sqrt{m}) x -  \sqrt{d} } {(2-\sqrt{m}) \cdot  \frac{(\sqrt{m}-\sqrt{d})x^2-4\sqrt{m}}{-2\sqrt{d}x}  +\frac{1}{2} (\sqrt{m}-\sqrt{d} -2) x +\sqrt{d} }   $$
$$ = \frac{v}{f} \cdot \frac{(\sqrt{m}-\sqrt{d})(\sqrt{m} -2) x^2 -2dx+2\sqrt{m}(\sqrt{m}+\sqrt{d})(\sqrt{m}-2)}{(2\sqrt{m}-4) x^2 -2dx +4\sqrt{m}(\sqrt{m}-2)}  $$
$$ = \frac{v}{f}\cdot \frac{(\sqrt{m}-\sqrt{d}) x^2 -2(\sqrt{m}+2) x + 2\sqrt{m}(\sqrt{m}+\sqrt{d})}{2x^2-2(\sqrt{m}+2) x + 4\sqrt{m}} . $$
Since
$$ f|_{D}= \frac{\sqrt{m}-\sqrt{d}-2}{\sqrt{m}+\sqrt{d}} \cdot x^2 +\frac{\sqrt{d}}{2}(\sqrt{m}-\sqrt{d}+2) x + \sqrt{m} (2-\sqrt{m} -\sqrt{d})  $$
and
$$ v|_{D}= \frac{\sqrt{m}+\sqrt{d}-2}{\sqrt{m}+\sqrt{d}} x^2 - \sqrt{d}[1 + \frac{1}{2}(\sqrt{m}-\sqrt{d})] x + \sqrt{m} (\sqrt{d}-\sqrt{m}+2) , $$
one concludes that
$$ \partial_{D}(R_n) = \frac{\sqrt{m}-\sqrt{d}}{2}\cdot \frac{\sqrt{m}+\sqrt{d}-2}{\sqrt{m}-\sqrt{d}-2} =-1 \in E(D)^\times/E(D)^{\times n} . $$
The other
residues are
$$  \partial_{D} (R_n)= \begin{cases}  -1 \in E(D)^\times/E(D)^{\times n}\ \ \ & \text{$D\in \{l_1^{-}, \ l_3^{-} \}$} \\
\frac{\sqrt{d}-\sqrt{m}}{2} \in E(D)^\times/E(D)^{\times n} \ \ \ & \text{$D\in \{l_4^{-}, \ l_6^{-}\} $}  \end{cases} $$
 by (\ref{div}) and
straightforward computations.
\end{proof}

\begin{lem} \label{br-nec} Let $K=k(\sqrt{m}, \sqrt{d}) \subset \bar{k}$.  Then
$$\Br(U_K)/\Br_1(U_K)\supset (\Bbb Z/n)  \ \ \ \ \text{if and only if} \ \ \  \mu_n \subset K \ \text{and} \ -1, \frac{\sqrt{d}-\sqrt{m}}{2} \in K^{\times n}  . $$
In this case, the element $R_n \in \Br(V)$ as defined in Theorem \ref{tr-br} belongs to $\Br(U_K) \subset \Br(V)$, is of order $n$, and generates  the  $n$-torsion subgroup of $\Br(U_K)/\Br_1(U_K) \subset \Br(\bar{U})$.
\end{lem}
\begin{proof}

Note that under the hypothesis $-1 \in K^{\times n}$, formula (\ref{obvious}) shows that the condition
$ \frac{\sqrt{d}-\sqrt{m}}{2} \in K^{\times n}  $ is independent of the choice of the square roots of $d$ and $m$
in $\bar{k}$.

If $\mu_n\subset K $ and $-1, (\sqrt{d}-\sqrt{m})/2\in K^{\times n}$, then $R_n\in \Br(U_K)$  by Proposition
\ref{nonclosedresidues} and it has
  image of order $n$ in $\Br({\overline U}) \simeq \Bbb Q/\Bbb Z$ by   Theorem \ref{tr-br}.
  This proves one implication.

\medskip

Let us  prove the converse statement.
Assume $(\Bbb Z/n)  \subset \Br(U_K)/\Br_1(U_K)$.
The isomorphism $\Br({\overline U}) \cong (\Bbb Q/ \Bbb Z)(-1)$ given by
Proposition \ref{br-tran} is Galois equivariant.
From  Lemma \ref{twist}, we then get $\mu_n\subset  K$.

Since
the lines
 $l_{i}^{-}$ in Lemma \ref{lines} are defined over $K \subset \bar{k}$ for $1\leq i\leq 6$, the open subset
 $$V=U_K\setminus \{ \bigcup_{i=1}^6  l_i^{-}\}$$ is
defined
 over $K$. It
 satisfies
 $\Pic(V_{\bar k})=0$ since $V_{\bar k} \cong \mathbb G_{m,\bar k}^2$.
 One has the following commutative
diagram
of exact sequences
 \begin{equation} \label{bd}
 \xymatrix{
0\ar[r] & \Br(K)=\Br_1(U_K) \ar[r]\ar[d] & \Br_1(V) \ar[r]^-{\partial_K} \ar[d] & \oplus_{i=1}^6  H^1(K, \Bbb Q/\Bbb Z)_{l_i^{-}}   \ar[d]^{=} \\
0 \ar[r] & \Br(U_K)  \ar[r]    & \Br(V) \ar[r]^-{\partial_K}  & \oplus_{i=1}^6 H^1(K, \Bbb Q/\Bbb Z)_{l_i^{-}}    }
 \end{equation}
 by \cite[Theorem 3.4.1, Remark 3.3.2]{CT}, \cite[Lemma 6.1]{Sansuc} and Theorem \ref{br} (which gives $\Br(K)=\Br_{1}(U_{K})$).
 From Prop. \ref{unit}
 we know that $\bar{k}^{\times}=\bar{k}[U]^{\times}$ and that $ \Pic({\overline U}) $ is a lattice.
From the exact sequence
 of lattices with trivial Galois action
$$ 1\rightarrow \bar{k}[V]^\times /\bar k^\times \xrightarrow{div} \oplus_{i=1}^6 \Bbb Z l_{i}^{-} \xrightarrow{\psi} \Pic({\overline U}) \rightarrow 1,$$
Galois cohomology
 gives the long exact sequence
$$  0=H^1(K, \Pic({\overline U})) \rightarrow H^2(K, \bar k[V]^\times/\bar k^\times) \xrightarrow{div} \oplus_{i=1}^6H^2(K, \Bbb Z)_{l_i^{-}} \rightarrow H^2(K, \Pic({\overline U})).$$
That $H^1(K, \Pic({\overline U}))=0$ follows from the fact that $\Pic({\overline U})$ is a lattice with trivial $\Gal(\bar{k}/K)$
action.
 The following diagram
$$\xymatrix{
H^2(K, \bar k[V]^\times) \ar[r]^{\simeq}\ar[d]_{div}  & \Br_1(V) \ar[d]^{\partial_K} \\
\oplus_{i=1}^6 H^2(K,  {\Bbb Z})_{    l_i^{-}    }   & \oplus_{i=1}^6 H^1(K, \Bbb Q/\Bbb Z)_{l_i^{-} } \ar[l]_-{\simeq} } $$ commutes up to sign by  \cite[Remark 3.3.2]{CT} and \cite[Lemma 2.1]{CTXu09}.

 Since $V$ has $K$-points, 
 the exact sequence 
$$ 1 \to    \bar{k}^\times \to  \bar{k}[V]^\times \to  \bar{k}[V]^\times/   \bar{k}^\times \to 1$$
splits as a sequence of Galois modules.
From identification  (\ref{Z4bis}) one gets   $$H^2(K, \Pic({\overline U})) \simeq \oplus_{i=1}^4 H^1(K, \Bbb Q/\Bbb Z)_{[l_i^{-}]}.$$
One then obtains the following  exact sequence
\begin{equation} \label{ext} 0\rightarrow \Br(K) \rightarrow \Br_1(V)  \xrightarrow{\partial_K}  \oplus_{i=1}^6 H^1(K, \Bbb Q/\Bbb Z)_{l_i^{-} } \xrightarrow{\phi} \oplus_{i=1}^4 H^1(K, \Bbb Q/\Bbb Z)_{[l_i^{-}]} \end{equation}
which extends the first line of (\ref{bd}). Here $\phi$ is induced by $\psi$. 
By   (\ref{Z4ter}), it   
is given on $(\chi_1,\cdots, \chi_6) \in \oplus_{i=1}^6 H^1(K, \Bbb Q/\Bbb Z)_{l_i^{-} }$
by the formula
 $$\phi(\chi_1,\cdots, \chi_6)=(\chi_1+\chi_5+\chi_6,\chi_2-\chi_5,\chi_3-\chi_6, \chi_4+\chi_5+\chi_6).$$

By Proposition \ref{nonclosedresidues},
one has
$$\partial_K(R_n)=(-1, -\frac{1}{2}(\sqrt{d}+\sqrt{m}),-1,\frac{\sqrt{d}-\sqrt{m}}{2} ,-1, \frac{\sqrt{d}-\sqrt{m}}{2})\in \oplus_{i=1}^6 K^\times/K^{\times n}.$$
 We now get:
$$ \phi(\partial_K(R_n))=(\frac{\sqrt{d}-\sqrt{m}}{2}, \frac{\sqrt{d}+\sqrt{m}}{2}, \frac{\sqrt{d}+\sqrt{m}}{2}, -(\frac{\sqrt{d}-\sqrt{m}}{2})^2) \in \oplus_{i=1}^4 K^\times/K^{\times n} .$$

By Theorem \ref{tr-br},   the class $R_n \in \Br(V)[n]$ is of order $n$, since it is  of order $n$ by going
over to $\bar k$.
By hypothesis, we have $  \mathbb Z/n \subset [\Br(U_K)/\Br_1(U_K)][n] \subset  \Br(\overline U)[n] \simeq \Bbb Z/n$.
The restriction map  $ \Br(\overline U)[n] \to \Br(V_{\bar{k}})[n]$ is an isomorphism, and the last
group is spanned by the class of $R_{n}$, which comes from $R_{n} \in \Br(V)$.
Thus there exists
$\mathcal B \in \Br(U_K)$ such that $R_n$ and $\mathcal B$  have the same image in $\Br(\overline U)$. Since $R_n, \mathcal B  $ are both contained in $ \Br(V)$, one concludes
$R_n-\mathcal B \in \Br_1(V)$.
 Then
$$\phi(\partial_K(R_n-\mathcal B))=\phi(\partial_K(R_n)) $$
$$=(\frac{\sqrt{d}-\sqrt{m}}{2}, \frac{\sqrt{d}+\sqrt{m}}{2}, \frac{\sqrt{d}+\sqrt{m}}{2}, -(\frac{\sqrt{d}-\sqrt{m}}{2})^2) \in \oplus_{i=1}^4 K^\times/K^{\times n}$$
is trivial.  This implies $-1 \text{ and } (\sqrt{d}-\sqrt{m})/2\in K^{\times n}$.
\end{proof}

 \begin{lem} \label{cores}
 Let $K=k(\sqrt{m}, \sqrt{d})$. Suppose  that $R_n=(f,g)_{\zeta_{n}}$ belongs to  $\Br(U_{K})$.
 Suppose $\mu_n\subset k$. Then the image of  $\mathcal B : = Cor_{K/k}(R_n) \in  \Br(U)$
  in $\Br(U)/\Br_1(U) \subset (\Bbb  Z/n)$ generates  a cyclic group  of order
  $n_1=n/gcd(n,[K:k])$.
 \end{lem}
\begin{proof} In $\Br(\overline U)$,
	one has $$Res_{k/\bar k}(\mathcal B)=Res_{k/\bar k}\circ Cor_{K/k}(R_n)=\sum_{\sigma}R_n^{\sigma},$$
	where $\sigma$ runs through the embeddings of $K$ into $\bar k$.
	Since $\mu_{n}\subset k$, one has $R_n^{\sigma}=R_n$. 
	Therefore $Res_{k/\bar k}(\mathcal B)=[K:k] \cdot R_n$ in $\Br(\overline U)$, and the proof is completed.
\end{proof}

\begin{lem} \label{br-odd} Let $K=k(\sqrt{m}, \sqrt{d})$. Suppose $\mu_n\subset k$.
Let $n_1=n/gcd(n,[K:k])$.

 1) Assume $-1  \in K^{\times n}$ and $(\sqrt{d}-\sqrt{m})/2\in K^{\times n}$. Then the element
   $\mathcal B : = Cor_{K/k}(R_n)$ belongs to $\Br(U)$ and generates
   the cyclic subgroup of order $n_{1}$ of
     $\Br(U)/\Br_1(U)$.

 2) Suppose $n$ is odd. Then
  $\Br(U)/\Br_1(U) \supset (\Bbb  Z/n)$ if and only if $(\sqrt{d}-\sqrt{m})/2\in K^{\times n}$.
 In that case, the element
   $\mathcal B : = Cor_{K/k}(R_n)$ belongs to $\Br(U)[n]$ and generates the cyclic subgroup
   of order $n$ of  $\Br(U)/\Br_1(U)$.
\end{lem}
\begin{proof}

1) Suppose  $-1 \text{ and } (\sqrt{d}-\sqrt{m})/2\in K^{\times n}$, then $R_n\in \Br(U_K)$ by the computation of residues in
Proposition \ref{nonclosedresidues}. By Lemma \ref{cores}, the image of
  $\mathcal B \in \Br(U)$ in  $\Br(U)/\Br_1(U)$ is cyclic of order $n_{1}$.

2) Suppose $n$ is odd. Then $n=n_1$ and $-1\in K^{\times n}$. The sufficiency follows from 1).
The converse follows from  $$\Bbb  Z/n \subset \Br(U)/\Br_1(U) \subset  \Br(U_K)/\Br_1(U_K)  \subset \Br({\overline U}).$$
and Lemma \ref{br-nec}.
\end{proof}

\begin{lem} \label{lem-sur} Let $F=k(\sqrt{d})$ and $G=\Gal(F/k)$. Then the natural map $\Br(U)\rightarrow \Br(U_F)^G$ is surjective.
\end{lem}

\begin{proof}
 We may assume that $F/k$ is of degree 2.
  We know  that $F^{\times}=H^0(U_{F}, \Bbb G_{m})$ by Proposition \ref{unit}.
 This implies
  $$ H^3(G,H^0(U_{F},\Bbb G_{m}))=H^3(G,F^{\times}) = H^1(G, F^{\times})=0$$
 by periodicity of the cohomology of cyclic groups and by Hilbert's theorem 90.
  The spectral sequence
  $$ E_2^{p,q} = H^p(G, H^q(U_F, \Bbb G_m)) \Rightarrow H^{p+q}(U, \Bbb G_m).$$
   then gives an exact sequence
   $$ \Br(U) \to \Br(U_{F})^G \to H^2(G, \Pic(U_{F})),$$
  which by periodicity of  the cohomology of cyclic groups for Tate cohomology groups  reads
  $$ \Br(U) \to \Br(U_{F})^G \to \hat{H}^0(G, \Pic(U_F)).$$

a) Suppose $F\neq k(\sqrt{m})$. Since $\overline k[U]^\times =\overline k^\times$, the natural map $\Pic(U_F) \hookrightarrow \Pic({\overline U})^{g_F}$ is injective (in fact, it is an isomorphism since $U(F) \neq \emptyset$).
This implies that $\Pic(U_F)^G \hookrightarrow \Pic({\overline U})^{g}$ is injective.
Since  $$\Pic({\overline U})^{g} = \Pic(U_K)^{\Gal(K/k)}=0$$ with $K=F(\sqrt{m})$ by (\ref{act}) in the proof of Theorem \ref{br},
one has $\Pic(U_F)^G=0$, hence $\hat H^0(G, \Pic(U_F)) =0$.
 
b) Suppose $F= k(\sqrt{m})$. Let $\rho$ be the generator of $G$. By the computation in Theorem \ref{br}  for the case $k(\sqrt{d})=k(\sqrt{m})\neq k$, the group
$\Pic(U_F)^{G}$ is generated by 
$$2[l_4] +[l_1] - [l_2]-[l_3]=(1+\rho)[l_4],$$
hence $\hat H^0(G, \Pic(U_F)) =0$.
\end{proof}

Let $K=k(\sqrt{d},\sqrt{m})$. Define  \begin{equation} \label{i} I=\{n\in \mathbb N: \mu_n \subset k \text { and } -1, \frac{\sqrt{d}-\sqrt{m}}{2} \in K^{\times n}\}. \end{equation}
If $p,q$ are coprime integers,  then $\mu_{pq} \subset k$ if and only if
$\mu_{p} \subset k$ and $\mu_{q} \subset k$. Similarly, for $p$ and $q$ coprime integers, and $\rho \in K^{\times}$,
one has $\rho \in K^{\times pq}$ if and only if $\rho \in K^{\times p}$ and $\rho \in  K^{\times q}$.
Going over to primary components, one concludes that if
 $p, q$ are integers in $I$, then the least common multiple $[p,q]$ of $p$ and $q$ is in $I$.
Therefore $I$ is a directed
set with respect to divisibility. The following theorem is the main result of this section.

\begin{thm} \label{br-gen} Let $K=k(\sqrt{d},\sqrt{m})$. Let
$$\label{i} I=\{n\in \mathbb N: \mu_n \subset k \text { and } -1, \frac{\sqrt{d}-\sqrt{m}}{2} \in K^{\times n}\}.$$
Then  $$\Br(U)/\Br_1(U)\cong \varinjlim_{n\in I} \mathbb Z/n .$$
In particular, if $I$ is  finite, for instance if $k$ is a number field, then $$\Br(U)/\Br_1(U)\cong \mathbb Z/N $$ where $N$ is the biggest integer in $I$.
\end{thm}
\begin{proof}
One has $\Br(U)/\Br_1(U)\subset \mathbb Q/\mathbb Z (-1)^g $ by Proposition \ref{br-tran}. Hence
$\Br(U)/\Br_1(U)$  is
a subgroup of the
abelian group $\mathbb Q/\mathbb Z$.
We thus only need to show:
\begin{equation} \label{basicequiv}
  {\mathbb Z/n} \subset  \Br(U)/\Br_1(U)   \ \ \ \text{ if and only if } \ \ \ n\in I
  \end{equation}
and we only need to show this for $n$ a power of a prime number.

Suppose $ \Br(U)/\Br_1(U) \supset \mathbb Z/n$. Then
 $\mu_n\subset k$ by Proposition \ref{br-tran}
 and Lemma \ref{twist}. We have
 $$\Br(U)/\Br_1(U) \subset  \Br(U_K)/\Br_1(U_K)  \subset \Br({\overline U}).$$
 Thus
$ \Bbb  Z/n  \subset  \Br(U)/\Br_1(U) $ implies $\Bbb  Z/n \subset \Br(U_K)/\Br_1(U_K)$.
Then $ n\in I$  follows from Lemma \ref{br-nec}.
This establishes one direction of the equivalence (\ref{basicequiv}).

Suppose $n \in I$ is an odd integer. Lemma \ref{br-odd} gives the reverse
direction in (\ref{basicequiv}) in a very precise form, namely
the image of the element $ Cor_{K/k}(R_n) \in \Br(U)[n]$  generates
the cyclic subgroup of order $n$ of
$ \Br(U)/\Br_1(U)$.

\bigskip

To complete the proof of the theorem, it is now enough to prove:
\begin{equation}\label{basicimplic}
n=2^s  \ \text{and} \ n\in I \Longrightarrow  \Br(U)/\Br_1(U) \supset \mathbb Z/n.
\end{equation}

Since $-1 \in K^{\times n}$, one concludes that $\mu_{2n}\subset K$.
Fix a primitive $2n$-th root of unity $\zeta_{2n}\in K$.
Essentially the same computations as in Proposition  \ref{nonclosedresidues} give:
\begin{equation} \label{res}  \partial_{D} (\frac{f}{g}, -\frac{u}{v})_{\zeta_{2n}}= \begin{cases} \frac{\sqrt{d}+\sqrt{m}}{2} \in K(D)^\times/K(D)^{\times 2n} \ \ \ & \text{$D=l_2^{-}$} \\
-1 \in K(D)^\times/K(D)^{\times 2n} \ \ \ & \text{$D=l_3^{-}$} \\
\frac{\sqrt{m}-\sqrt{d}}{2} \in K(D)^\times/K(D)^{\times 2n} \ \ \ & \text{$D=l_4^{-}$} \\
\frac{\sqrt{d}-\sqrt{m}}{2} \in K(D)^\times/K(D)^{\times 2n} \ \ \ & \text{$D=l_6^{-}$} \\
 1 \in K(D)^\times/K(D)^{\times 2n} \ \ \ & \text{$D\in \{l_1^{-}, l_5^{-}\}$}
  \end{cases} \end{equation}

\medskip

 Let $F=k(\sqrt{d})$. If $K/F$ is of degree 2,  let $\tau$ be the generator of $\Gal(K/F)$.
 If $F/k$ is of degree 2, let $\sigma$ denote the generator of $\Gal(F/k)$.
 We break up the discussion according to the structure of the field extension $K/k$.
 In each case, we shall produce an explicit element  $\mathcal B \in \Br(U_{F})$
 which is of order $n$ over the algebraic closure and
 which is invariant under $\Gal(F/k)$. Lemma \ref{lem-sur} will then ensure  that it comes
  from a  class in $\Br(U)$ whose image in $\Br(U)/\Br_{1}(U)$ is of order $n$.

\bigskip

$\bullet$
 Suppose $[K:k]=4$.
 Let $$\mathcal B=Cor_{K/F}(\frac{f}{g}, -\frac{u}{v})_{\zeta_{2n}}+Cor_{K/F}(\frac{u_1}{v_1}, \frac{\sqrt{d}-\sqrt{m}}{2})_{\zeta_{2n}} \in \Br(F(X))$$
where  $u_1=y-2t $ and $ v_1=x+\frac{1}{2} (\sqrt{d}-\sqrt{m})y- z+\sqrt{m}t$.
Since
$$\{u_1=0\} \cap X =L_2+l_2^{-}+l_2 $$$$  \{v_1=0\} \cap X  =l_6^{-}+\tau(l_4^{-})+l_2  $$ by  Bezout's theorem,  one obtains that
\begin{equation} \label{res-1} \partial_D (\frac{u_1}{v_1}, \frac{\sqrt{d}-\sqrt{m}}{2})_{\zeta_{2n}} =  \frac{\sqrt{d}-\sqrt{m}}{2} \in K(D)^\times/K(D)^{\times 2n}
  \end{equation}
for $D\in \{l_2^{-}, \tau(l_4^{-}), l_6^{-}\}$.
Since $(\sqrt{d}-\sqrt{m})/2 \in K^{\times n}$, we have
$$-1=N_{K/F}((\sqrt{d}-\sqrt{m})/2 ) \in F^{\times n}  \ \ \ \text{and} \ \ \ \mu_{2n}\subset F.  $$

When $D$ is defined over $F$, the corestriction map
$$ H^1(K(D), \Bbb Z/2n)=K(D)^\times/K(D)^{\times 2n}  \xrightarrow{Cor_{K/F}} H^1(F(D), \Bbb Z/2n) = F(D)^\times/F(D)^{\times 2n} $$
is given by norm. Since the residue maps commute with corestriction, the residues of $\mathcal B$ at $D\in \{l_i^{-}\}_{i=1}^3$ are trivial by  (\ref{res}) and (\ref{res-1}).

Suppose we have $D\in \{l_i^{-}\} $  with 
 $i \in \{4,5,6\}$. Then  $D$ is not defined over $F$.  
 One can identify $K(D)$ 
 with $F({\mathcal D})$
where $\mathcal D$ is the integral divisor on  $X_{F}$ which is the image of the divisor $D$ on $X_{L}$
via the projection map $X_{L} \to X _{F}$. We shall say that $\mathcal D$ is below $D$.
   Then $\tau$ induces an isomorphism from $K(\tau D)$ to $F(\mathcal D)$.

For $\mathcal D$ below $l_4^{-}$, one has
$$ \partial_{\mathcal D} ({\mathcal B} )=\frac{\sqrt{m}-\sqrt{d}}{2} \cdot (\frac{\sqrt{d}+\sqrt{m}}{2})^{-1} =(\frac{\sqrt{m}-\sqrt{d}}{2})^2 \in F(\mathcal D)^\times /F(\mathcal D)^{\times 2n} $$ by (\ref{res}), (\ref{res-1}) and the above identification. For $\mathcal D$ below $l_6^{-}$, one has
$$ \partial_{\mathcal D} ({\mathcal B} )=1  \in F(\mathcal D)^\times /F(\mathcal D)^{\times 2n} $$ by (\ref{res}), (\ref{res-1}) and the above identification.
Since $$\frac{\sqrt{d}-\sqrt{m}}{2} \in K^{\times n} \subset K(D)^{\times n} = F(\mathcal D)^{\times n} ,$$ the class $ \partial_{\mathcal D} ({\mathcal B} )$ is trivial in $H^1(F(\mathcal D), \Bbb Z/2n)$.
We thus get
\begin{equation}\label{nonramF}
 \mathcal B\in \Br(U_F).
\end{equation}
 Note that $\mu_{2n}\subset F$.   Then $\mathcal B$ is of order $n$ in $ \Br({\overline U})$ by Lemma \ref{cores} (replacing  $k$ by $F$).

Since we have $\mu_n\subset k$, Proposition \ref{br-tran} shows that
 the Galois group $\Gal(\overline k/k)$ acts trivially on the unique subgroup of order $n$ in $\Br({\overline U})$.
 This implies that $\mathcal B-\mathcal B^\sigma \in \Br_1(U_F)$, and $\Br_1(U_F)=\Br(F)$ by
 Theorem \ref{br}. Let $A=\mathcal B-\mathcal B^\sigma\in \Br(F)$.
 We shall prove that $A=0$, hence $\mathcal B=\mathcal B^\sigma$.
 
 \medskip
 
We need to distinguish two subcases.
\medskip

Subcase a).  Suppose $\mu_{2n}\subset k$. By evaluating $\mathcal B$ and $\mathcal B^\sigma$ at the special point $(-2,0,\sqrt{d})$ in $U(F)$, one obtains  
\begin{align*}
A=&Cor_{K/F}(\frac{-2\sqrt{d}(\sqrt{m}-\sqrt{d})}{-m+\sqrt{md}+2\sqrt{m}},\frac{-\sqrt{m}}{\sqrt{d}-2})_{\zeta_{2n}}-Cor_{K/F}(\frac{-2\sqrt{d}}{\sqrt{d}-\sqrt{m}+2},\frac{2}{\sqrt{m}-\sqrt{d}})_{\zeta_{2n}}\\
     &+Cor_{K/F}(\frac{2}{\sqrt{d}-\sqrt{m}+2},\frac{\sqrt{d}-\sqrt{m}}{2})_{\zeta_{2n}}-Cor_{K/F}(\frac{2}{\sqrt{d}-\sqrt{m}+2},\frac{-\sqrt{d}-\sqrt{m}}{2})_{\zeta_{2n}}
     \end{align*}
in $\Br(F)$.
Since $(\alpha, \beta)_{\zeta_{2n}}= (\alpha^{-1}, \beta^{-1})_{\zeta_{2n}}$ in $\Br (K)$ for $\alpha, \beta\in K^\times$, and  $((1-\alpha)^{-1}, \alpha)_{\zeta_{2n}} =0$ for any $\alpha\neq 0,1$ in $K$, one has
\begin{align*}  &  (\frac{-2\sqrt{d}(\sqrt{m}-\sqrt{d})}{-m+\sqrt{md}+2\sqrt{m}},\frac{-\sqrt{m}}{\sqrt{d}-2})_{\zeta_{2n}}=(-\frac{\sqrt{m}(\sqrt{m}+\sqrt{d}-2)}{4\sqrt{d}}, -\frac{\sqrt{d}-2}{\sqrt{m}})_{\zeta_{2n}}  \\
& = (-\frac{\sqrt{m}(\sqrt{m}+\sqrt{d}-2)}{4\sqrt{d}} \cdot (1+\frac{\sqrt{d}-2}{\sqrt{m}})^{-1}, -\frac{\sqrt{d}-2}{\sqrt{m}})_{\zeta_{2n}} = (-\frac{m}{4\sqrt{d}}, -\frac{\sqrt{d}-2}{\sqrt{m}})_{\zeta_{2n}} \end{align*}
 in $\Br(K)$. 

 Similarly, one has
\begin{align*} & (\frac{-2\sqrt{d}}{\sqrt{d}-\sqrt{m}+2},\frac{2}{\sqrt{m}-\sqrt{d}})_{\zeta_{2n}}= (\frac{\sqrt{d}-\sqrt{m}+2}{-2\sqrt{d}},\frac{\sqrt{m}-\sqrt{d}}{2})_{\zeta_{2n}} \\
& = (\frac{\sqrt{d}-\sqrt{m}+2}{-2\sqrt{d}}\cdot (1-\frac{\sqrt{m}-\sqrt{d}}{2})^{-1},\frac{\sqrt{m}-\sqrt{d}}{2})_{\zeta_{2n}}=(\frac{-1}{\sqrt{d}},\frac{\sqrt{m}-\sqrt{d}}{2})_{\zeta_{2n}} .
\end{align*}
Therefore
\begin{align*} A =&Cor_{K/F}(-\frac{m}{4\sqrt{d}}, -\frac{\sqrt{d}-2}{\sqrt{m}})_{\zeta_{2n}}-Cor_{K/F}(\frac{-1}{\sqrt{d}},\frac{\sqrt{m}-\sqrt{d}}{2})_{\zeta_{2n}} \\
& +Cor_{K/F}(\frac{2}{\sqrt{d}-\sqrt{m}+2},(\frac{\sqrt{d}-\sqrt{m}}{2})^2)_{\zeta_{2n}}\\
     =&(-\frac{m}{4\sqrt{d}}, \frac{m-4\sqrt{d}}{-m})_{\zeta_{2n}}+(-\frac{1}{\sqrt{d}},-1)_{\zeta_{2n}}.
\end{align*}
Since $(\alpha, -\alpha)_{\zeta_{2n}}=0$ in $\Br(F)$ for any $\alpha\in F^\times$, one has
\begin{align*}  & (-\frac{m}{4\sqrt{d}}, \frac{m-4\sqrt{d}}{-m})_{\zeta_{2n}}= (-\frac{m}{4\sqrt{d}}, \frac{m}{4\sqrt{d}}\cdot \frac{m-4\sqrt{d}}{-m})_{\zeta_{2n}}=(-\frac{m}{4\sqrt{d}}, 1-\frac{m}{4\sqrt{d}})_{\zeta_{2n}} \\
& = (-1, 1-\frac{m}{4\sqrt{d}})_{\zeta_{2n}} = (-1, \frac{(\sqrt{d}-2)^2}{-4\sqrt{d}})_{\zeta_{2n}} =(-1, \frac{1}{-4\sqrt{d}})_{\zeta_{2n}}=(-1, -\frac{1}{\sqrt{d}})_{\zeta_{2n}} .
\end{align*}
One concludes that $A=0$.

\medskip

Subcase b). Suppose $\mu_{2n} \not \subset k$. Since $\mu_{2n}\subset F$ and $[F:k]=2$, one has $F=k(\zeta_{2n})$. Note that $\mu_n\subset k$, one gets $\zeta_{2n}^\sigma=\zeta_{2n}^{1+n}$. Considering the action of Galois group on the cyclic algebra $(a,b)_{\zeta_{2n}}$ for $a,b \in K(U)^\times$, one has $$(a,b)_{\zeta_{2n}}^{\sigma} =(a^\sigma, b^\sigma)_{\zeta_{2n}^\sigma} . $$
Since the character given by $b^\sigma$ and $\zeta_{2n}^\sigma$ is the $(n+1)$-th power of the character given by $b^\sigma$ and $\zeta_{2n}$, one concludes
$$ (a^\sigma, b^\sigma)_{\zeta_{2n}^\sigma} =(n+1) (a^\sigma, b^\sigma)_{\zeta_{2n}} $$ in $\Br(K(U))$.

By evaluating $\mathcal B$ and $\mathcal B^\sigma$ at the special point $(-2,0,\sqrt{d})$ in $U(F)$, one concludes  
\begin{align*}
A=&Cor_{K/F}(\frac{-2\sqrt{d}(\sqrt{m}-\sqrt{d})}{-m+\sqrt{md}+2\sqrt{m}},\frac{-\sqrt{m}}{\sqrt{d}-2})_{\zeta_{2n}}-(1+n)Cor_{K/F}(\frac{-2\sqrt{d}}{\sqrt{d}-\sqrt{m}+2},\frac{2}{\sqrt{m}-\sqrt{d}})_{\zeta_{2n}}\\
     &+Cor_{K/F}(\frac{2}{\sqrt{d}-\sqrt{m}+2},\frac{\sqrt{d}-\sqrt{m}}{2})_{\zeta_{2n}}-(1+n)Cor_{K/F}(\frac{2}{\sqrt{d}-\sqrt{m}+2},\frac{-\sqrt{d}-\sqrt{m}}{2})_{\zeta_{2n}}
\end{align*}
in $\Br(F)$.
Since  $$\frac{2}{\sqrt{m}-\sqrt{d}}, \ \ \frac{-\sqrt{d}-\sqrt{m}}{2} \in K^{\times n} , $$ one obtains  
$$    n(\frac{-2\sqrt{d}}{\sqrt{d}-\sqrt{m}+2},\frac{2}{\sqrt{m}-\sqrt{d}})_{\zeta_{2n}}=n(\frac{2}{\sqrt{d}-\sqrt{m}+2},\frac{-\sqrt{d}-\sqrt{m}}{2})_{\zeta_{2n}} =0 $$ in $\Br(K)$. Therefore the computation in Subcase a) is still available and $A=0$.

We have thus proved $\mathcal B \in \Br(U_F)^G$. By Lemma \ref{lem-sur}, this implies that
  $\mathcal B$  is in the image of $\Br(U) \to \Br(U_{F})$.

\bigskip
$\bullet$
Suppose $m \in k^{\times 2}$ and $d \notin k^{\times 2}$.
Then $F=K$. Let $\mathcal B=R_n$ as in Theorem \ref{tr-br}. Then $\mathcal B\in \Br(U_F)$
by Lemma \ref{br-nec}. By Proposition \ref{br-tran}, we have $R_n^\sigma-R_n \in \Br_1(U_F)$.
By Theorem \ref{br}, we have $ \Br(F)=\Br_1(U_F)$. Thus 
$R_n^\sigma=R_n + A  \in \Br(F(U))$ with $A \in \Br(F)$.
 By evaluating $R_n$ and $R_n^\sigma$ at the special point $(-\sqrt{m},0,0)$, one concludes that $A=0$. Therefore $R_n\in \Br(U_F)^G$ and the result  again follows   from Lemma \ref{lem-sur}.

\bigskip
$\bullet$
Suppose $d \in k^{\times 2}$ and $m \notin k^{\times 2}$.
 Let $$\mathcal B=Cor_{K/k}(\frac{f}{g}, -\frac{u}{v})_{\zeta_{2n}}+Cor_{K/k}(\frac{u_1}{v_1}, \frac{\sqrt{d}-\sqrt{m}}{2})_{\zeta_{2n}} $$
where $$u_1=y-2 \ \ \ \text{ and } \ \ \ v_1=x+\frac{1}{2}(\sqrt{d}-\sqrt{m})y- z+\sqrt{m} . $$ The result follows from (\ref{nonramF}) and $F=k$.

\bigskip
$\bullet$
 Suppose $md \in  k^{\times 2}$ and $d \notin k^{\times 2}$.
Recall that  $n=2^s >1$. By the definition of $I$, one has $\frac{\sqrt{d}-\sqrt{m}}{2} =(\alpha+\beta \sqrt{d})^2$ 
 where $\alpha,\beta \in k^\times$. 
Therefore we have $\alpha^2+d\beta^2=0$. This implies $\sqrt{-d}\in k$. 
Therefore $F=k(\sqrt{d})=k(\sqrt{-1})\neq k$, hence $\sqrt{-1}\not \in k$, so $n=2$ by the definition of $I$.

 Let $\mathcal B=R_2$ in Theorem \ref{tr-br}. Then $\mathcal B\in \Br(U_F)$ by Lemma \ref{br-nec}.  
 Let  $\rho$ be the generator of $\Gal(F/k)$.
 By Proposition \ref{br-tran} and Theorem \ref{br}, there exists $$A\in \Br_1(U_F)=\Br(F) \ \ \text{ such that } \ \ R_2^\rho=R_2 + A .$$

 By evaluating $R_2$ and $R_2^\sigma$ at the special point $(-2,0,\sqrt{d})$ and a similar computation as in case $[K:k]=4$, one concludes
 \begin{align*}
A=&-(\frac{-2\sqrt{d}(\sqrt{m}-\sqrt{d})}{-m+\sqrt{md}+2\sqrt{m}},\frac{\sqrt{m}}{\sqrt{d}-2})_{-1}+(\frac{-2\sqrt{d}}{\sqrt{d}+\sqrt{m}+2},\frac{2}{\sqrt{m}+\sqrt{d}})_{-1}\\
=&-(\frac{-2\sqrt{d}(\sqrt{m}-\sqrt{d})}{-m+\sqrt{md}+2\sqrt{m}},\frac{-\sqrt{m}}{\sqrt{d}-2})_{-1}+0=-(-\frac{m}{4\sqrt{d}}, -\frac{\sqrt{d}-2}{\sqrt{m}})_{-1}\\
=&-(-\frac{m}{4\sqrt{d}}, \frac{(\sqrt{d}-2)^2}{m})_{\zeta_4}=-(-\frac{m}{4\sqrt{d}},\frac{m-4\sqrt{d}}{m})_{\zeta_4}= (-\frac{4\sqrt{d}}{m},1-\frac{4\sqrt{d}}{m})_{\zeta_4}\\
 =& (-1, 1-\frac{4\sqrt{d}}{m})_{\zeta_4}=(-1, \frac{(\sqrt{d}-2)^2}{m})_{\zeta_4}
\end{align*}
in $\Br(F)$,
where $\zeta_4$ is a primitive $4$-th root of unity. Note that $\sqrt{-1}, \sqrt{m} \in F$. Thus we have $A=0$.
Therefore $R_2\in \Br(U_F)^G$ and the result follows from Lemma \ref{lem-sur}.

\bigskip

$\bullet$
 The case $K=k$ follows from Lemma \ref{br-nec}.
\end{proof}

\begin{cor}\label{br-prac} Suppose that $k$ is a field with an ordering. Then
$\Br(U)/\Br_{1}(U) \subset \Bbb Z/2$.
If  $d$ is positive in that ordering, then
 $\Br_1(U)=\Br(U)$. \end{cor}
\begin{proof}  Let $n \in I$. By (the easy part of the proof of) Theorem \ref{br-gen}, we have $\mu_{n} \subset k$ and $-1 \in K^{\times 2}$.
 If $k$ can be ordered, this implies $n \in \{1,2\}$.
If  $d$ is positive with respect to an ordering, then $d$ and $m=d+4$ are both positive in the real closure $R$ of $k$ with respect to this ordering. There
 is an embedding $K \subset R$. Thus $-1$ is not a square in $K$.
This implies $I=\{1\}$.
\end{proof}

\begin{cor}\label{br-Q} Let $k$ be a field of characteristic zero. If  $-1 \notin k^{\times 2}$ and $-d \notin k^{\times 2}$,
then the quotient $\Br(U)/\Br_1(U)$ has no $2$-primary part.
If moreover $k$ admits an ordering then $\Br_1(U)=\Br(U)$.
\end{cor}
\begin{proof} The hypothesis is equivalent to $\sqrt{-1}\not \in k(\sqrt{d})$.
Suppose $2 \in I$. By (the easy part of the proof of) Theorem \ref{br-gen},
we then have $$\sqrt{-1} \in K^\times \ \ \ \text{and} \ \ \  \frac{\sqrt{d}-\sqrt{m}}{2} \in K^{\times 2}$$ with $K=k(\sqrt{m}, \sqrt{d})$.  Since $\sqrt{-1}\not \in k(\sqrt{d})$, one has $k(\sqrt{d})\neq K$ and $\sqrt{m} \not \in k(\sqrt{d})$. Therefore $$-1=N_{K/k (\sqrt{d})}(\frac{\sqrt{d}-\sqrt{m}}{2}) \in k(\sqrt{d})^{\times 2}$$ which contradicts $-d \notin k^{\times 2}$.
\end{proof}

\begin{rem}
In the case $k=\Bbb Q$,  we find that $\Br_{1}(U)=\Br(U)$ as soon as $-d  \notin \Bbb Q^{\times 2}$.
\end{rem}
\begin{rem}
Suppose $-1 \notin k^{\times 2}$.
There exist $\gamma,\delta  \in k^{\times}$ be such that $\gamma^2+\delta ^2=1$ and $\gamma \neq \pm \delta $.
Set $u=4\gamma\delta $ and $v=2(\delta ^2-\gamma^2)$. Then $u^2+v^2=4$. Let $d=-u^2$ and $m=4-u^2=v^2$.
Fix $i:=\sqrt{-1} \in \bar k$.
Then $K=k(\sqrt{d},\sqrt{m})=k(i)$ is of degree 2 over $k$, contains $\sqrt{-1}$ and we have:
$$(\sqrt{d}-\sqrt{m})/2= (ui-v)/2 = \gamma^2-\delta ^2+2 \gamma \delta i= (\gamma+\delta  i)^2 \in K^{\times 2}.$$
For $U=U_{m}$, the hard part of the proof of Theorem \ref{br-gen} then gives $\Bbb Z/2 \subset \Br(U)/\Br_{1}(U)$.
If $k=\Bbb Q$,  
it then gives  $\Br(U)/\Br_{1}(U)=\Bbb Z/2$.
\end{rem}

\begin{rem} Suppose $m \in k^{\times 2}$
 and $d \notin k^{\times 2}$, so that
 $K=k(\sqrt{d})\neq k $.   Suppose $n \in I$ is a power of $2$. If $n=2$, assume $\mu_{4} \subset k$.
Then we can write down an explicit element   in $\Br(U)$ whose image generates
the cyclic subgroup of order $n$ of $\Br(U)/\Br_{1}(U)$.

Indeed, by assumption we have  $\mu_n \subset  k$ and $-1,\alpha\in K^{\times n}$ where $\alpha=(\sqrt{d}-\sqrt{m})/2$.
Let $$\chi_1 \in H^1(\Gal(k(\mu_{4n})/k),\mathbb Q/\mathbb Z) \ \ \ \text{and} \ \ \ \chi_2\in H^1(\Gal(k(\sqrt{d}, \sqrt[2n]{\alpha})/k),\mathbb Q/\mathbb Z)$$ be such that the restrictions of $\chi_1$ and $\chi_2$ to
$$\Gal(K(\mu_{4n})/K) \ \ \ \text{and} \ \ \  \Gal(k(\sqrt{d}, \sqrt[2n]{\alpha})/k(\sqrt{d}))$$ are  respective generators of these groups. Then the element
$$\mathcal B=Cor_{K/k}(\frac{f}{g}, \frac{u}{v})_{\zeta_{2n}}+((x-2)(y-\sqrt{m})(z-2), \chi_1) +((x-\sqrt{m})(y-2)(z-\sqrt{m}),\chi_2)$$
 is in $\Br(U)[2n]$, where $\zeta_{2n}$ is a primitive $2n$-th root of unity. Under the assumption  $\mu_4 \subset k$ if $n=2$,
 the image of $\mathcal B$ is of order $n$ in $\Br({\overline U})$.
\end{rem}

\section{Failure of the integral Hasse principle}

In this section, we explain that all examples which do not satisfy the Hasse principle in \cite{GS}
can be accounted for by integral Brauer-Manin obstruction or by the combination of integral Brauer-Manin obstruction with the reduction theory.

Given a scheme $\mathcal{U}$ over $\Bbb Z$, and $U:=\mathcal{U}\times_{\Bbb Z}\Bbb Q$,
we let $\mathcal{U}(A_{\Bbb Z})
= \prod_{p}
\mathcal{U}(\Bbb Z_{p}),$ where $p$ runs through all primes
and $\infty$, and $\Bbb Z_{\infty}={\Bbb R}$. We let
$${\mathcal{U}(A_{\Bbb Z}  )}_{\bullet} =
 \prod_{p < \infty}
  \mathcal{U}(\Bbb Z_{p})
 \times \pi_{0}(U(\Bbb R))$$
where $ \pi_{0}(U(\Bbb R))$ is the set of connected components of $U(\Bbb R)$.
We have the Brauer-Manin pairing
$${\mathcal{U}(A_{\Bbb Z}  )}_{\bullet} \times \Br(U) \to {\Bbb Q/\Bbb Z}.$$
The (reduced) Brauer-Manin set is the left kernel of this pairing.
Note that the Legendre symbol takes values in $\pm 1$ but the Hilbert symbols used below
take values $0$ or $1/2$ in $\Bbb Q/\Bbb Z$.

\subsection{Integral Brauer-Manin obstructions}

Let  $m \neq 0, 4$ be an integer and $d=m-4$.
 Let $ {\mathcal U}_{m}$ be the scheme over $\Bbb Z$ defined  by  equation (\ref{markoff1})
and $U_{m} ={\mathcal U}_{m} \times_{\Bbb Z} \Bbb Q$.

\begin{lem} \label{gen}    If $p$ is an odd prime with $(p, d)=1$, then each element in the following set
$$\{ (x\pm 2, d), \ (y\pm 2, d),   \ (z\pm 2, d)   \} \subset \Br(U_{m})$$   vanishes over ${\mathcal U}_{m}(\Bbb Z_p)$ and $(x^2-4,d)=(y^2-4,d)=(z^2-4,d)$ vanishes over $U_{m}(\Bbb Q_p)$.
If $d>0$, these elements vanish over $U_{m}(\Bbb R)$.
\end{lem}

\begin{proof}  One only needs to consider the case that $(\frac{d}{p})=-1$. Since (\ref{markoff1}) is equivalent to (\ref{markoff2}),
  over $\Bbb Z$,
one concludes that
$$ord_p(x_p^2-4)=ord_p(y_p^2-4)=0$$ for all $M_p=(x_p,y_p,z_p)\in \mathcal U_{m}(\Bbb Z_p)$. By symmetry, one  further obtains $$ord_p(x_p^2-4)=ord_p(y_p^2-4) =ord_p(z_p^2-4)=0$$ for all $M_p=(x_p,y_p,z_p)\in \mathcal U_{m}(\Bbb Z_p)$. This implies that the elements $(x\pm 2, d), (y\pm 2, d),  (z\pm 2, d)$ vanish over $\mathcal U_{m}(\Bbb Z_p)$.

If $(x_p,y_p,z_p)\in U_{m}(\Bbb Q_p)\setminus \mathcal U_{m}(\Bbb Z_p)$, one of $x_p, y_p, z_p\in \Bbb Q_p\setminus \Bbb Z_p$. Without loss of generality, we assume that $x_p\in \Bbb Q_p\setminus \Bbb Z_p$. Then $ord_p(x_p^2-4)$ is even and $(x_p^2-4,d)_p=0$. The result follows.
\end{proof}

\begin{lem} \label{infty} If $m<0$, then $|x|>2, |y|>2, |z|>2$ for any $(x,y,z)\in \mathcal U_m(\mathbb R)$.
\end{lem}
\begin{proof} Let $(x,y,z)\in \mathcal U_m(\mathbb R)$.   
 Suppose $|x|\leq 2$. Then $$m=(y-xz/2)^2+(1-x^2/4)z^2 + x^2 \geq 0 $$ which contradicts $m<0$. So $|x|>2$. Similarly $|y|>2,|z|>2$.
\end{proof}

\begin{rem} Let $f: U_m \rightarrow \mathbb A^2$ be the morphism defined by projecting $(x,y,z)$ to $(x,y)$. Therefore the image of $U_m(\mathbb R)$ by $f$ is the subset   $$W:=\{(x,y)\in \mathbb R^2: (x^2-4)(y^2-4)+4(m-4)\geq 0\} \subset \mathbb R^2.$$
The connected components of $U_m(\mathbb R)$ are just the preimages of connected components of  $W$ by $f$. The four lines $x=\pm 2$ and $y=\pm 2$ divide the plane $\mathbb R^2$ into nine parts.
Considering the signature of $(x^2-4)(y^2-4)$ on the nine parts, we have
$$\#\pi_{0}(U_m(\mathbb R))=\#\pi_{0}(W)=\begin{cases}
1 \ \ \ & \text{ if } m\geq 4\\
5 \ \ \ & \text{ if } 0\leq m <4\\
4 \ \ \ & \text{ if } m<0. 
\end{cases}$$
All connected components of $U_m(\mathbb R)$ are unbounded except the connected component defined by $|x|,|y|<2$ when $0\leq m <4$, and the bounded connected component becomes a single  point $(0,0,0)$ when $m=0$.
If $m<4$, $\Gamma$ permutes the
  four  unbounded components  transitively. 
  Full details are given in section \ref{real}.
\end{rem}

Let 
$ \mathcal B_1=(x-2,d)$, $ \mathcal B_2=(y-2,d)$, $  \mathcal B_3=(z-2,d) $ in $ \Br_{1}(U_{m}).$
By Theorem \ref{br}, for $m$ not a square, these  three elements  generate $\Br_1(U_m)/\Br_0(U_m)$.
 Let $B=(\mathcal B_1, \mathcal B_2,  \mathcal B_3) $. One can define the evaluation of $B$ over  $\mathcal U_m(\mathbb Z_p)$ by
 $$ B(M_p)= (\mathcal B_1(M_p), \mathcal B_2(M_p), \mathcal B_3(M_p) ) \in (\Bbb Q/\Bbb Z)^3$$ for $M_p\in \mathcal U_m(\Bbb Z_p)$ and
 $$B(\mathcal U_m(\Bbb Z_p))= \{ B(M_p): \ M_p \in \mathcal U_m(\Bbb Z_p) \}\subset (\Bbb Q/\Bbb Z)^3 $$ for $p\leq \infty$. By the symmetry of the  
 coordinates of (\ref{markoff1}),
the symmetric group $S_3$ acts on $B(\mathcal U_m(\Bbb Z_p))$ by coordinate permutation.

\begin{lem}\label{value-2} If $m\equiv 1 \mod 8$, then  $$B(\mathcal U_m(\Bbb Z_2))=  \{(1/2,1/2,0),(1/2,0,1/2),(0,1/2,1/2)\}.$$
\end{lem}
\begin{proof} Since $m\equiv 1 \mod 8$, one obtains that $d\equiv 5 \mod 8$ and 
 by (\ref{markoff1})
there is one and only one coordinate of any point in $\mathcal U_m(\mathbb Z_2)$ belonging to $\Bbb Z_2^\times$.

 The remaining two coordinates belong to $4\Bbb Z_2$ 
  by (\ref{markoff1}).
 The result follows from the straightforward computation of the Hilbert symbols and the symmetry of the coordinates.
\end{proof}

\begin{lem} \label{value-35} If $p=3$ or $p=5$ and
 $\ord_p(d)$ is odd, then
$$ B(\mathcal U_m(\mathbb Z_p)) = \begin{cases}   \{(1/2,0,0),(0,1/2,0),(0,0,1/2)\} \ \ \ & \text{for $p=3$ and $\ord_3(d)=1$} \\
 (\frac{1}{2}\Bbb Z/\Bbb Z)^3 \ \ \ & \text{for $p=3$ and $\ord_3(d)\geq 3$} \\
 (\frac{1}{2}\Bbb Z/\Bbb Z)^3\setminus (0,0,0) \ \ \ & \text{for $p=5$ and $\ord_5(d)=1$} \\
  (\frac{1}{2}\Bbb Z/\Bbb Z)^3 \ \ \ & \text{for $p=5$ and $\ord_5(d)\geq 3$.}
\end{cases} $$
\end{lem}
\begin{proof} 
$\bullet$   
Assume $p=3$ and $ord_3(d)=1$. Since
  (\ref{markoff1})
 is equivalent to
equation (\ref{markoff2})
and its variants by  coordinate permutations, any point in $\mathcal U(\mathbb Z_3)$ must have two coordinates in $3\Bbb Z_3$ and  the remaining coordinate  in $\Bbb Z_3^\times$  
by (\ref{markoff2}).
Without loss of generality, we assume $x, y\in 3\Bbb Z_3$ and $z\in \Bbb Z_3^\times$.
Therefore $$(x-2,d)_3=(y-2,d)_3=0 \text{ and } (x+2,d)_3=1/2.$$
By (\ref{relation-br}), one has $(z-2, d)_3=1/2$, hence $B((x,y,z))=(0,0,1/2)$.
The result follows by permutation of the coordinates.

$\bullet$
Assume $p=3$ and $ord_3(d)\geq 3$. Let $d=3^{2n+1}d_0$ with $d_0\in \Bbb Z_3^\times$ and $n\geq 1$.

By Hensel's lemma, there is $\xi\in \Bbb Z_3^\times$ such that $$4\xi+ 3^{2n+1} \xi^2=d_0 .$$ This implies:
$$(3^{2n+1}\xi, d)_3=(3\xi, d)_3=(3d_0,d)_3=(3d_0,3d_0)_3=(-1, 3d_0)_3=1/2.$$  
Then for $M_3= (0,0, 2+3^{2n+1} \xi)\in \mathcal U_m(\Bbb Z_3)$
we have $B(M_3)= (0, 0, 1/2)$.

By Hensel's lemma, for any $a\in \Bbb Z_3^\times$, there is $\xi\in \Bbb Z_3^\times$ such that $$\xi^2-(4a+3a^2)  \xi =3^{2n-1} d_0.$$
This implies:  $$\xi\in a(\Bbb Z_3^\times)^2 \ \ \  \text{and} \ \ \ \ (3\xi, d)_3=(3a, d)_3=(-ad_0, 3d_0)_3 .$$
Take $$M_3=(2+3\xi, 2+3a, 2+3a)\in \mathcal U_m(\Bbb Z_3) . $$ Then
$$ B(M_3)= \begin{cases} (0, 0, 0) \ \ \ & \text{if $ad_0\in 2 + 3\Bbb Z_3$} \\
(1/2, 1/2,1/2) \ \ \ & \text{if $ad_0 \in 1+ 3\Bbb Z_3$ .} \end{cases} $$

Since there is $\xi\in \Bbb Z_3^\times$ such that
 $$ \xi^2+d_0(4-3d_0)\xi=3^{2n-1} d_0$$ by Hensel's lemma, one obtains: $$-\xi\in d_0(\Bbb Z_3^\times)^2 \ \ \ \text{and} \ \ \  (3\xi, d)_3=(-3d_0, 3d_0)_3=0 . $$
Then
 $$ M_3=(-2+3d_0,-2+3d_0, 2+3\xi)\in \mathcal U_m(\mathbb Z_3) \ \ \ \text{and} \ \ \ B(M_3)=(1/2, 1/2, 0) . $$
The result follows by permutation of the coordinates.

$\bullet$ 
Assume  $p=5$ and $\ord_5(d)=1$. One can use the lifting of smooth points of  $\mathcal U_m(\mathbb Z/5)$ as in \cite[Proposition 5.7]{LM}
to show that $B$ can take all possible values over $\mathcal U_m(\Bbb Z_5)$ except $(0,0,0)$. We prove  $(0,0,0)\not\in B(\mathcal U_m(\Bbb Z_5))$.

By (\ref{markoff2}), there is
 at most one coordinate of  a point in $\mathcal U_m(\Bbb Z_5)$ which is congruent to $3 \mod 5$.
 If that is the case, the sum of  the two remaining coordinates is congruent  to $0 \mod 5$  as one sees by reducing (\ref{markoff1}) over $\Bbb Z/5$.
 By inspecting cases, one sees that
  $B$ cannot take  the value $(0,0,0)$ over such points.

By (\ref{markoff2}), there is at most one coordinate of a point in $\mathcal U_m(\Bbb Z_5)$ which  is congruent to $2 \mod 5$.  If that is the case,  both remaining coordinates are congruent to  $1$ or $4  \mod 5$ simultaneously as one sees by reducing 
(\ref{markoff1})
over $\Bbb Z/5$. One only needs to show that  $B$ cannot take the value $(0,0,0)$ when both remaining coordinates are congruent to  $1 \mod 5$. Without loss of generality, we assume that $(x_5,y_5,z_5)\in \mathcal U_m(\Bbb Z_5)$ satisfies $x_5\equiv y_5\equiv 1 \mod 5$ and $z_5\equiv 2 \mod 5$. Since $(x_5-2, d)_5=(y_5-2, d)_5=0 $, one obtains that $(z_5+2, d)_5=0$ by (\ref{markoff3}). By Proposition \ref{x}, one has
$$(x_5^2-4, d)_5=(y_5^2-4, d)_5=(z_5^2-4, d)_5=1/2 .$$ This implies   $(z_5-2, d)_5=1/2$.

The only remaining  possibility which one needs to consider is that all coordinates of the points in $\mathcal U_m(\Bbb Z_5)$ are congruent to $1 \mod 5$. This is impossible as one sees by reducing (\ref{markoff1}) over $\Bbb Z/5$.

$\bullet$  
Assume $p=5$ and $ord_5(d)\geq 3$. One only needs to show $(0,0,0)\in B(\mathcal U_m(\Bbb Z_5))$. Let $d=5^{2n+1}d_0$ with $(d_0, 5)=1$ and $n\geq 1$.  There is $\xi\in \Bbb Z_5^\times$ such that $$ \xi^2+d_0(4-5d_0)\xi = 5^{2n-1} d_0 $$ by Hensel's lemma. This implies that $\xi\equiv -d_0 \mod 5$ and $(5\xi, d)_5=(-5d_0, 5d_0)_5=0$.
Then  $$M_5=(2+5\xi,-2+5d_0,-2+5d_0)\in \mathcal U_m(\mathbb Z_5) \ \ \ \text{and} \ \ \ B(M_5)=(0, 0, 0) $$ as required.
\end{proof}

\medskip

The following Proposition
 extends  \cite[Prop.8.1(i) and Prop. 8.2]{GS}, propositions which only involve elements in $\Br(X)$.

\begin{prop}  \label{8.1-8.2} Let $\mathcal U$ be the scheme over $\Bbb Z$ given by
\begin{equation} \label{p-eq}  x^2+y^2+z^2 -xyz=4+r v^2 \end{equation}
where $r \in \Bbb Z$ is one of $2, -2, -3, 12, -12$ and all prime factors of $v$ are congruent to
$$  \begin{cases}    \pm 1  \mod 8 \  \  & \text{when} \ r=2 \\
  \pm 1  \mod 12 \   \text{ and } \ v^2  \equiv 25  \mod 32 \ & \text{when} \  r=12   \\
  1 \  \  \text{or} \  \  3 \mod 8 \ \ & \text{when} \ r=-2 \\
 1 \mod 3  & \text{when} \ r=-3 \\
 1 \mod 3  & \text{when} \ r=-12
  \end{cases}
$$
and $v\neq \pm 1$ when $r=-2,-3$.
Let
$$B=(x^2-4,  r)=(y^2-4,  r)=(z^2-4,  r)\in \Br_1(U) $$ with $U=\mathcal U\times_\Bbb Z \Bbb Q$.
Then
$$ \mathcal U(A_\mathbb Z)^{B} =\emptyset . $$
\end{prop}

\begin{proof}
When $r=\pm 2$,  for any $M_2=(x_2, y_2,z_2)\in \mathcal U(\Bbb Z_2)$, one of $x_2, y_2, z_2$ is a unit of $\Bbb Z_2$ by (\ref{p-eq}). For example, 
if $x_2$ is a unit, then
$$ x_2^2-4\equiv 5 \mod 8 \ \ \  \text{and} \ \ \  \ (x_2^2-4, \pm 2)_2= 1/2 .$$
Under the assumption $v\neq \pm 1$ when $r=-2$,  by Lemma \ref{infty}, $(x_\infty^2-4, \pm 2)_\infty=0$. 
For $M_p\in \mathcal U(\mathbb Z_p)$, one has
$$
B(M_p)=\begin{cases} 1/2 &\text{ if } p=2,\\
0 &\text{ otherwise }
\end{cases}
$$
by Lemma \ref{gen} and the given condition for $v$.
This implies $$\sum_{p\leq \infty}B(M_p)=1/2\neq 0,$$
 hence
$$ \mathcal U(A_\mathbb Z)^{B} =\emptyset.$$

\medskip

Suppose $r=-3,\pm 12$.
For any local solution $M_3=(x_3,y_3,z_3)\in \mathcal U(\Bbb Z_3)$, there is at least one coordinate of $M_3$ belonging to $3\Bbb Z_3$.  Otherwise, suppose $x_3$ and $y_3$ are in $\Bbb Z_3^\times$. Then $(x_3^2-4)(y_3^2-4)\in 9\Bbb Z_3$. A contradiction is derived by (\ref{p-eq}).
Since $(\alpha^2-4, r)_3= 1/2$ for  $\alpha\in 3\Bbb Z_3$,  one concludes that $B(M_3)=1/2$.

When $r=12$,  then $B=(x^2-4,  3)=(y^2-4,  3)=(z^2-4,  3)$. Since
$ (\frac{3}{p})=(-1)^{\frac{1}{2}(p-1)} (\frac{p}{3}) =1 $ for any $p\equiv \pm 1 \mod 12$ by the quadratic reciprocity law, by Lemma \ref{gen},  one only needs to consider $p=2$.  Similarly, for $r=-3,-12$, since $(\frac{-3}{p})=(\frac{p}{3})=1$ for $p\equiv 1 \mod 3$,  by Lemma \ref{infty} one reduces to the computation for $p=2$.

We claim that for any local solution $M_2=(x_2,y_2,z_2)\in \mathcal U(\Bbb Z_2)$,  there is at least one coordinate of $M_2$ in $\Bbb Z_2^\times$ when for $r=-3,\pm 12$. This is clear for $r=-3$ since $v$ is odd. Suppose $r=\pm 12$,  otherwise,  we can write $x_2=2\xi, y_2=2\eta$ and $z_2=2\delta$ with $\xi, \eta, \delta\in \Bbb Z_2$ and obtain the following equation
\begin{equation} \label{equ1} ( \xi^2-1)(\eta^2-1)=(\delta-\xi \eta)^2 - rv^2/4  \end{equation}  by (\ref{p-eq}). Since $\pm 3\not\in \Bbb Z_2^{\times 2}$, one concludes that $\xi$ and $\eta$ are in $2\Bbb Z_2$ by (\ref{equ1}). Similarly, $\delta \in 2\Bbb Z_2$.

Suppose $r=-12$. The left hand side of (\ref{equ1})  is $\equiv 1 \mod 4$, but the right hand side is $\equiv 3 \mod 4$, which is impossible. So there is at least one coordinate of $M_2$ in $\Bbb Z_2^\times$.

Suppose $r=12$. Write $\xi=2\xi_1, \eta=2\eta_1$ and $\delta=2\delta_1$ with $\xi_1, \eta_1, \delta_1\in \Bbb Z_2$.  One obtains that
\begin{equation} \label{equ2}  (4\xi_1^2-1)(4\eta_1^2-1) = 4 (\delta_1-2\xi_1\eta_1)^2 -3 v^2 . \end{equation}

If all $\xi_1, \eta_1$ and $\delta_1$ are in $2\Bbb Z_2$, then $-3\in \Bbb Z_2^{\times 2}$ by (\ref{equ2}), which is impossible.

If two of $\{ \xi_1, \eta_1, \delta_1\}$ are in $2\Bbb Z_2$ and the remaining one is in $\Bbb Z_2^\times$, we can write $$\xi_1=2a, \ \ \ \eta_1=2b \ \ \text{ with }  \ \ a, b  \in \Bbb Z_2$$ and $\delta_1\in \Bbb Z_2^\times$ by symmetry. Then
$$ 4-3v^2 \equiv (16 a^2-1)(16 b^2-1) \equiv \begin{cases}  1 \mod 32 \  & \text{when $a\in 2\Bbb Z_2$, $b\in 2\Bbb Z_2$} \\
-15  \mod 32 \  & \text{when } ab \in 2\Bbb Z_2 \\
15^2  \mod 32 \  & \text{when }  ab\in \Bbb Z_2^\times \end{cases} $$ by (\ref{equ2}).
This implies  
$$ v^2\equiv \begin{cases} 1 \mod 32 \  & \text{when $a\in 2\Bbb Z_2$, $b\in 2\Bbb Z_2$} \\
17 \mod 32 \ & \text{when $ab\in 2\Bbb Z_2$} \\
1 \mod 32 \ & \text{when $ab\in \Bbb Z_2^\times$} \end{cases} $$
which contradicts  the assumption on $v$.

If two of $\{ \xi_1, \eta_1, \delta_1\}$ are in $\Bbb Z_2^\times$ and the remaining one is in $2\Bbb Z_2$, we can assume $\delta_1\in 2\Bbb Z_2$ and $\xi_1, \eta_1\in \Bbb Z_2^\times$ by symmetry. This implies that $-3\in (\Bbb Z_2^\times)^2$ by (\ref{equ2}), which is impossible.

If all $\xi_1, \eta_1$ and $\delta_1$ are in $\Bbb Z_2^\times$, then
$ 3\cdot 3 \equiv 4- 3v^2 \mod 32 $ by (\ref{equ2}). Therefore $v^2\equiv 9 \mod 32$ which contradicts   the assumption on $v$.

Therefore the above claim follows, i.e., there is at least one coordinate of $M_2$ in $\Bbb Z_2^\times$.
Since $(\alpha_2^2-4,  \pm 3)_2=(-3,\pm 3)_2=0$ for $\alpha_2\in \Bbb Z_2^\times$,
one concludes that $B$ vanishes over $\mathcal U(\Bbb Z_2)$. For $M_p\in \mathcal U(\mathbb Z_p)$, one has
$$
B(M_p)=\begin{cases} 1/2 &\text{ if } p=3,\\
0 &\text{ otherwise. }
\end{cases}
$$
This implies
$$\sum_{p\leq \infty}B(M_p)=1/2\neq 0,$$
 hence $ \mathcal U(A_\mathbb Z)^{B} =\emptyset  $.
\end{proof}

\begin{rem}
	The element $B=(x^2-4,r) \in \Br(U)$ actually belongs to $\Br(X)$.
	Let $S$ be the finite set of primes which divide  $2d= 2rv^2$.
	For a prime $p \notin S$,  the element $B$
	vanishes not only on ${\mathcal U}(\Bbb Z_{p})$ but also on ${\mathcal U}(\Bbb Q_{p})$
	(Lemma \ref{gen}).  
	From $m>4$  and $m<0$  
	we get that $B$ vanishes on $U(\Bbb R)$
	(Lemma \ref{gen} and Lemma \ref{infty}).
	The above proof then shows that
	$$[\prod_{p \in S} \mathcal U(\Bbb Z_{p}) \times \prod_{p \notin S} U(\Bbb Q_{p})]^B$$ is empty.
	In particular, assuming there are $\Bbb Q_{p}$-points everywhere locally, we get that $U(\Bbb Q)$
	does not meet the open subset of
	$\prod_{p \in S} \mathcal U(\Bbb Z_{p}) $ which is orthogonal to the element $B$.
	This represents a lack of weak approximation -- which is a stronger result than
	the same statement for $\mathcal U(\Bbb Z)$.  
	
	On the other hand, for $m\neq 0,4$, it is a special case of a theorem of Salberger and Skorobogatov  \cite{SS} that the smooth cubic surface given by
	$t(x^2+y^2+z^2)-xyz= mt^3$  satisfies weak approximation with Brauer--Manin 
	obstruction.
	
\end{rem}

\begin{rem}  There is an error in the proof of \cite[Proposition 8.1 (i)]{GS}.  A contradiction is derived from the fact that $q\equiv \pm 5 \mod 8$ and $\{\pm 2\}$ is a quadratic residue modulo $q$. However, when $q\equiv 3 \mod 8$, $-2$ is a quadratic residue modulo $q$ and this is not a contradiction. The corresponding result should be modified. Moreover, the additional requirement that $v\in \{ 0,\pm 3, \pm 4 \} \mod 9$ can be replaced by the local condition in \cite[Proposition 6.1]{GS}.
\end{rem}

Proposition 8.3 in \cite{GS} can be improved as follows.

\begin{prop}  Let $v$ be an integer all prime factors of which are  congruent to $\pm 1 \mod 5$.
Let $\mathcal U$ be the scheme over $\Bbb Z$ given by the equation
$$ x^2+y^2+z^2 -xyz=m=4+20 v^2$$
and let $U=\mathcal U\times_{\Bbb Z} \Bbb Q$. 
Then
$ \mathcal U(A_\mathbb Z)^{\Br_1(U)} =\emptyset$.  \end{prop}

The smallest positive  such $v$ is $v=11$, which gives $m=4+20 v^2=2424$.

\begin{proof} We only consider the following subset  $A$ of $\Br_1(U)$
$$ \{(x \pm 2,  5), (y \pm 2,  5), (z \pm 2,  5)\}. $$
Then each element  $\beta \in A$  vanishes  over $\mathcal U(\Bbb Z_p)$ for $p\neq 2, 5$ by Lemma \ref{gen} and  the property $(\frac{5}{p})=(\frac{p}{5})=1$ for $p\equiv \pm 1 \mod 5$.

\medskip

 Let $M_5=(x_5,y_5,z_5)\in \mathcal U(\Bbb Z_5)$.
 By permutation of the coordinates and reduction of the equation
 $$ (x^2-4)(y^2-4)=(2z-xy)^2 -80 v^2$$
  modulo 25, one sees that there is
 at most one coordinate of $M_5$ which is congruent $\pm 2 \mod 5$.

We consider
$$ V= (x_5^2-4, 5)_5=(y_5^2-4, 5)_5=(z_5^2-4, 5)_5.$$

We  have two possibilities.

a5) At least one of the coordinates is $\pm 1 \mod 5$, then
$V=1/2$. Therefore
half of the elements in $A$  vanish at $M_5$ and the other half do not vanish.

b5)  Two coordinates of $M_5$ are in $5\Bbb Z_5$ and the remaining one is $\pm 2 \mod 5$.
In this case, $V=0$. Without loss of generality, we assume $x_5, y_5 \in 5 \Bbb Z_5$. Then  $ z_5^2 \equiv 4+ 20 \mod 25$ by the given equation. This implies that $z_5\equiv \pm 7 \mod 25$. Therefore
$$ (x_5\pm 2,  5)_5= (y_5\pm 2,  5)_5=1/2 \ \ \ \text{and} \ \ \ (z_5\pm 2,  5)_5=0 .$$

Thus for any point $M_{5} \in U(\Bbb Z_5)$ at most 3 of the elements in $A$ vanish at $M_{5}$.

\medskip

Let now  $M_2=(x_2,y_2,z_2)\in \mathcal U(\Bbb Z_2)$.
Recall that $(2,5)_{2}=1/2$ and $(u,5)_{2}=0$ for any $u \in \Bbb Z_{2}^{\times}$.

a2) If one  coordinate, say $x_{2}$, belongs to $\Bbb Z_2^{\times}$,  then
each of $x_{2} \pm 2 $ is in $\Bbb Z_2^{\times}$ hence $(x_{2} \pm 2 ,5)_{2}=0$.
From the given equation we immediately see that if
 $M_2$  has one coordinate in  $\Bbb Z_2^\times$,  then it has at least 2.
 This then implies that at least 4 elements   in $A$ vanish at $M_2$.

b2)  If no coordinate of $M_{2}$  is in $\Bbb Z_2^\times$, then
one can write $$x_2=2\xi,  \ y_2=2\eta, \ z_2=2\delta \  \ \ \text{with} \  \ \  \xi, \eta, \delta\in \Bbb Z_2$$
 and the equation gives
$$ ( \xi^2-1)(\eta^2-1)=(\delta-\xi \eta)^2 - 5v^2 . $$
Since $5\not\in \Bbb Z_2^{\times 2}$, one concludes that $\xi$ and $\eta$ are in $2\Bbb Z_2$.
Similarly, $\delta \in 2\Bbb Z_2$.  For each element in the set
$$  \{(x \pm 2,  5),  (y \pm 2,  5),  (z \pm 2,  5)\}  $$
the value it takes on $M_{2}$ is of the shape $(2u,5)_{2}$ with $u \in \Bbb Z_{2}^{\times}$.  We see that
 all elements in $A$  take the value $1/2$ at $M_2$.

 \medskip
It is then an easy matter to see that in whichever combination of one of a5), b5) with one of a2),  b2),
there exists an element $\beta \in B$ such that
$\beta(M_{5})+ \beta(M_{2}) \neq 0$.  Hence
for any ad\`ele $\{ M_{p}\} \in  \mathcal{U} (A_\mathbb Z)$ there exists an element $\beta \in A$
with the property
$$ \sum_{p} \beta(M_{p}) \neq 0 \in \Bbb Q/\Bbb Z.$$
\end{proof}

\subsection{Combination of Brauer-Manin obstruction with the reduction theory}

\begin{lem} \label{vanish} Suppose $m\neq 0,4$ and $d=m-4$. Let $p$ be an odd prime such that $ord_p(d)$ is even and positive.  Then there is a point $(x_p,y_p,z_p)\in \mathcal U_m(\mathbb Z_p)$ such that
$$(x_p-2,d)_p=(y_p-2,d)_p=(z_p-2,d)_p=0.$$
\end{lem}
\begin{proof} For any odd prime $p$ and  $a\neq \pm 2$ in the  finite field $\mathbb F_p$,
the point  $(a,a,2)$  is a smooth point of   the affine variety over $\mathbb F_p$ defined by  $x^2+y^2+z^2-xyz=4$.
  By Hensel's Lemma, there exists a point $(x_p,y_p,z_p)\equiv (a,a,2) \mod p$ in $\mathcal U_m(\mathbb Z_p)$. Therefore
$$(x_p+2,d)_p=(x_p-2,d)_p=(y_p-2,d)_p=0.$$
By (\ref{relation-br}),
one has $(z_p-2,d)_p=0$.
\end{proof}

The following proposition points out that \cite[Proposition 8.1 ii)]{GS} cannot be explained only by Brauer-Manin obstruction.

\begin{prop}\label{non-empty}  Let $\mathcal U$ be the scheme over $\Bbb Z$ given by
\begin{equation}\label{equ-non-empty}
 x^2+y^2+z^2 -xyz=4+ 2  l^2w^2
 \end{equation}
 where $w$ is an odd integer and $l$ is a prime with $l \equiv   \pm 3 \mod 8$.

  If $l w \equiv \pm 4 \mod 9$,
 then
$\mathcal U(A_\mathbb Z)^{\Br}  \neq \emptyset$.
\end{prop}

\begin{proof}  By \cite[Proposition 6.1]{GS}, the condition $l w \equiv \pm 4 \mod 9$ implies 
$\prod_{p\leq \infty} \mathcal U(\Bbb Z_p) \neq \emptyset $. Since $l w$ is odd, the integer $4+2l^2w^2$ is not a square. Therefore, by Corollary \ref{br-prac} and Theorem \ref{br}, the quotient
  $\Br(U)/\Br_0(U)$ is generated by \begin{equation} \label{elements} \{(x-2, 2), (y-2, 2), (z-2, 2) \} \end{equation}

By Lemma \ref{gen}, for $p\nmid 2lw$,
the  three elements in (\ref{elements})  vanish over $\mathcal U(\Bbb Z_p)$.  By Lemma \ref{vanish}, there is a $\mathbb Z_p$-point $M_p$
at which all three elements in (\ref{elements})  vanish for any $p\mid w$ and $p\neq l$. We fix such points.

We shall construct suitable  local points $M_p=(x_p,y_p,z_p)$ for $p=2, l$.

For $p=2$, we take $x_2=y_2=1$.  By Hensel's Lemma, there is $z_2\in \Bbb Z_2^\times$ satisfying  \begin{equation} \label{hr} z^2-z=2+2l^2w^2 \end{equation}.
 Then
$  (x_2- 2,  2)_2= (y_2- 2,  2)_2=0$ and
$$(z_2-2, 2)_2= (-1-r, 2)_2= \frac{1}{2},$$
where $r$ is the other root of (\ref{hr}) with $\ord_2(r)=\ord_2(2+2l^2w^2)=2$.

Over the finite field $\Bbb F_{l}$, we can choose $(a,b,c)\in \Bbb F_{l}\times \Bbb F_{l}^\times\times\Bbb F_{l}^\times$ satisfying $a^2-4bc=2w^2$. Obviously $a-b-c\neq 0$, otherwise
we have $(b-c)^2=2w^2$, which is impossible since $(\frac{2}{l})=-1$.
Therefore $(b,c, a-b-c)$
is a solution of  the equation
$$(x'+y'+z')^2-4x'y'=2w^2 \ \mod l$$ with $x'y'z'\neq 0$, hence  by Hensel's lemma there is a solution $(\alpha_l, \beta_l, \gamma_l)$ of the equation
$$ (x'+y'+z')^2-x'y'(4+l \cdot z')=2w^2$$ over $\Bbb Z_l$ with $\gamma_l \in \Bbb Z_l^\times$.
Then $$(x_l,y_l,z_l)=(-2+\alpha_l  l,-2+\beta_l l, 2+\gamma_l l)\in \mathcal U_m(\mathbb Z_l) $$ with  $$(x_l-2,2)_l=(y_l-2,2)_l=0 \text{ and } (z_l-2,2)_l=1/2.$$
One concludes that  $$(x_p,y_p,z_p)_{p\leq \infty}\in \mathcal U(A_\mathbb Z)^{\Br} $$ as desired.
\end{proof}

If   $w=1$ in Proposition \ref{non-empty} and
$l$ is a sufficiently large prime,
 one can still prove
  the equation (\ref{equ-non-empty})
has no integral solutions by combining Brauer-Manin obstruction with the reduction theory as given in \cite[Proposition 8.1 ii)]{GS}. In fact, we  produce more counterexamples.

\begin{prop}  \label{cb} The equation
$$ x^2+y^2+z^2-xyz= 4+ r l^2$$ has no integral solution in each of the following cases:

i) $r=2$ and $l\geq 13$ is a prime with $l \equiv \pm 4 \mod 9$;

ii) $r=12$ and $l\geq 37$ is a prime, $l^2 \equiv 25 \mod 32$ and $1+3l^2$ is not a sum of two squares (e.g. $l=37,43,...$);

iii) $r=-2$ and $l\geq 13$ is a prime;

iv) $r=-3$ and $l\geq 17$ is a prime;

v) $r=-12$ and  $l\geq 37$ is a prime.
\end{prop}

\begin{proof}
 Let us first check that in each of the above cases, $m=4+rl^2$ is ``generic''
as defined in \cite{GS}, i.e.
there is no integral solution  with one of the coordinates of absolute value 0, 1 or 2.
This is automatic for $m<0$, hence in cases (iii), (iv), (v).
In case i), see the proof of \cite[Proposition 8.1]{GS}.
In case  ii),  $u^2+3v^2=4(m-1)=4(3+12l^2)$ is not solvable over $\mathbb Z$ because
$$(-3, 4(3+12l^2))_3=(-3,1+4l^2)_3=(-3,5)_3=1/2.$$
By our assumption, $u^2+v^2=4+12l^2$ is not solvable over $\mathbb Z$. Since $12l^2$ is not a square, $4+12 l^2$ is generic.

Let us now   suppose that  one of the given equations has an integral solution.

In the cases i) and ii),
by the reduction theory
  (\cite[Theorem 1.1]{GS}),
 there is an integral solution $(x_0, y_0, z_0)$ satisfying $$3\leq |x_0|\leq |y_0|\leq |z_0| \text {   and } |x_0| \leq (4+rl^2)^{\frac{1}{3}}.$$

Suppose $r=2$ and $l\geq 13$, or  $r=12$ and $l\geq 37$.  We have $|x_0|+2< l$. This implies that $x_0^2-4$ has no $l$-factor. We therefore have $(x_0^2-4, r)_l=0$.

By the purely local computations in Proposition \ref{8.1-8.2}, if $r=2$, we have $(x_0^2-4, r)_2=1/2$.
Then   we have
$$(x_0^2-4, r)_p=\begin{cases} 0 & \text{   if }p\neq 2\\
1/2 & \text{   if }p= 2;
\end{cases}
$$
Similarly, by the purely local computations in Proposition \ref{8.1-8.2}, if $r=12$,  we have $$(x_0^2-4, r)_2=0 \text{ and }(x_0^2-4, r)_3=1/2. $$
Therefore
$$(x_0^2-4, r)_p=\begin{cases} 0 & \text{   if }p\neq 3\\
1/2 & \text{   if }p= 3.
\end{cases}$$ This contradicts the Hilbert reciprocity law.
\medskip

In the cases iii), iv) and v),
by the reduction theory
(\cite[Theorem 1.1]{GS}), there is an integral solution $(x_0, y_0, z_0)$ satisfying $$3\leq x_0\leq y_0\leq z_0\leq \frac{1}{2} x_0y_0 . $$

We claim $x_0<l-2$. Otherwise, we would have
\begin{align*}
-rl^2-4=&x_0y_0z_0-x_0^2-y_0^2-z_0^2\geq x_0y_0z_0-x_0^2-y_0^2-\frac{1}{2}x_0y_0z_0\\
=&\frac{1}{2}x_0y_0z_0-x_0^2-y_0^2\geq \frac{1}{2}(l-2)y_0^2-2y_0^2\\
=&\frac{1}{2}(l-6)y_0^2 \geq \frac{1}{2}(l-6)(l-2)^2.
\end{align*}
If $r=-2 \text{ and } l\geq 13$, or $r=-3 \text{ and } l\geq 17$, or $r=-12 \text{ and } l\geq 37$. This  is impossible. This implies that $x_0^2-4$ has no $l$-factor and thus
$(x_0^2-4, 2)_l=0$.

\medskip

By the purely local computations in Proposition \ref{8.1-8.2}, if $r=-2$, we have $(x_0^2-4, r)_2=1/2$.
 Then
$$(x_0^2-4, r)_p=\begin{cases} 0 & \text{   if }p\neq 2\\
1/2 & \text{   if }p= 2.
\end{cases}
$$
This contradicts  the Hilbert reciprocity law.

By the purely local computations in Proposition \ref{8.1-8.2},  if $r=-3,-12$,  one has $$(x_0^2-4, r)_2=0 \text{ and }(x_0^2-4, r)_3=1/2. $$ So
$$(x_0^2-4, r)_p=\begin{cases} 0 & \text{   if }p\neq 3\\
1/2 & \text{   if }p= 3.
\end{cases}$$ This contradicts  the Hilbert reciprocity law.
\end{proof}

The following Lemma is an extension of the previous proposition.  One needs this extension in order to get the lower bound in Theorem \ref{bound}.
\begin{lem}\label{almost} Let $r=2,-2,-3,-12$. Let $a>0$ be an integer and $l$ be a prime.
Let $m=4+ra^2l^2$. Suppose $a>0$ is
 prime to $r$ and that the Hilbert symbol $(p,r)_p=0$ for any prime divisor $p$  of $a$. In the case $r=2$,   suppose moreover
 $al\equiv \pm 4 \mod 9$.

Then there exists
 a positive constant $\theta_r>0$ only depending on $r$, such that, if $a< \theta_r l^{1/2}$ and $l$ is large enough (depending on $\theta_r$),
 then
the equation
$$ x^2+y^2+z^2-xyz= 4+ r a^2 l^2$$ has no integral solution.
\end{lem}
\begin{proof}

   Assume there is an integral solution.

i) Suppose $r=2$.  By the last part of the proof of \cite[Proposition 8.1]{GS},  it is clear that $4+ r a^2 l^2$ is "generic". By the reduction theory (\cite[Theorem 1.1]{GS}),
 there is an integral solution $(x_0, y_0, z_0)$ satisfying $$3\leq |x_0|\leq |y_0|\leq |z_0| \text {   and } |x_0| \leq (4+2a^2l^2)^{\frac{1}{3}}.$$

If $\theta_2 <1/\sqrt{2}$, then
$$|x_0| \leq (4+2a^2l^2)^{\frac{1}{3}} <(4+2\theta_2^2l^3)^{1/3} <l-2,$$
the last inequality holds for $l$ large enough. This implies that $x_0^2-4$ has no $l$-factor. Therefore 
$(x_0^2-4, 2)_l=0$.
By  similar purely local computations as in Proposition \ref{cb},
we conclude that the integral  Brauer-Manin set of the equation
$$ x^2+y^2+z^2-xyz= 4+ r a^2 l^2$$ is empty, hence this equation
has no integral solution.

ii) Suppose $r=-2,-3,-12$.
 By the reduction theory
(\cite[Theorem 1.1]{GS}), there is an integral solution $(x_0, y_0, z_0)$ satisfying $$3\leq x_0\leq y_0\leq z_0\leq x_0y_0/2. $$

We have
\begin{align*}
-ra^2l^2-4=&x_0y_0z_0-x_0^2-y_0^2-z_0^2 \geq x_0y_0z_0/2-x_0^2-y_0^2\\
\geq & (x_0 /2-1)y_0^2     -x_0^2 \geq x_0\cdot x_0^2/2-x_0^2-x_0^2= x_0^3/2-2x_0^2.
\end{align*}
If we choose $0<\theta_r< 1/\sqrt{-2r}$, then $x_0<l-2$ for $l$ large enough. Therefore  
 $(x_0^2-4, r)_l=0$. By    purely local computations as in Proposition \ref{cb},
we conclude that the integral  Brauer-Manin set of the equation
$$ x^2+y^2+z^2-xyz= 4+ r a^2 l^2$$ is empty, hence this equation
has no integral solution.
\end{proof}

The following result improves
  upon
 the lower bound $\sqrt{N}(\log N)^{-1}$ in \cite[Theorem 1.5]{LM}.
\begin{thm}\label{bound}
Let  $\mathcal U_m$ be the affine scheme over $\Bbb Z$ defined by the equation
$$ x^2+y^2+z^2-xyz= m. $$
We have
\begin{align*}
 &\#\{m\in \mathbb Z: 0<m<N, \ \mathcal U_m(A_\mathbb Z)^\Br\neq \emptyset \ \text{ but } \ \mathcal U_m(\mathbb Z)=\emptyset\}\gg  \sqrt{N}(\log N)^{-1/2};\\
 &\#\{m\in \mathbb Z: -N<m<0, \ \mathcal U_m(A_\mathbb Z)^\Br\neq \emptyset \ \text{ but } \ \mathcal U_m(\mathbb Z)=\emptyset\}\gg  \sqrt{N}(\log N)^{-1/2}
 \end{align*}
 as $N \to +\infty$.
\end{thm}
\begin{proof}
a)  To prove the first asymptotic inequality, we restrict attention to positive integers   $m=4+2a^2l^2$  
with  $l$   a prime, $l \equiv 19  \mod 72$ and   $a$  an odd positive integer satisfying
$$ (*):  \ \   a \equiv \pm 4 \mod 9 \  \text{ and all prime divisors of $a$ are congruent to $\pm 1 \mod 8$ }.$$
Fix $\theta_2 < 1/\sqrt{2}$ as in the proof of Lemma  \ref{almost}.
By this lemma, 
if $a< \theta_2 l^{1/2}$
and $l$ is large enough, then
the equation
$$ x^2+y^2+z^2-xyz= 4+ 2 a^2 l^2$$ has no integral solution.
  By Proposition \ref{non-empty}, we have $\mathcal U_m(A_\mathbb Z)^\Br\neq \emptyset$ for the above values of $m$.

Let $$N_B = \#\{m\in \mathbb Z: 0<m<N, \ \mathcal U_m(A_\mathbb Z)^\Br\neq \emptyset \ \text{ but } \ \mathcal U_m(\mathbb Z)=\emptyset\} . $$ By Lemma \ref{almost},
one obtains
\begin{align*}
N_B \geq &\sum_{l<\sqrt{N},  \ l\equiv 19\mod 72 }\#\{a:a<\theta_2\sqrt{l}, a<\sqrt{N}/l, a \text { satisfies } (*)\}\\
\geq &\sum_{\theta_2^{-2/3}N^{1/3}< l<N^{1/2} , \ l\equiv 19\mod 72 }\#\{a:a<\sqrt{N}/l, a \text { satisfies } (*)\}\\
\geq &\sum_{\theta_2^{-2/3}N^{1/3}< l<N^{5/12} , \ l\equiv 19\mod 72 }\#\{a:a<\sqrt{N}/l, a \text { satisfies } (*)\}\\
\end{align*} 
By a well known lemma (e.g., \cite[\S 5.8]{LM}), 
one has
 $$\#\{a<N: a \text { satisfies } (*)\} \thicksim c N (\log N)^{-1/2} \ \ \ \ \ \ \ \text{as $N\to +\infty$} $$ where $c>0$ is a constant.
Using  \cite[p.156, Ex. 6]{A}, we obtain
 \begin{align*}
N_B\gg &\sum_{\theta_2^{-2/3}N^{1/3}< l<N^{5/12}, \ l\equiv 19 \mod 72  } \sqrt{N}(\log \sqrt{N}-\log l)^{-1/2}l^{-1}\\
\geq &\sqrt{N}(\log N)^{-1/2}  \sum_{\theta_2^{-2/3} N^{1/3}< l<N^{5/12},  \ l\equiv 19 \mod 72  }  l^{-1}\\
\gg  &\sqrt{N}(\log N)^{-1/2}(\log\log (N^{5/12})-\log\log (N^{1/3})-\log(1-\frac{2\log(\theta_2)}{\log N})+O((\log N)^{-1}))\\
=  &\sqrt{N}(\log N)^{-1/2}(\log(5/4)+O((\log N)^{-1})) \gg  \sqrt{N}(\log N)^{-1/2}
\end{align*}
as $N\to +\infty$

b) To prove the second asymptotic inequality, we now restrict attention to integers  $m=4-2a^2l^2$  
and apply Lemma \ref{almost} to the case $r=-2$. Since $\sqrt{-1}\not \in \mathbb Q(\sqrt{d})=\mathbb Q(\sqrt{-2}),$  Corollary \ref{br-Q} gives $\Br(U_m)= \Br_1(U_m)$. The result follows from an argument entirely analogous to the previous one.
\end{proof}

\section{Strong approximation always fails} \label{fsa}

Let $\mathcal U_m$ be the scheme  over $\Bbb Z$ defined by the equation
\begin{equation}\label{equ-m}   x^2+y^2+z^2-xyz= m.  \end{equation}
The following proposition complements \cite[Theorem 1.1 (i)]{GS}
 (see also the discussion below \cite[Lemma 2.1]{GS}), which goes back to
 Markoff, Hurwitz, Mordell.  Theorem 1.1(i) of  \cite{GS} contains the further information
 that if $m\in \Bbb Z$ is ``generic'', i.e. there no point on $U_{m}(\Bbb Z)$ with
 $x=0,1,2$,
  then $\Gamma$ acts transitively on the solutions
 and it describes an explicit fundamental set for the set
 of integral solutions.
 
\begin{prop} \label{induction}  If $m>0$, then any integral point in $\mathcal U_m(\mathbb Z)$ is  $\Gamma$-equivalent to an integral point $(x_0,y_0,z_0)\in \mathcal U_m(\mathbb Z)$ such that
\begin{equation}\label{orbit}
3\leq x_0 \leq y_0 \leq -z_0 \ \ \ \text{ or } \ \ \ x_0=0,1,2.
\end{equation}
\end{prop}
\begin{proof}
For a given integral point, if its $\Gamma$-orbit contains an integral point with the coordinate $x=0,1,2$, then the proof is completed. Therefore,
we may assume there is no integral point in the $\Gamma$-orbit  with $x=0,1,2$. 
By changing sign of two coordinates and  permutation of  the coordinates, 
one only needs to consider the generic case, i.e. 
$\Gamma$-orbits of integral points such  that  for any   point  $(x,y,z)$ in the orbit we have $$\min \{|x|, |y|, |z|\} \geq 3 . $$

By changing sign of two coordinates simultaneously, we only need to consider the following two cases:  two coordinates of $(x, y, z)$ are positive and the remaining one is negative; or all coordinates  of  $(x, y, z)$ are positive.

Suppose that there is an integral point $(x, y, z) \in \mathcal U_m(\mathbb Z)$ such that two coordinates of $(x, y, z)$ are positive and the remaining one is negative.
Then the result follows from changing sign of two coordinates so that all of them are negative, permutation of the coordinates
so as to get $ |x|  \leq  |y| \leq  |z|$
and then change of sign of $x$ and $y$.

Now we consider an integral point $(x, y, z)\in \mathcal U_m(\mathbb Z)$ such that $3\leq x\leq y\leq z$.

 If $z\leq \frac{1}{2} xy$, then one obtains
$$ z= \frac{1}{2} (xy -\sqrt{x^2y^2-4(x^2+y^2-m)}) $$ by solving (\ref{markoff1}) for $z$. This implies  
$$ \sqrt{x^2y^2-4(x^2+y^2-m)} =xy-2z \leq xy-2y .$$ Therefore one has
$$ (x-2)y^2 \leq x^2 -m $$ by squaring. From $x\geq 3$ and $m>0$ one concludes $y^2< x^2$. A contradiction is derived.

For any integral point $(x, y, z)\in \mathcal U_m(\mathbb Z)$ with $3\leq x\leq y\leq z$,
we thus have  $z>\frac{1}{2}xy$.
Applying the Vieta involution, one obtains a new integral point $(x, y, xy-z)$
which satisfies $xy-z< z$.
If $xy-z\leq 2$, since we are in the generic case we must have $xy-z \leq -3$, so we have a situation with two coordinates positive and one negative, and we conclude as
above.
Suppose $xy-z\geq 3$. We obtain a new integral point $(x_1, y_1, z_1)$ in the $\Gamma$-orbit of $(x, y, z)$ with positive coordinates
and $x_{1}+y_{1}+z_{1} < x+y+z$. This process must stop, that is we reach a situation  with two coordinates positive and one negative.
\end{proof}

The main result of this section is the following theorem.

\begin{thm} \label{wsa}  Let $m$ be any integer.
Suppose $\mathcal U_m (A_\mathbb Z)\neq \emptyset$. For any finite set $S$
of primes,
 the image of the natural map $\mathcal U_m(\mathbb Z) \to \prod_{p \notin S} \mathcal U_m (\mathbb Z_p)$  is not dense.\end{thm}
\begin{proof}
For any sets of primes $S_1\supset S_2$, if $\mathcal{U}_m(\Bbb Z)$ is not dense in $\prod_{p\not \in S_1} \mathcal{U}_m(\Bbb Z_p)$, then $\mathcal{U}_m(\Bbb Z)$ is not dense in $\prod_{p\not \in S_2} \mathcal{U}_m(\Bbb Z_p)$. One can thus  enlarge $S$ if necessary.

i)  Suppose $m\neq 0$. We may assume $S$ contains $2$ and $\infty$. Let $S'=\{p\text{ prime}: p\mid m\}$ and $R= \prod_{p\in S\setminus S'} p$.
Let $a$ be a positive integer prime to $m$ such that
\begin{equation}\label{inequality-a}
a^2R^2-2aR-m\geq 0 \text{ and }aR>\sqrt{|m|+9}.
\end{equation}
Let $d'=a^2R^2-m$ and $e'_p=ord_p(d')$.

Denote
 $$\mathcal V_{\epsilon,1,d'}:= \prod_{p\mid d'}\{(x_p,y_p,z_p)\in \mathcal U_m(\mathbb Z_p): (x_p,y_p,z_p) \equiv (\epsilon aR ,0,0) \mod p^{e'_p} \},$$
 $$\mathcal V_{\epsilon,2,d'}:= \prod_{p\mid d'}\{(x_p,y_p,z_p)\in \mathcal U_m(\mathbb Z_p): (x_p,y_p,z_p) \equiv (0,\epsilon aR,0) \mod p^{e'_p} \},$$
 $$\mathcal V_{\epsilon,3,d'}:= \prod_{p\mid d'}\{(x_p,y_p,z_p)\in \mathcal U_m(\mathbb Z_p): (x_p,y_p,z_p) \equiv (0,0,\epsilon aR ) \mod p^{e'_p} \},$$
 where $\epsilon= \pm 1.$  Let $$\mathcal V_{\epsilon,d'}=\bigcup_{i=1}^3\bigcup_{\epsilon=\pm 1}\mathcal V_{\epsilon,i,d'}.$$
   It is clear that $\mathcal V_{\epsilon, d'}$ is $\Gamma$-invariant, where $\Gamma$ is the group defined in \S 1.
   Since $d'$ has no prime factor in $S\cup S'$,
we can take the local point $(x'_p,0,0)$ of $\mathcal U_m(\mathbb Z_p)$ with $x'_p\equiv a R \mod p^{e'_p}$ for any $p\mid d'$ by Hensel's lemma.
  Obviously, $\prod_{p\mid d'}(x'_p,0,0)\in \mathcal V_{1,1,d'}$.
   Therefore $\mathcal V_{\epsilon, d'}$ is a non-empty open subset of $\prod_{p\mid d'} \mathcal U_m(\mathbb Z_p)$.

a) Suppose $m>0$.
 Assume that $\mathcal U_m(\Bbb Z)$ is dense in $\prod_{p \notin S} \mathcal U_m (\mathbb Z_p)$. Then $\mathcal U_m(\mathbb Z)\cap \mathcal V_{\epsilon, d'} \neq \emptyset$.
 By Proposition \ref{induction},
there is an integral point $(x_0,y_0,z_0)\in \mathcal U_m(\mathbb Z)\cap\mathcal V_{\epsilon,d'}$ such that  \begin{equation} \label{pos-red} 3\leq x_0 \leq y_0 \leq -z_0  \ \ \ \text{ or } \ \ \ x_0=0,1,2. \end{equation}
Since $(x_0,y_0,z_0)\in \mathcal V_{\epsilon,d'}$, we have
$$(x_0,y_0,z_0) \equiv (\pm aR,0,0), (0,\pm aR,0) \text{ or } (0,0,\pm aR )\ \ \ \mod d' .$$
If $x_0> 0$, then
\begin{equation} \label{inequality}
x_0\geq min\{d',d'-aR,aR\}=aR>\sqrt{m+9} >3
\end{equation}
by (\ref{inequality-a}).
Hence $3 \leq x_0 \leq (m-27)^{1/3}$ by (\ref{equ-m}) and (\ref{pos-red}).
We have $\sqrt{m+9}>(m-27)^{1/3}$.  By  (\ref{inequality})
 a contradiction is derived.
Therefore
$$x_0=0, y_0^2+z_0^2=m \text{ and } (y_0,z_0)\equiv (\pm aR,0) \text{ or } (0,\pm aR) \mod d',$$ which is impossible by (\ref{inequality-a}). Therefore $\mathcal U_m(\mathbb Z)$ is not dense in
$\prod_ {p\mid d'} \mathcal U_m (\mathbb Z_p)$, hence is not dense in
$\prod_{p \notin S} \mathcal U_m (\mathbb Z_p)$.

b) Suppose $m<0$. Assume that $\mathcal U_m(\Bbb Z)$ is dense in $\prod_{p \notin S} \mathcal U_m (\mathbb Z_p)$. Then $\mathcal U_m(\mathbb Z)\cap \mathcal V_{\epsilon, d'} \neq \emptyset$.
By \cite[Theorem 1.1 (ii)]{GS},
there is an integral point $(x_0,y_0,z_0)\in \mathcal U_m(\mathbb Z)\cap\mathcal V_{\epsilon,d'}$ such that $$3\leq x_0 \leq y_0 \leq z_0 \leq x_0y_0/2.$$
By \cite[Lemma 2.2]{GS},  one has $3\leq x_0\leq \sqrt{|m|+9}$.
Since $(x_0,y_0,z_0)\in \mathcal V_{\epsilon,d'}$, we have
$$(x_0,y_0,z_0) \equiv (\pm aR,0,0), (0,\pm aR,0) \text{ or } (0,0,\pm aR )\ \ \ \mod d',$$
Since $x_0> 0$, then
\begin{equation*}
x_0\geq min\{d',d'-aR,aR\}=aR>\sqrt{m+9}
\end{equation*}
by (\ref{inequality-a}), which contradicts $x_0\leq\sqrt{|m|+9}$. 
Therefore $\mathcal U_m(\mathbb Z)$ is not dense in
$\prod_ {p\mid d'} \mathcal U_m (\mathbb Z_p)$, hence is not dense in
$\prod_{p \notin S} \mathcal U_m (\mathbb Z_p)$.

ii) Suppose $m=0$.

We can choose a prime $l\notin S$ and $l\equiv 1 \mod 4$. Then we may take $\delta\in \mathbb Z_l^\times$ such that $\delta^2=-1$. Therefore $(\delta l,l,0)\in \mathcal U_0(\mathbb Z_l)$. If  $\mathcal U_0(\Bbb Z)$ is dense in $\prod_{p \notin S} \mathcal U_0 (\mathbb Z_p)$, then there is an integral point $(x_0,y_0,z_0)\equiv (\delta l,l,0) \mod l^2$. Therefore $(x_0,y_0,z_0)\neq (0,0,0)$ and $x_0,y_0,z_0$ are all divisible by $l$. Since $\mathcal U_0(\Bbb Z)$  has just two orbits $(0,0,0)$ and $(3,3,3)$ (see \cite[\S 3.1]{GS}), $(x_0,y_0,z_0)$ is contained in the orbit $(3,3,3)$. One has $l\mid 3$ since $x_0,y_0,z_0$ are all divisible by $l$, which is impossible.  Therefore $\mathcal U_0(\mathbb Z)$ is not dense in $\prod_{p \notin S} \mathcal U_0 (\mathbb Z_p)$. The proof is completed. 
\end{proof}

We can ask for
 a lighter version of strong approximation: could it be that  the  reduction map $\mathcal U_m(\mathbb Z) \to \mathcal U_m (\mathbb Z/l)$ is surjective for almost all primes $l$? For $m$ not a square, the following proposition gives a conditional negative answer. Indeed it is a special case of Schinzel's conjecture that
under this hypothesis  on $m$ the polynomial  $x^2-m \in  {\mathbb Z}[x] $  represents infinitely many primes as $x$ varies in $\mathbb Z$.
	 	
	\begin{prop} Assume that $m$ is not a square and that     the polynomial  $x^2-m \in  {\mathbb Z}[x] $  represents infinite many primes. Then  there exist infinitely many primes $l$ for which there is a point in $\mathcal U_m(\mathbb Z/l)$ of the shape $(\overline x,0,0)$ with $\overline x\neq 0$ which is not in the image of $\mathcal U_m(\mathbb Z) \to \mathcal U_m(\mathbb Z/l)$.
\end{prop}
\begin{proof}  Let $l$ be a prime of the shape $l=a^2-m$ with  $m \in \mathbb Z$ and
$a$ is a positive integer prime to $m$, such that
\begin{equation}\label{inequality-a2}
a^2-2a-m\geq 0 \text{ and }a>\sqrt{|m|+9}.
\end{equation}  By the above conjecture, there exists infinitely many such pairs $(l,a)$. Denote
 $$\mathcal V_{l}:= \{(\pm \overline a,0,0),(0,\pm \overline a,0),(0,0,\pm \overline a)\} \subset  (\mathbb Z/l)^3,$$
 here $\overline a$ is the image of $a$ in $\mathbb Z/l$. It is clear that $\mathcal V_{l} \subset \mathcal U_m(\mathbb Z/l)$ is $\Gamma$-invariant.
		
We will assume $m>0$ (the case $m<0$ can be proved similarly).
Assume that the  map $\mathcal U_m(\mathbb Z) \to \mathcal U_m (\mathbb Z/l)$ is surjective. Then there is an integral point $\vec{x}\in \mathcal U_m(\mathbb Z)\cap \mathcal V_{l}$. By Proposition \ref{induction} (\cite[Theorem 1.1 (ii) and Lemma 2.2]{GS} for $m<0$),
there is an integral point $(x_0,y_0,z_0)\in \mathcal U_m(\mathbb Z)\cap \mathcal V_{l}$ such that $$3\leq x_0 \leq y_0 \leq -z_0 ,\text{ or } x_0=0,1,2.$$
Since $(x_0,y_0,z_0)\in \mathcal V_{l}$, we have
$$(x_0,y_0,z_0) \equiv (\pm a,0,0), (0,\pm a,0) \text{ or } (0,0,\pm a )\ \ \ \mod l,$$
hence, if $x_0> 0$,  
\begin{equation} \label{inequality2}
x_0\geq min\{l,l-a,a\}=a>\sqrt{m+9}
\end{equation}
by (\ref{inequality-a2}).
Since $\sqrt{m+9}>3$, one has $x_0\neq 1,2$. If $3\leq x_0 \leq y_0 \leq -z_0$,
hence $3 \leq x_0 \leq (m-27)^{1/3}$ by (\ref{equ-m}).
But $(x_0,y_0,z_0)\in \mathcal V_{l}$,
one has $x_0>\sqrt{m+9}>(m-27)^{1/3}$ by (\ref{inequality2}), which is a contradiction to $x_0 \leq (m-27)^{1/3}$. Therefore $$x_0=0, y_0^2+z_0^2=m\text{ and } (y_0,z_0)\equiv (\pm a,0) \text{ or } (0,\pm a) \mod l.$$
Then $$(y_0,z_0)\equiv (\pm a,0) \text{ or } (0,\pm a) \mod l$$ implies $|y_0| \text{ or } |z_0| \geq  min \{l-a,a\}=a$, hence $a^2 \leq m$, which is impossible by (\ref{inequality-a2}).
Therefore $\mathcal U_m(\mathbb Z) \to \mathcal U_m (\mathbb Z/l)$ is not surjective.
\end{proof}

\begin{rem}\label{compareBGS} 
When comparing the above results with \cite{BGS},
one should note that the failures of strong approximation described here correspond to
points $(x_{p},y_{p},z_{p}) \in  U_m({\mathbb Z}_{p})$  whose reduction modulo $p$ has two coordinates
equal to $0$, hence which geometrically lift to points whose $\Gamma$-orbit is finite.
\end{rem}

\begin{lem} \label{lem:wsa}Let $k$ be a number field. Let $U$ be a smooth geometrically connected variety over $k$ such that $\Br(U)/\Br_0(U)$ is finite. Let $v$ run through the places of $k$.
Suppose $\mathcal U$ is an integral model of $U$ over $\frak o_k$ with $\mathcal U(A_{\frak o_k})^\Br \neq \emptyset $, here $\mathcal U(A_{\frak o_k})=\prod_{v\mid \infty}U(k_v)\times \prod_{v<\infty}\mathcal U(\frak o_v)$. Let $pr_f: \mathcal U(A_{\frak o_k})\to \prod_{v<\infty} \mathcal U(\frak o_v)$ be the natural projection.

If $\mathcal U(\frak o_k)$ is dense in $pr_f(\mathcal U(A_{\frak o_k})^\Br )$, then there exists a finite set $S$ of places containing $\infty_k$ such that the natural map $\mathcal U(\frak o_k) \to \prod_{v \notin S} \mathcal U_m (\frak o_v)$  has
dense image.
\end{lem}
\begin{proof} Suppose  $\mathcal B_1,\cdots, \mathcal B_n$ generate  $\Br(U)/\Br_0(U)$. Then, there exists a finite set $S$ of places containing $\infty_k$ such that $\mathcal B_1,
\cdots, \mathcal B_n$ vanish on $\mathcal U(\frak o_v)$ for any $v\notin S$.
 Since $\mathcal U(A_{\frak o_k})^\Br \neq \emptyset $, the natural projection
$\mathcal U(A_{\frak o_k})^\Br \rightarrow \prod_{v \notin S} \mathcal U (\frak o_v)$
is surjective. So, if $\mathcal U(\frak o_k)$ is dense in $pr_f(\mathcal U(A_{\frak o_k})^\Br )$, then $\mathcal U(\frak o_k)$ is dense in $\prod_{v \notin S} \mathcal U (\frak o_v)$.
\end{proof}

The above lemma is the exact analogue of the well known statement:
if $X$ is projective over a number field $k$  and $\Br(X)/\Br(k)$ is finite, and $X(k)$ is dense in $X(A_{k})^\Br$ nonempty, then
weak weak approximation holds for $X$.

\begin{cor} \label{notsa} Suppose $ m\neq 0,4$ and $\mathcal U_m(A_\mathbb Z)^{\Br}\neq \emptyset$. Then $\mathcal U_m(\Bbb Z)$ is not dense in
$pr_f(\mathcal U_m(A_\mathbb Z)^{\Br})$,
where $pr_f: \mathcal U_m(A_\mathbb Z) \rightarrow \prod_{p<\infty }\mathcal U_m(\Bbb Z_p) $ is the natural projection.
\end{cor}
\begin{proof} By Theorem \ref{br} and \ref{br-gen}, $\Br(U_m)/\Br_0(U_m)$ is finite.  The proof follows from Theorem \ref{wsa} and Lemma \ref{lem:wsa}.
\end{proof}

\begin{cor} \label{notzariski}
 Let $pr_f: \mathcal U_m(A_\mathbb Z) \rightarrow \prod_{p<\infty }\mathcal U_m(\Bbb Z_p) $ be the natural projection. Assume that $\mathcal U_m(\mathbb Z)\neq \emptyset$.

If $m>4$ is not a square,  or $m$ is a square with a prime factor congruent  to $1 \mod 4$, or $m<0$, then  $\mathcal U_m(\mathbb Z)$ is Zariski dense but is not dense in
$pr_f(\mathcal U_m(A_\mathbb Z)^{\Br})$.
\end{cor}
\begin{proof} By \cite[\S 5.2]{GS}, $\mathcal U_m(\Bbb Z)$ is Zariski dense. The result follows from  Corollary \ref{notsa}.
\end{proof}

Let $X$ be a smooth, projective and geometrically connected variety over a number field $k$ such that $\Br(X)/\Br_0(X)$ is finite and the  Brauer-Manin set of $X$ is not empty.  It is well known that $X(k)$ is Zariski dense in $X$ if $X(k)$ is dense in its Brauer-Manin set. Indeed this then follows from weak weak approximation.
Let $S\supset \infty_k$  be a finite subset of $\Omega_k$, $\frak o_S$ the ring of $S$-integers of $k$. Let $U$ be a smooth geometrically connected variety $U$ over $k$,  $\mathcal U$ an integral model over $\frak o_S$.  We denote
$$\mathcal{U}(A_{\frak o_S})= \prod_{v\in S}U(k_v)\times \prod_{v\not \in S} \mathcal{U}(\frak o_{v})$$ where $k_v$ and $\frak o_v$ are the completion of $k$ and $\frak o_S$ with respect to $v\in \Omega_k$ respectively. One has the following integral analogy.

\begin{prop}\label{aiz} Let $U$ be a smooth geometrically connected variety over a number field $k$ such that $\Br(U)/\Br_0(U)$ is finite.
Suppose $\mathcal U$ is an integral model of $U$ over $\frak o_S$ with $\mathcal U(A_{\frak o_S})^\Br \neq \emptyset $. If $\mathcal U(\frak o_S)$ is dense in $pr_{S}(\mathcal U(A_{\frak o_S})^\Br )$  where $pr_S: \mathcal U(A_{\frak o_S})\to \prod_{v \not \in S} \mathcal U(\frak o_v)$ is the natural projection, then $\mathcal U(\frak o_S)$ is Zariski dense in $\mathcal U$.
\end{prop}
\begin{proof} Let $\mathcal N$ be a non-empty Zariski open subset of $\mathcal U$ and fix a finite set $B\subset \Br(U)$ generating $\Br(U)/\Br_0(U)$. There is a sufficiently large finite subset   $S' \supset S$ of $\Omega_k$ such that $\mathcal N(\frak o_{v})\neq \emptyset$, $\mathcal N$ is smooth over $\frak o_{v}$ and each element in $B$ vanishes over  $\mathcal U(\frak o_{v})$ for all $v\not\in S'$.
 
Take $v_0\not\in S'$. Then the open subset
$$\mathcal N(\frak o_{v_0}) \times \prod_{v\not\in (S\cup\{ v_0\}) } \mathcal U(\frak o_{v}) 
\subset pr_S(\mathcal{U}(A_{\frak o_S})^{\Br}) $$
has non-empty intersection with $\mathcal U(\frak o_{S})$
  by the assumption. This implies that
$$ \mathcal U(\frak o_{v_0}) \supset \mathcal U(\frak o_S) \cap \mathcal N(\frak o_{v_0}) \neq \emptyset .$$  Therefore $\mathcal N\cap \mathcal U(\frak o_S) \neq \emptyset$ as desired.
\end{proof}

As we have seen in this section, the converse of Proposition \ref{aiz}  does not hold.

\section{Appendix: the real locus}\label{real}

We here provide details for Remark 5.3. The following lemma should be well known. We provide the proof for convenience of the reader.

\begin{lem} \label{a1} Let $X$ be a topological space with a covering $\{X_i\}$  of connected subsets of $X$. Assume that  
 for any two elements $Y$ and $Z$ in $\{X_i\}$, there are $X_{1}, \cdots, X_{k}$ in $\{X_i\}$  satisfying $$ \overline{Y}\cap \overline{X}_{1} \neq \emptyset, \  \overline{X}_{1} \cap \overline{X}_{2} \neq \emptyset,  \cdots, \overline{X}_{k-1} \cap \overline{X}_{k} \neq \emptyset, \ \overline{X}_{k} \cap \overline {Z} \neq \emptyset $$ 
where $\overline{Y}, \overline{X}_{1}, \cdots, \overline{X}_{k}, \overline{Z}$ are the topological closures of $Y, X_{1}, \cdots, X_{k}, Z$ in $X$ respectively. Then $X$ is connected. 
\end{lem}

\begin{proof} Suppose that $X$ is not connected. Then $X$ contains a non-empty, open and closed subset $D\neq X$. Since $\{X_i\}$ is a covering of $X$, there is $Z$ in $\{X_i\}$ such that $Z\not\subset D$. 

 On the other hand, one has \begin{equation} \label{extrem} D\cap \overline{X_i}=\emptyset   
\ \ \ \text{or} \ \ \  \overline{X_i} \subset D\end{equation} for each element $X_i$ in $\{X_i\}$ by the connectedness of $X_i$. Since $D$ is not empty, there is $Y$ in $\{X_i\}$ such that  $\overline{Y} \subset D$ by (\ref{extrem}). By the assumption, there are $X_{1}, \cdots, X_{k}$ in $\{X_i\}$ satisfying $$ \overline{Y}\cap \overline{X}_{1} \neq \emptyset, \  \overline{X}_{1} \cap \overline{X}_{2} \neq \emptyset,  \cdots, \overline{X}_{k-1} \cap \overline{X}_{k} \neq \emptyset, \ \overline{X}_{k} \cap \overline {Z} \neq \emptyset . $$ 
Therefore $ \overline {X}_{1} \subset D$ by (\ref{extrem}).  Applying (\ref{extrem}) repeatedly, one gets $$ \overline {X}_{2} \subset D,  \cdots ,  \overline {X}_{k} \subset D. $$ Finally, one concludes that $ \overline {Z} \subset D$ by (\ref{extrem}). A contradiction is derived. 
\end{proof}

Recall that  $U_m$ is the affine scheme over $\Bbb R$ defined by the equation
\begin{equation} \label{b-eq} x^2+y^2+z^2-xyz= m.  \end{equation}

\begin{prop} For $m\in \Bbb R$, the number of connected components of $U_m(\Bbb R)$ is given by 
$$\#\pi_{0}(U_m(\mathbb R))=\begin{cases}
1 \ \ \ & \text{ for } m\geq 4\\
5 \ \ \ & \text{ for } 0\leq m <4\\
4 \ \ \ & \text{ for } m <0 . 
\end{cases}$$
More precisely, 

When $m<0$, the connected components of $U_m(\Bbb R)$ are 
$$ \begin{cases} \{(x,y,z) \in U_m(\Bbb R): x\geq 2, \ y\geq 2 \} \\
\{ (x, y, z) \in U_m(\Bbb R): x\leq -2, \ y\geq 2 \} \\
\{ (x, y, z)\in U_m(\Bbb R): x\leq -2, \ y\leq -2 \} \\
\{(x, y, z)\in U_m(\Bbb R): x\geq 2, \ y\leq -2 \}. 
\end{cases} $$
They are unbounded and transitively permuted by
$\Gamma$.

When $0\leq m <4$, the connected components of $U_m(\Bbb R)$ are
$$ \begin{cases} \{(x,y,z) \in U_m(\Bbb R): x\geq 2, \ y\geq 2 \} \\
\{ (x, y, z) \in U_m(\Bbb R): x\leq -2, \ y\geq 2 \} \\
\{ (x, y, z)\in U_m(\Bbb R): x\leq -2, \ y\leq -2 \} \\
\{(x, y, z)\in U_m(\Bbb R): x\geq 2, \ y\leq -2 \} \\ 
\{ (x, y, z)\in U_m(\Bbb R): -2\leq x\leq 2, \ -2 \leq y \leq 2 \}. 
\end{cases} $$
The first four components are unbounded and  $\Gamma$ permutes them transitively.
The last component  is bounded and reduced to the  point  $(0,0,0)$ if $m=0$.

When $4\leq m$, then $U_m(\Bbb R)$ is connected and unbounded.
\end{prop}

\begin{proof}  Since (\ref{b-eq}) is equivalent to 
$$ (2z-xy)^2= (x^2-4)(y^2-4)+ 4(m-4) ,$$ one concludes that the following closed subsets of $U_m(\Bbb R)$ 
$$ \begin{cases} D_1=\{(x,y,z) \in U_m(\Bbb R): x\geq 2, \ y\geq 2 \} \\
D_2= \{ (x,y,z) \in U_m(\Bbb R): -2\leq x\leq 2, \ y\geq 2 \} \\
D_3= \{ (x, y, z) \in U_m(\Bbb R): x\leq -2, \ y\geq 2 \} \\
D_4= \{ (x, y, z) \in U_m(\Bbb R): x\leq -2, \ -2\leq y\leq 2 \} \\
D_5= \{ (x, y, z)\in U_m(\Bbb R): x\leq -2, \ y\leq -2 \} \\
D_6=\{ (x,y,z)\in U_m(\Bbb R): -2\leq x\leq 2, \ y\leq -2\} \\
D_7=\{(x, y, z)\in U_m(\Bbb R): x\geq 2, \ y\leq -2 \} \\
D_8=\{ (x, y, z) \in U_m(\Bbb R): x\geq 2, \ -2\leq y \leq 2 \}\\
D_9= \{ (x, y, z)\in U_m(\Bbb R): -2\leq x\leq 2, \ -2 \leq y \leq 2 \} 
\end{cases} $$
are connected with $U_m(\Bbb R)=\bigcup_{i=1}^9 D_i$. 

When $m\geq 4$, then $D_9 \cap D_i \neq \emptyset$ for $1\leq i\leq 8$.  Therefore $U_m(\Bbb R)$ is connected by Lemma \ref{a1}.

When $m<4$, then $D_2=D_4=D_6=D_8=\emptyset$. Moreover $D_9=\emptyset$ if and only if $m<0$. In this case, one obtains that $D_1, D_3, D_5, D_7$ are the connected components of $U_m(\Bbb R)$, which are unbounded. Using $(x,y,z) \mapsto (-x,-y,z)$ and $(x,y,z) \mapsto (-x,y,-z)$
one sees that $\Gamma$ transitively permutes these 4 components. 
For $0\leq m<4$, one has $D_9\cap D_i=\emptyset$ for $i=1,3,5,7$. Therefore $D_9$ is a bounded connected component of $U_m(\Bbb R)$. 
\end{proof}

{\it Acknowledgements.} Dasheng Wei and Fei Xu are supported by National Natural Science Foundation of China, Grant No. 11631009.
Dasheng Wei is also supported by National Natural Science Foundation of China, Grant No.  11622111. The referee's comments helped us
improve the presentation of the paper.

\bibliographystyle{alpha}

\begin{thebibliography}{99}

\bibitem{A}  Tom M. Apostol,  \emph{Introduction to Analytic Number Theory},  Springer-Verlag, 1976.


\bibitem{BO} S. Bloch and A. Ogus,
 Gersten's conjecture and the homology of schemes, Ann. Sc. E.N.S. {\bf 7} (1974) 181--202.
 
\bibitem{BGS}  J. Bourgain, A. Gamburd, P. Sarnak,  
Markoff triples and strong approximation.  
C. R. Math. Acad. Sci. Paris {\bf 354} (2016), no. 2, 131--135. 

\bibitem{CT} J.-L. Colliot-Th\'el\`ene,
Birational invariants, purity and the Gersten conjecture, in  \emph{K-Theory and Algebraic Geometry: Connections with Quadratic Forms and Division Algebras}, AMS Summer Research Institute, Santa Barbara 1992, ed. W. Jacob and A. Rosenberg, Proceedings of Symposia in Pure Mathematics {\bf 58}, Part I (1995) 1--64.

\bibitem{CTSkSD}  J.-L. Colliot-Th\'el\`ene, A.N. Skorobogatov and Sir Peter Swinnerton-Dyer, Double fibres and double covers: paucity of rational points,
Acta Arithmetica LXXIX.2 (1997) 113--135.

\bibitem{CTXu09} J.-L. Colliot-Th\'el\`ene and F. Xu,  {B}rauer-{M}anin obstruction for integral points of homogeneous spaces and representation of integral quadratic forms,  Compositio Mathematica {\bf 145} (2009), 309--363.



\bibitem{CTW} J.-L. Colliot-Th\'el\`ene et O. Wittenberg,
Groupe de Brauer et points entiers de deux familles de surfaces cubiques affines, American Journal of Mathematics {\bf 134} (2012), no. 5, 1303--1327.

\bibitem{Gr} A. Grothendieck,  Le groupe de Brauer III,  in
 \emph{Dix expos\'es sur la cohomologie des sch\'emas},
Masson, North-Holland, 1968.

\bibitem{GS06} P. Gille and T. Szamuely, \emph{Central simple algebras and Galois cohomology}, Second Edition, Cambridge studies in advanced mathematics {\bf 165}, Cambridge University Press, 2017.



 \bibitem{GS} A. Ghosh and P. Sarnak, Integral points on Markoff type cubic surfaces,
 arXiv: 1706.06712v2



\bibitem{H} Y. Harpaz,  Geometry and arithmetic of certain log $K3$ surfaces, 
 Ann. Inst. Fourier (Grenoble) {\bf 67} (2017), no. 5, 2167--2200.
  
\bibitem{Hartshorne} R. Hartshorne, \emph{Algebraic Geometry}, GTM {\bf 52}, Springer.

 \bibitem{JS} J. Jahnel and D. Schindler,  On integral points of degree four del Pezzo surfaces,
 Israel J.  Math.  {\bf 222} (2017), 21-62.



\bibitem{kato} K. Kato, A Hasse principle for two dimensional global fields, J. f\"{u}r die reine und angew. Math. {\bf 366} (1986) 142--181.


\bibitem{LM} D. Loughran and V. Mitankin, Integral Hasse principle and strong approximation for Markoff surfaces, International Mathematics Research Notices, to appear,
arXiv: 1807.10223v3
 

\bibitem{Milne80} J. S. Milne,  \emph{\'Etale cohomology}, Princeton Mathematical Series {\bf 33}, Princeton University Press, 1980.

\bibitem{Mo} L. J. Mordell, On the integer solutions of the equation $x^2+y^2+z^2+2xyz=n$,  Journal of the London Math. Soc. {\bf 28} 
 (1953) 500--510.

\bibitem{NSW}  J. Neukirch, A.Schmidt and K.Wingberg,  \emph{Cohomology of Number Fields},  Grundlehren der Math. {\bf 323}, Springer,  2000.

\bibitem{SS} P. Salberger and A. Skorobogatov, Weak approximation for surfaces defined by two quadratic forms,
Duke Math. J. {\bf 63} (1991), no. 2, 517--536.

\bibitem{Sansuc} J.-J. Sansuc,  Groupe de Brauer et arithm\'etique des groupes alg\'ebriques lin\'eaires sur un corps de nombres, J. f\"{u}r die reine und angew. Math. {\bf 327} (1981), 12--80.

 

\bibitem{swd}  H.P.F. Swinnerton-Dyer,  The birationality of cubic surfaces over a given field. Michigan Math. J. {\bf 17} (1970) 289--295.

\end{thebibliography}
 
\end{document}